\documentclass[12pt]{article}
\usepackage{amssymb,amsmath,amscd, amsthm,stmaryrd,verbatim, mathrsfs,graphicx,subfigure,caption,lipsum}
\usepackage[english]{babel}
\numberwithin{equation}{section}
\newtheorem{theorem}{Theorem}[section]
\newtheorem{lemma}{Lemma}[section]
\newtheorem{corollary}{Corollary}[section]
\newtheorem{proposition}{Proposition}[section]

\renewcommand{\figurename}{\textbf{Fig.}}

\begin{document}
\baselineskip=14pt

\newcommand{\la}{\langle}
\newcommand{\ra}{\rangle}
\newcommand{\psp}{\vspace{0.4cm}}
\newcommand{\pse}{\vspace{0.2cm}}
\newcommand{\ptl}{\partial}
\newcommand{\dlt}{\delta}
\newcommand{\sgm}{\sigma}
\newcommand{\al}{\alpha}
\newcommand{\be}{\beta}
\newcommand{\G}{\Gamma}
\newcommand{\gm}{\gamma}
\newcommand{\vs}{\varsigma}
\newcommand{\Lmd}{\Lambda}
\newcommand{\lmd}{\lambda}
\newcommand{\td}{\tilde}
\newcommand{\vf}{\varphi}
\newcommand{\yt}{Y^{\nu}}
\newcommand{\wt}{\mbox{wt}\:}
\newcommand{\rd}{\mbox{Res}}
\newcommand{\ad}{\mbox{ad}}
\newcommand{\stl}{\stackrel}
\newcommand{\ol}{\overline}
\newcommand{\ul}{\underline}
\newcommand{\es}{\epsilon}
\newcommand{\dmd}{\diamond}
\newcommand{\clt}{\clubsuit}
\newcommand{\vt}{\vartheta}
\newcommand{\ves}{\varepsilon}
\newcommand{\dg}{\dagger}
\newcommand{\tr}{\mbox{Tr}}
\newcommand{\ga}{{\cal G}({\cal A})}
\newcommand{\hga}{\hat{\cal G}({\cal A})}
\newcommand{\Edo}{\mbox{End}\:}
\newcommand{\for}{\mbox{for}}
\newcommand{\kn}{\mbox{ker}}
\newcommand{\Dlt}{\Delta}
\newcommand{\rad}{\mbox{Rad}}
\newcommand{\rta}{\rightarrow}
\newcommand{\mbb}{\mathbb}
\newcommand{\lra}{\Longrightarrow}
\newcommand{\X}{{\cal X}}
\newcommand{\Y}{{\cal Y}}
\newcommand{\Z}{{\cal Z}}
\newcommand{\U}{{\cal U}}
\newcommand{\V}{{\cal V}}
\newcommand{\W}{{\cal W}}
\newcommand{\sta}{\theta}
\setlength{\unitlength}{3pt}
\newcommand{\msr}{\mathscr}
\newcommand{\wht}{\widehat}
\newcommand{\mfk}{\mathfrak}
\renewcommand{\figurename}{\textbf{Fig.}}

\begin{center}{\large \bf Orthogonal Oscillator Representations, Laplace Equations \\ \pse and Intersections of Determinantal Varieties} 
\end{center}

\vspace{0.2cm}

\begin{center}{\large Hengjia Zhang\footnote{Corresponding author.} and Xiaoping Xu\footnote{Research supported
 by National Key R\&D Program of China 2020YFA0712600.
} }\end{center}

\begin{center}{
HLM, Institute of Mathematics, Academy of Mathematics \& System
Sciences\\ Chinese Academy of Sciences, Beijing 100190, P.R. China
\\ \& School of Mathematics, University of Chinese Academy of Sciences,\\ Beijing 100049, P.R. China}\end{center}

\begin {abstract}
\quad

Associated varieties are geometric objects appearing in infinite-dimensional representations of semisimple Lie algebras (groups). By applying Fourier transformations to the natural orthogonal oscillator representations of special linear Lie algebras, Luo and the second author (2013) obtained a big family of infinite-dimensional irreducible representations of the algebras on certain spaces of homogeneous solutions of the Laplace equation. In this paper, we prove that the associated varieties of these irreducible representations are the intersections of explicitly given determinantal varieties. This provides an explicit connection among representation theory, partial differential equations and algebraic geometry.

 \vspace{0.3cm}

\noindent{\it Keywords}:\hspace{0.3cm} special linear Lie
algebra; orthogonal oscillator
 representation; associated variety; irreducible module; filtrations; determinantal variety.

\end{abstract}

\section {Introduction}
In 1971, Bernstein \cite{Bj} defined a variety associated with a filtrated module of the algebra of differential operators, in order to construct the fundamental solution of a linear partial differential equation with constant coefficients. A primitive ideal of the universal enveloping
algebra of a semisimple Lie algebra is the kernel of an irreducible representation of the Lie algebra. Borho and Jantzen \cite{BJ} (1977) found a translation principle among primitive ideals and used it to prove the Dixmier orbit map is well defined for $sl(n)$. Vogan \cite{Vd1} (1978) found particularly refined techniques
in the description of the leading coefficient in the Hilbert-Samuel polynomial of a Harish-Chandra module for a semisimple Lie algebra and its maximal compact subalgebra, where the filtration of the module is naturally induced from the canonical filtration  of the universal enveloping
algebra of the Lie algebra. Moreover, he \cite{Vd2} (1980) derived a new order relation in primitive ideals and used it to prove a Joseph's conjecture on primitive ideals and multiplicities in the composition series of principal series representations.

Borho and Brylinsky \cite{BB} (1982) proved that the kernel of the representation of the  universal enveloping
algebra of a semisimple Lie algebra by global differential operators on a complete homogeneous space of the corresponding Lie group is the annihilator of a generalized Verma module, and the associated variety (the term firstly appeared) of the ideal is irreducible. Joseph \cite{Ja1} (1983) gave an extensive review on the classification of primitive ideals in  the  universal enveloping
algebra of a semisimple Lie algebra. He \cite{Ja2} (1984) continued his investigation on associated varieties aiming for a way to compute them and to relate it more naturally to the nilpotent orbits and the corresponding Springer representations  of the Weyl group. Moreover, he  \cite{Ja3} (1985) proved that the associated varieties are the Zariski closures of nilpotent orbits. Matumote \cite{Mh} (1987) proved a necessary condition for the existence of certain Whittaker vectors in terms of associated varieties. Vogan \cite{Vd3} (1991)  proved that the associated varieties of $(G,K)$-modules are the Zariski closures of union of $K$-nilpotent orbits. Melnikov \cite{Ma} (1993) conditionally proved that the associated variety of an irreducible highest weight $sl(n)$-module is irreducible and Williamson \cite{GW} (1995) found a counter example. Barbasch and Bozicevic \cite{BB1} (1999) derived a formula for the associated varieties of induced modules, which is analogous to the one for wave front set of a derived functor module obtained by Barbasch and Vogan in 1980.  There are also many other interesting significant works related to associated varieties
 (e.g., cf. \cite{BV1, BV2, Ja3, Yh1, Yh2, Jh, Yh3, Mw, GY, Kt, Mi, Tp, NTW, AV, BXX, BMW, BHXZ}). It seems desirable to know if there exist geometrically sophisticated explicit known algebraic varieties that are associated varieties.

  Determinantal varieties are natural important geometric objects in algebraic geometry. Barshay \cite{Bj2} (1973) used them to prove the coordinate ring of the Veronese is arithmetically normal and Cohen-Macaulay. Catanese \cite{Cf} (1981) used symmetric determinantal varieties to prove Babbage's conjecture. de Concini, Eisenbud and Procesi \cite{DEP} (1980) found an important connection of Young diagrams with determinantal varieties. Griffiths \cite{Gp} (1983) found that determinantal varieties play important roles in the infinitesimal variations of Hodge structure. Enright and Hunziker \cite{EH} (2004) gave explicit formulas for the numerator polynomials of  the Hilbert series associated with the closures of the nilpotent orbits corresponding to Wallach representations through determinantal varieties. Moreover,  Enright, Hunziker and Pruett
  \cite{EHP} (2014) presented  a natural generalization of Young diagrams for Hermitian symmetric spaces and used it to study syzygies of determinantal varieties.  Bernardara, Bolognesi and Faenzi \cite{BBF} (2016) derived a homological projective duality for determinantal varieties.
  Wong \cite{Wk} (2017) gave an explicit calculation of the Euler characteristic of the spaces of classical vacua in terms of the Euler characteristics of suitably chosen determinantal subvarieties. Homma and  Manabe \cite{HM} (2018) showed that determinantal Calabi-Yau varieties in Grassmannians are closely related to Givental I-functions.

 By applying Fourier transformations to the natural orthogonal oscillator representations of special linear Lie algebras, Luo and the second author \cite{LX} (2013) obtained a big family of infinite-dimensional irreducible representations of the algebras on certain spaces of homogeneous solutions of the Laplace equation. Bai \cite{Bz} (2015) found the Gel'fand-Kirillov dimensions of these representations. In this paper, we prove that the associated varieties of these irreducible representations are intersections of  explicitly given determinantal varieties. This provides an explicit connection among  representation theory, partial differential equations and algebraic geometry. It also partially meets Bernstein's original concern in \cite{Bj}. Below we give a more detailed technical introduction to our work.

Let $\mbb F$ be a field with characteristic 0, and let $n>1$ be an integer. Denote by $\msr A=\mbb F[x_1,x_2,...,x_n,y_1,y_2,...,y_n]$ the polynomial algebra in $2n$ variables. Let $E_{r,s}$ be the square matrix with 1 as its $(r,s)$-entry and 0
as the others. The special linear Lie algebra
\begin{equation}
sl(n)=\sum_{1\leq i<j\leq
n}(\mbb{F}E_{i,j}+\mbb{F}E_{j,i})+\sum_{r=1}^{n-1}\mbb{F}(E_{r,r}-E_{r+1,r+1})\end{equation}
with the commutator as its Lie bracket. For any
two integers $p\leq q$, we denote $\ol{p,q}=\{p,p+1,\cdots,q\}$. The natural orthogonal oscillator representation $\pi$ of $sl(n)$ on $\msr A$ is given as follows:
\begin{equation}\label{a1.2}
\pi(E_{i,j})=x_i\ptl_{x_j}-y_j\ptl_{y_i}\qquad\for\;\;i,j\in\ol{1,n}.\end{equation}

With respect to the above representation, the Laplace operator
\begin{equation}\Dlt=\sum_{i=1}^n\ptl_{x_i}\ptl_{y_i}.\end{equation}
Denote by $\mbb N$ the set of nonnegative integers. For $\ell_1,\ell_2\in \mbb N$, we denote by $\msr A_{\ell_1,\ell_2}$ the subspace of homogeneous polynomials with the degree $\ell_1$ in $\{x_1,x_2,...,x_n\}$ and the degree $\ell_2$ in $\{y_1,y_2,...,y_n\}$. Then $\msr A=\bigoplus_{\ell_1,\ell_2\in\mbb N}\msr A_{\ell_1,\ell_2}$ is an $\mbb N^2$-graded algebra.

Set the subspace of homogeneous harmonic polynomials as follows:
\begin{equation}{\msr H}_{\ell_1,\ell_2}=\{f\in {\msr B}_{\ell_1,\ell_2}\mid
\Dlt(f)=0\}.\label{a1.4}\end{equation}
Then ${\msr H}_{\ell_1,\ell_2}$ forms an irreducible $sl(n)$-module with highest weight $\ell_1\lmd_1+\ell_2\lmd_{n-1}$ (e.g., cf. \cite{Xx}), where $\lmd_i$ denotes the $i$th fundamental weight of $sl(n)$.

Denote by $\msr F_z$ the Fourier transformation on $z$. Fix $n_1,n_2\in\ol{1,n}$ with $n_1\leq n_2$. Letting (\ref{a1.2}) act on the space
\begin{equation}{\msr A}'=\msr F_{x_1}\cdots \msr F_{x_{n_1}}\msr F_{x_{n_2+1}}\cdots \msr F_{x_n}(\msr A)\end{equation}
of distributions, one can obtain a big family of infinite-dimensional irreducible representations of $sl(n)$ on certain spaces of homogeneous solutions of the Laplace equation. However, this picture is not convenient to be dealt with. Instead, we consider another equivalent representation as follows.
In the representation formulas given in (\ref{a1.2}), applying the Fourier transformations on operators:
\begin{equation}\label{a1.6}
\ptl_{x_r}\mapsto -x_r,\;
 x_r\mapsto\ptl_{x_r},\;\ptl_{y_s}\mapsto -y_s,\;
 y_s\mapsto\ptl_{y_s}\qquad \for\; r\in\ol{1,n_1},\;s\in\ol{n_2+1,n},\end{equation}
we obtain a new representation $\td{\pi}$ on $\msr A$. The representations $(\pi,\msr A')$ and $(\td{\pi},\msr A)$ are equivalent (or isomorphic).
Under the changes in (\ref{a1.6}), the Laplace operator becomes
\begin{equation}\td\Dlt=\sum_{i=1}^{n_1}x_i\ptl_{y_i}-\sum_{r=n_1+1}^{n_2}\ptl_{x_r}\ptl_{y_r}+\sum_{s=n_2+1}^n
y_s\ptl_{x_s}.\label{a1.7}\end{equation}

Take the degree of $\{x_1,...,x_{n_1},y_{n_2+1},...,y_n\}$ as $-1$ and that of $\{x_{n_1+1},...,x_n,y_1...,y_{n_2}\}$
as 1. Denote by $\mbb Z$ the ring of integers. For $\ell_1,\ell_2\in \mbb Z$, we denote by $\msr A_{\la\ell_1,\ell_2\ra}$ the subspace of homogeneous polynomials with the degree $\ell_1$ in $\{x_1,x_2,...,x_n\}$ and the degree $\ell_2$ in $\{y_1,y_2,...,y_n\}$.
 Then $\msr A=\bigoplus_{\ell_1,\ell_2\in\mbb Z}\msr A_{\la\ell_1,\ell_2\ra}$ is a $\mbb Z^2$-graded algebra. Set
\begin{equation}{\msr H}_{\la\ell_1,\ell_2\ra}=\{f\in {\msr A}_{\la\ell_1,\ell_2\ra}\mid
\td{\Dlt}(f)=0\}.\end{equation}The following is a result from \cite{LX}:\psp

{\bf Proposition 1}\quad {\it  For
$\ell_1,\ell_2\in\mbb{Z}$ , the necessary and sufficient condition for ${\msr H}_{\la\ell_1,\ell_2\ra}$ to be an infinite-dimensional irreducible $sl(n)$-module is as follows:

(1)  When $n_1+1<n_2$, (a) $\ell_1+\ell_2\leq n_1-n_2+1$; (b) $n_2=n$, $\ell_1\in \mbb N$ and $\ell_2=0$; (c) $n_2=n$, $\ell_2\in \mbb N  $ and $\ell_1 \geq n_1-n+2$;

(2)  If $n_1+1=n_2$,  $\ell_1+\ell_2\leq 0$ or $n_2=n$ and
$0\leq\ell_2\leq\ell_1$;

(3)  In the case $n_1=n_2$, $\ell_1+\ell_2\leq 0$ and:(a)
 $\ell_2\leq 0 $, $n_1<n-1$ and $n\geq 3$;
 (b) $\ell_1\leq 0$, $1<n_1<n $ and $n\geq 3$; (c) $\ell_1,\ell_2\leq
 0$, $n_1=1$ and $n=2$. } \psp

 The above irreducible modules are infinite-dimensional highest weight modules with distinct highest weights. For instance,  if $n_1+1<n_2<n$
and $m_1,m_2\in\mbb{N}$ with $m_1+m_2\geq n_2-n_1-1$, ${\msr
H}_{\la-m_1,-m_2\ra}$ has  a highest-weight vector
$x_{n_1}^{m_1}y_{n_2+1}^{m_2}$ of weight
$m_1\lmd_{n_1-1}-(m_1+1)\lmd_{n_1}-(m_2+1)\lmd_{n_2}+m_2(1-\dlt_{n_2,n-1})\lmd_{n_2+1}$.

Let $\mfk g$ be a finite-dimensional semisimple Lie algebra and let $U(\mfk g)$ be the universal enveloping algebra of $\mfk g$. Denote
$U_0(\mfk g)=\mbb F$ and for $0<m\in\mbb Z$, $U_m(\mfk g)=\mbb F+\sum_{r=1}^m\mfk g^r$. Then $\{U_k(\mfk g)\mid k\in\mbb N\}$ forms a filtration of $U(\mfk g)$. Suppose that $M$ is a finitely generated $U(\mfk g)$-module and $M_0$ is a finite-dimensional subspace of $M$ satisfying
\begin{equation}\label{a1.9}
(U(\mfk g))(M_0)=M.\end{equation}
Setting $M_k =(U_k( \mathfrak{g}))(M_0)$, we get a filtration $\{M_k\mid k\in\mbb N\}$ of $M$.

The associated graded module is given by
\begin{equation}\label{a1.10}
\ol{M}=\mbox{gr}( M; M_0) = \bigoplus_{k = 0}^\infty \ol{M_k}\quad \text{with}\quad \ol{M}_k=M_k / M_{k-1},
\end{equation}
and it is a finitely generated graded $S( \mathfrak{g})$-module, since graded algebra $\mbox{gr}(  U( \mathfrak{g}))$ is canonically isomorphic to the symmetric tensor algebra $S( \mathfrak{g})$ by the PBW theorem. Here we identify $S( \mathfrak{g})$ with $\mathbb{F}[\mathfrak{g}^*]$, the polynomial ring over $\mathfrak{g}^*$, in the canonical way through the Killing form of $  \mathfrak{g} $.

The annihilator of $\ol{M}$ in $S( \mathfrak{g})$ is defined as
\begin{equation}\label{a1.11}
\operatorname{Ann}_{S( \mfk g)} (\ol{M})=\{ \xi \in  S( \mathfrak{g})\mid \xi(v)=0,\; \forall\; v \in \ol{M} \},\end{equation}
which is a graded ideal of $S( \mathfrak{g})$. It defines the associated variety in the dual space $\mathfrak{g}^*  $:
\begin{equation}
{\msr V}(M)=\{ f \in \mathfrak{g}^*\mid u(f) =0,\;\forall\; u\in \operatorname{Ann}_{S(\mathfrak{g})}(\ol{M})\}.\end{equation}
The variety does not depend on the choice of the subspace $M_0$ satisfying (\ref{a1.9}).

Let $S_1$ and $S_2$ be subsets of $\ol{1,n}$.
Denote the determinantal ideal $I_t (S_1,S_2)$ to be the ideal generated by all the $t$-minors in the matrix $(z_{i,j})_{i \in S_1,j\in S_2}$, and $\msr{V}_t(S_1,S_2) =\msr{V}(I_t)$ is called  a {\it determinantal variety}.
For convenience, we denote $I_t(S_1,S_2)=0$ and $\msr{V}_t(S_1 , S_2)=\mbb A^{|S_1|  |S_2|}$, when $t>|S_1|$ or $t>|S_2|$.

Recall $n_1,n_2\in\ol{1,n}$ with $n_1\leq n_2$. Set
\begin{equation}\label{aa1.13}
J_1=\ol{1,n_1},\quad J_2=\ol{n_1+1,n_2},\quad J_3=\ol{n_2+1,n}.\end{equation}
The following is our main theorem in this paper: \psp

{\bf Main Theorem}\quad {\it If $n_1\leq n_2$, $\ell_1 \leq 0$ or $\ell_2 \leq 0$,
\begin{equation}\msr V({\msr H}_{\la\ell_1,\ell_2\ra})\cong \msr{V} _3(J_2\cup J_3, J_1\cup J_2)\cap \msr{V}_2(J_2 , J_1)\cap \msr{V}_2(J_3 , J_2) \cap\msr{V}_1(J_2 , J_2).\end{equation} 
If $n_1<n_2=n$ and $\ell_1,\ell_2>0$,
\begin{equation}
\msr V( {\msr H}_{\la\ell_1,\ell_2\ra})\cong \msr{V}_2(J_2 , J_1).\end{equation}}\pse

 The theorem covers all the irreducible $sl(n)$-modules ${\msr H}_{\la\ell_1,\ell_2\ra}$ listed in Proposition 1.
In particular, when $n_1=n_2$ and $|J_1|,|J_3| \geq 3$,
\begin{equation}\msr V({\msr H}_{\la\ell_1,\ell_2\ra}) \cong
\msr{V}_3 (J_3 , J_1)\end{equation}
 is exactly a determinantal variety. When $n_2=n$, the projective version of the associated variety
\begin{equation}\msr V({\msr H}_{\la\ell_1,\ell_2\ra}) \cong
\msr{V}_2 (J_2 , J_1)\end{equation}
is the well-known Serge variety in algebraic geometry.

The paper is organized as follows. In Section 2, we explicitly describe the filtration of ${\msr H}_{\la\ell_1,\ell_2\ra}$ with respect to the representation $\td{\pi}$ and a properly selected finite-dimensional subspace $M_0$ by an adjusted degree $\mfk d$ (cf. (\ref{a2.46})-(\ref{a2.48})). In Section 3, we study certain linear independence on certain subsets in the spanning subsets of filtration subspaces by introducing the notion of ``3-chain" , which plays a key role in next section. Finally in Section 4, we present a proof of the main theorem by using certain ``order" on the products of quadratic alternating polynomials.

\section{Explicit Description of the Filtration for ${\msr H}_{\la\ell_1,\ell_2\ra}$}
In this section, we select a special subspace $M_0$ of $M={\msr H}_{\la\ell_1,\ell_2\ra}$ and determine a spanning subset  for each $M_k=(U_k( \mathfrak{g}))(M_0)$.  Moreover, a new degree function ${\mfk d}$ on $M$  is introduced in order to determine the structure of $M_k$.

In the rest of this paper, we only consider the representation $\td\pi$ of $sl(n)$ obtained under the Fourier transformations (\ref{a1.6}), whose representation formulas are given by
\begin{equation}\label{a2.1}
\td{\pi}(E_{i,j})=E_{i,j}^x-E_{j,i}^y\qquad\for\;\;i,j\in\ol{1,n}\end{equation} with
\begin{equation}E_{i,j}^x=\left\{\begin{array}{ll}-x_j\ptl_{x_i}-\delta_{i,j}&\mbox{if}\;
i,j\in\ol{1,n_1};\\ \ptl_{x_i}\ptl_{x_j}&\mbox{if}\;i\in\ol{1,n_1},\;j\in\ol{n_1+1,n};\\
-x_ix_j &\mbox{if}\;i\in\ol{n_1+1,n},\;j\in\ol{1,n_1};\\
x_i\partial_{x_j}&\mbox{if}\;i,j\in\ol{n_1+1,n}
\end{array}\right.\end{equation}
and
\begin{equation}\label{a2.3}
E_{i,j}^y=\left\{\begin{array}{ll}y_i\ptl_{y_j}&\mbox{if}\;
i,j\in\ol{1,n_2};\\ -y_iy_j&\mbox{if}\;i\in\ol{1,n_2},\;j\in\ol{n_2+1,n};\\
\ptl_{y_i}\ptl_{y_j} &\mbox{if}\;i\in\ol{n_2+1,n},\;j\in\ol{1,n_2};\\
-y_j\partial_{y_i}-\delta_{i,j}&\mbox{if}\;i,j\in\ol{n_2+1,n}.
\end{array}\right.\end{equation}
Moreover, we adopt the notion
\begin{equation}\xi(f)=\td{\pi}(\xi)(f)\qquad\for\;\;\xi\in sl(n),\;f\in\msr A.\end{equation}

Denote
\begin{equation}x^\al=x_1^{\al_1}x_2^{\al_2}\cdots x_n^{\al_n},\;\;y^\al=y_1^{\al_1}y_2^{\al_2}\cdots y_n^{\al_n}\qquad \for\;\;\al=(\al_1,\al_2,...,\al_n)\in \mbb N^n.\end{equation}
In the rest of this section, we always assume $n_1<n_2$.
Define the linear operator $T$ on $\msr A$ by
\begin{align}\label{a2.6}
T(x^\al y^\be)=\sum_{i=0}^{\infty} \frac{\left(x_{n_1+1} y_{n_1+1}\right)^i
	 \left(\tilde{\Delta}+\partial_{x_{n_1+1}}\partial_{ y_{n_1+1}}\right)^i }
 {\prod_{r=1}^i \left(\alpha_{n_1+1} +r\right) \left(\beta_{n_1+1}+ r\right)}(x^\al y^\be)
\end{align}
 for $\al,\be\in \mbb N^n$.
 Then we have
\begin{equation}\label{a2.7}
T \td{\pi}(E_{i,j})= \td{\pi}(E_{i,j}) T\qquad\for\;\;n_1+1\neq i,j\in \ol{1,n}.\end{equation}
 Moreover, the set
\begin{align}\label{a2.8}
\left\{ T\left(x^\alpha y^\beta \right) \mid \alpha,\beta  \in \mathbb{N}^n,\alpha_{n_1+1}\beta_{n_1+1}=0,x^\alpha y^\beta \in {\msr A}_{\la\ell_1,\ell_2\ra}\right\}
\end{align}forms a basis of the module ${\msr H}_{\la\ell_1,\ell_2\ra}$, as shown in \cite{LX}. Furthermore, we introduce the notion
\begin{eqnarray}\label{a2.9}
&&
N\begin{pmatrix}
k_{12} & k_{12} & k_{13} \\
k_{21} & k_{22} &  k_{23}
\end{pmatrix} \nonumber\\ &=&
\left\{\left.   x^\alpha y^\beta \in \msr A_{\la\ell_1,\ell_2\ra}    \right\vert
  \sum_{r=1}^{n_1} \alpha_r =k_{11},\sum_{r=n_1+1}^{n_2} \alpha_r =k_{12},\sum_{r=n_2+1}^{n} \alpha_r =k_{13},
\right.  \nonumber\\  &&  \left.
 \sum_{r=1}^{n_1} \beta_r =k_{21},\sum_{r=n_1+1}^{n_2} \beta_r =k_{22},\sum_{r=n_1+1}^{n_2} \beta_r =k_{22} ,\alpha_{n_1+1}\beta_{n_1+1}=0
\right\}
\end{eqnarray} and
\begin{align}\label{a2.10}
TN\begin{pmatrix}
k_{12} & k_{12} & k_{13} \\
k_{21} & k_{22} &  k_{23}
\end{pmatrix}=\text{Span} \left\{T(f)\mid f\in N\begin{pmatrix}
k_{12} & k_{12} & k_{13} \\
k_{21} & k_{22} &  k_{23}
\end{pmatrix}\right\}\end{align}
where $k_{11},k_{12},k_{13},k_{21},k_{22},k_{23} \in \mathbb{N}$. Now we take
\begin{align}\label{a2.11}
M_0=
\begin{cases}
   TN\begin{pmatrix}
-\ell_1 & 0 & 0 \\
0 & 0 & -\ell_2
\end{pmatrix} & \text{if } \ell_1\leq 0,\ell_2\leq 0, \pse\\
TN\begin{pmatrix}
-\ell_1 & 0 & 0 \\
0 & \ell_2 & 0
\end{pmatrix} & \text{if } \ell_1\leq 0,\ell_2\geq 0, \pse\\
TN\begin{pmatrix}
0 & \ell_1 & 0 \\
0 & 0 & -\ell_2
\end{pmatrix} & \text{if } \ell_1\geq 0,\ell_2\leq 0,
\end{cases}
\end{align}
which is a finite-dimensional subspace of $M={\msr H}_{\la\ell_1,\ell_2\ra}$. Set
\begin{equation}M_k=U_k(sl(n))(M_0)\qquad\for\;\;k\in \mbb N.\end{equation}

For convenience, we denote
\begin{equation}X_S=\{x_i\mid i\in S\},\quad Y_S=\{y_i\mid i\in S\}\qquad\for\;\;S\subset\ol{1,n}.\end{equation}
In order to express $M_k$ in terms of the basis given in (\ref{a2.8}), we need the following conclusions.

\begin{lemma}
\label{alem:1}
Fix $k,k_{13},k_{21}\in\mbb N$. For each $j$, $i_{1,j} \in J_1$, $i_{3,j} \in J_3$, $v_0\in \mathbb{F}\left[X_{J_1},X_{J_2 \setminus \{n_1+1\}},Y_{J_3} \right]\\ \cup  \mathbb{F}\left[X_{J_1},Y_{J_2 \setminus \{n_1+1\}},Y_{J_3} \right]$, we have
\begin{eqnarray}\label{a2.14}
& &
\left( -1\right)^{k} \frac{k! }{\left( k- k_{13}\right) !} \  T\left( v_0 \cdot
\left( \prod_{j=1}^{k}  x_{i_{1,j}}\right)
\left( \prod_{j=1}^{k_{13}}  x_{i_{3,j}}\right) \cdot  x_{n_1+1}^{k-k_{13}}
 \right)\nonumber\\& =&
 E_{i_{3,k_{13}},n_1+1} \cdots  E_{i_{3,1},n_1+1}  E_{n_1+1,i_{1,1}}  \cdots E_{n_1+1,i_{1,k}}
 \left( v_0 \right)
\label{aa31}
\end{eqnarray}
when $k \geq k_{13}$,
and
\begin{eqnarray}\label{a2.15}
& &\left( -1\right)^{k} \frac{k! }{\left( k- k_{21}\right) !} \  T\left(v_0 \cdot
\left( \prod_{j=1}^{k_{21}}  y_{i_{1,j}}\right)
\left( \prod_{j=1}^{k }  y_{i_{3,j}}\right) \cdot  y_{n_1+1}^{k-k_{21}}
 \right) \nonumber\\&=&
 E_{n_1+1,i_{1,1}}  \cdots E_{n_1+1,i_{1,k_{21}}}  E_{i_{3,1},n_1+1} \cdots  E_{i_{3,k},n_1+1}
 \left( v_0 \right)
\label{aa32}
\end{eqnarray}
if $k \geq k_{21}$.
\end{lemma}

\begin{proof}
For arbitrary $k \geq 0$, we prove (\ref{aa31}) by induction on $k_{13}$. According to (\ref{a2.1})-(\ref{a2.3}),
\begin{equation}\label{a2.16}
\td{\pi}(E_{i_2,i_1})= - x_{i_1} x_{i_2}- y_{i_1}\partial_{y_{i_2}},\;\;\td{\pi}(E_{i_3,i_2})=x_{i_3} \partial_{x_{i_2}} +y_{i_3}y_{i_2}\end{equation}
for $i_1\in J_1$, $i_2\in J_2$ and $i_3\in J_3$.

When $k_{13}=0$, the first equation of the above expression gives
\begin{align}\label{a2.17}
 E_{n_1+1,i_{1,1}}  \cdots E_{n_1+1,i_{1,k}}
 \left( v_0 \right) &=\left( -1 \right)^k   \left(v_0 \cdot
\left( \prod_{j=1}^{k}  x_{i_{1,j}}\right)  \cdot  x_{n_1+1}^{k}
 \right);
\end{align}
that is, (\ref{a2.14}) holds. Indeed, the numerator $k!$ plays this role of normalization.

Assume that when $k_{13}=k^\prime -1 \leq k-1$,
\begin{align}\begin{split}
\left( -1\right)^{k} \frac{  k  ! }{\left( k- k^\prime +1\right) !}
\  T\left(v_0 \cdot
\left( \prod_{j=1}^{k}  x_{i_{1,j}}\right)
\left( \prod_{j=1}^{k^\prime -1}  x_{i_{3,j}}\right) \cdot  x_{n_1+1}^{k- k^\prime +1}
 \right) \\
= E_{i_{3,k^\prime -1},n_1+1} \cdots  E_{i_{3,1},n_1+1}  E_{n_1+1,i_{1,1}}  \cdots E_{n_1+1,i_{1,k}}
 \left( v_0 \right) .
\end{split}\end{align}
When $k_{13}=k^\prime$,
we set
\begin{align}\label{a2.19}
a_0=\left( -1\right)^{k} \frac{   k  ! }{\left( k- k^\prime +1\right) !}
\  \left(v_0 \cdot
\left( \prod_{j=1}^{k}  x_{i_{1,j}}\right)
\left( \prod_{j=1}^{k^\prime -1}  x_{i_{3,j}}\right) \cdot  x_{n_1+1}^{k- k^\prime +1}
 \right),
\end{align}
and
\begin{eqnarray}
a_l&=&\left( -1\right)^{k} \frac{ k! }{\left( k- k^\prime +1\right) !}
\  \big(v_0 \cdot
\left( \prod_{j=1}^{k}  x_{i_{1,j}}\right)  \cdot
 x_{n_1+1}^{k- k^\prime +1} \nonumber\\
& &\times
\frac{ \left( x_{n_1+1} y_{n_1+1}\right)^l \left( k-k^\prime +1\right)! }
{\left( k-k^\prime +l+1  \right)!\
 l!   }
\tilde{\Delta}^l \left( f_0 \right)
 \big)\nonumber\\
&=&\left( -1\right)^{k} \frac{k ! }{\left( k-k^\prime +l+1  \right)!  \  l!  }
\  \big(v_0
\left( \prod_{j=1}^{k}  x_{i_{1,j}}\right)
 x_{n_1+1}^{k- k^\prime +l+1}
 y_{n_1+1}^l \cdot
\tilde{\Delta}^l \left( f_0 \right)
 \big),
\end{eqnarray}
where
\begin{align}\label{a2.21}
f_0=\prod_{j=1}^{k^\prime -1}  x_{i_{3,j}}.
\end{align}
According to (\ref{a1.7}),
\begin{equation}\label{a2.22}
\td{\Dlt}x_{i_{1,j}}=x_{i_{1,j}}\td{\Dlt},\;\;\td{\Dlt}y_{i_{3,l}}=y_{i_{3,l}}\td{\Dlt}\end{equation}
and
\begin{equation}[\td{\Dlt},v_0]=-\sum_{n_1+1\neq j\in J_2}(\ptl_{x_j}(v_0)\ptl_{y_j}+\ptl_{y_j}(v_0)\ptl_{x_j}).\end{equation}
Since
\begin{equation}\ptl_{x_{n_1+1}}\ptl_{y_{n_1+1}}( x_{n_1+1}^{k- k^\prime +1})=0,\end{equation}
we have
\begin{equation}T(a_0)=\sum_{r=0}^{k'-1}a_r\end{equation}by (\ref{a2.6}).
Thus the second equation in (\ref{a2.16}) yields
\begin{eqnarray}\label{a2.26}
&& E_{i_{3,k^\prime},n_1+1} \cdots  E_{i_{3,1},n_1+1}  E_{n_1+1,i_{1,1}}  \cdots E_{n_1+1,i_{1,k}}
 \left( v_0 \right) \nonumber\\
&=& E_{i_{3,k^\prime},n_1+1} \left(
\left( -1\right)^{k} \frac{   k  ! }{\left( k- k^\prime +1\right) !}
\  T\left(v_0 \cdot
\left( \prod_{j=1}^{k}  x_{i_{1,j}}\right)
\left( \prod_{j=1}^{k^\prime -1}  x_{i_{3,j}}\right) \cdot  x_{n_1+1}^{k- k^\prime +1}
 \right)
\right)\nonumber\\
&=& E_{i_{3,k^\prime},n_1+1}(T(a_0))\nonumber\\
&=&\left( x_{i_{3,k^\prime }} \partial_{x_{n_1+1}} \right)\left( \sum_{l=0}^{k^\prime} a_l\right)+ y_{i_{3,k^\prime }} y_{n_1+1}\left( \sum_{l=0}^{k^\prime} a_l\right)\nonumber\\
&=&x_{i_{3,k^\prime }} \partial_{x_{n_1+1}} \left(a_0\right) + \sum_{l=1}^{k^\prime} \left(  y_{i_{3,k^\prime }} y_{n_1+1} a_{l-1}+ x_{i_{3,k^\prime }} \partial_{x_{n_1+1}} \left(  a_l\right) \right).
\end{eqnarray}
We set
\begin{equation}\label{a2.27}
b_0=x_{i_{3,k^\prime }} \partial_{x_{n_1+1}} \left(a_0\right)=\left( -1\right)^{k} \frac{   k  ! }{\left( k- k^\prime\right) !}
\  \left(v_0 \cdot
\left( \prod_{j=1}^{k}  x_{i_{1,j}}\right)
\left( \prod_{j=1}^{k^\prime}  x_{i_{3,j}}\right) \cdot  x_{n_1+1}^{k- k^\prime}
 \right)\end{equation}
by (\ref{a2.19}) and
\begin{equation}
 b_l =  y_{i_{3,k^\prime }} y_{n_1+1} a_{l-1}+ x_{i_{3,k^\prime }} \partial_{x_{n_1+1}} \left(  a_l\right)\qquad\for\;\;l \geq 1.\end{equation}
We will show that (\ref{a2.26}) is equivalent to $T \left( b_0 \right)$.

According to (\ref{a1.7}),
\begin{equation}\label{a2.29}
[\td{\Dlt}, x_{i_3}]= y_{i_3}\qquad\for\;\;i_3\in J_3.\end{equation}
Moreover, (\ref{a2.29}) and the second equation in (\ref{a2.22}) show
\begin{equation}[\td{\Dlt}^{s+1}, x_{i_3}]=(s+1)y_{i_3}\td{\Dlt}^s\qquad\for\;\;s\in\mbb N.\end{equation}
Thus for each $l\in\mbb N$, we have
\begin{align}
\tilde{\Delta}^l \left( f_0 \cdot x_{i_3}\right)
=x_{i_3} \tilde{\Delta}^l \left( f_0 \right) + l  \tilde{\Delta}^{l-1} \left( f_0 \right) y_{i_3}.
\end{align}
Therefore,
\begin{eqnarray}
b_l &=&  y_{i_{3,k^\prime }} y_{n_1+1} a_{l-1}+ x_{i_{3,k^\prime }} \partial_{x_{n_1+1}} (  a_l)\nonumber\\
&= &
( -1)^{k} \frac{ k ! }{( k-k^\prime +l)!  \  (l-1)!  }
\  \big(y_{i_{3,k^\prime }} y_{n_1+1} \cdot
v_0
\big( \prod_{j=1}^{k}  x_{i_{1,j}}\big)
 x_{n_1+1}^{k- k^\prime +l}
 y_{n_1+1}^{l-1} \cdot
\tilde{\Delta}^{l-1} ( f_0)
 \big)+
\nonumber\\
& &x_{i_{3,k^\prime }} \partial_{x_{n_1+1}} \big(
( -1)^{k} \frac{ k ! }{( k-k^\prime +l+1)!  \  l!  }
\  \big(
v_0
\big( \prod_{j=1}^{k}  x_{i_{1,j}}\big)
 x_{n_1+1}^{k- k^\prime +l+1}
 y_{n_1+1}^l \cdot
\tilde{\Delta}^l ( f_0)
 \big)
 \big)\nonumber\\
&=&( -1)^{k} \frac{k ! }{( k-k^\prime +l)!  \  (l-1)!  }
\  \big(
v_0
\big( \prod_{j=1}^{k}  x_{i_{1,j}}\big)
 x_{n_1+1}^{k- k^\prime }
(  x_{n_1+1} y_{n_1+1}) ^{l}
 ) \cdot
\tilde{\Delta}^{l-1} ( f_0 )   y_{i_{3,k^\prime }}\nonumber\\&
&+ ( -1)^{k} \frac{k ! }{( k-k^\prime +l)!  \  l!  }
\big(
v_0
\big( \prod_{j=1}^{k}  x_{i_{1,j}}\big)
 x_{n_1+1}^{k- k^\prime  }
(  x_{n_1+1}y_{n_1+1})^l \cdot
\tilde{\Delta}^l ( f_0)
 \big) x_{i_{3,k^\prime }}\nonumber
\\
&=&  \frac{( -1)^{k} k ! }{( k-k^\prime +l )!  \  l!  }
\  \big(
v_0
\big( \prod_{j=1}^{k}  x_{i_{1,j}}\big)
 x_{n_1+1}^{k- k^\prime }
(  x_{n_1+1} y_{n_1+1}) ^{l}
 \big)
( l  \tilde{\Delta}^{l-1}( f_0)   y_{i_{3,k^\prime }}  +  x_{i_{3,k^\prime }} \tilde{\Delta}^l ( f_0))\nonumber\\
&=&  \frac{( -1)^{k} k ! }{( k-k^\prime +l)!  \  l!  }
\  \big(
v_0
\big( \prod_{j=1}^{k}  x_{i_{1,j}}\big)
 x_{n_1+1}^{k- k^\prime }
(  x_{n_1+1} y_{n_1+1}) ^{l}
 \big) \cdot \tilde{\Delta}^l ( f_0 x_{i_{3,k^\prime }})\nonumber\\
&=&  \frac{( -1)^{k} k ! }{( k-k^\prime +l)!  \  l!  }
\  \big(
v_0
\big( \prod_{j=1}^{k}  x_{i_{1,j}}\big)
 x_{n_1+1}^{k- k^\prime }
(  x_{n_1+1} y_{n_1+1}) ^{l}
 \big) \cdot (\tilde\Delta+\partial_{x_{n_1+1}}\partial_{ y_{n_1+1}})^l ( f_0 x_{i_{3,k^\prime }})\nonumber\\
&=&  \frac{(  x_{n_1+1} y_{n_1+1}) ^{l}}{\big(\prod_{r=1}^l(k-k' +r)\big)  \  l! }
\
 (\tilde\Delta+\partial_{x_{n_1+1}}\partial_{ y_{n_1+1}})^l\left(\frac{( -1)^{k} k !}{(k-k')!} v_0
\big( \prod_{j=1}^{k}  x_{i_{1,j}}\big)
\big( \prod_{j=1}^{k^\prime}  x_{i_{3,j}}\big) x_{n_1+1}^{k- k^\prime }
 \right)\nonumber\\
&=&\frac{(x_{n_1+1} y_{n_1+1})^l}{\big(\prod_{r=1}^l(k-k' +r)\big)l!}
 (\tilde\Delta+\partial_{x_{n_1+1}}\partial_{ y_{n_1+1}})^l
( b_0)
\end{eqnarray}
by (\ref{a2.21}) and (\ref{a2.27}).
Thus
\begin{align}
\sum^{k^\prime }_{l=0} b_l =T\left( b_0 \right).
\end{align}

The proof of the second formula ($\ref{aa32}$) is similar to the first one.

\end{proof}

\begin{corollary}Fix $k,\alpha_{n_1+1},\beta_{n_1+1} \in\mbb N$.
For $j$, $i_{1,j} \in J_1$, $i_{3,j} \in J_3$, $v_1\in \mathbb{F}[X_{J_1},X_{J_2 \setminus \{n_1+1\}},\\ Y_{J_3}] \cup  \mathbb{F}\left[X_{J_1},Y_{J_2 \setminus \{n_1+1\}},Y_{J_3} \right]$, we have
\begin{align}\label{a2.34}
\begin{split}
 \frac{k! }{ \alpha_{n_1+1}  !} \  T\left( v_1
\left( \prod_{j=1}^{k- \alpha_{n_1+1} }  x_{i_{3,j}}\right) \cdot  x_{n_1+1}^{\alpha_{n_1+1}}
 \right) =
 E_{i_{3,k-\al_{n_1+1}},n_1+1} \cdots  E_{i_{3,1},n_1+1}
 \left( v_1 x_{n_1+1}^{k} \right) ,
\end{split}
\end{align}
when $k \geq \alpha_{n_1+1}$,
and
\begin{align}\label{a2.35}
\begin{split}
 \frac{k! }{ \beta_{n_1+1}  !} \  T\left(v_1
\left( \prod_{j=1}^{k-\beta_{n_1+1} }  y_{i_{1,j}}\right) \cdot  y_{n_1+1}^{\beta_{n_1+1}}
 \right) =
 E_{n_1+1,i_{1,1}}  \cdots E_{n_1+1,i_{1,k-\be_{n_1+1}}}
 \left( v_1 y_{n_1+1}^{k } \right) ,
\end{split}
\end{align}
where $k \geq \beta_{n_1+1}$.
\end{corollary}

\begin{proof}  By (\ref{a2.14}) and (\ref{a2.17}), we have

\begin{eqnarray}\label{a2.36}
& &
 \frac{k! }{\left( k- k_{13}\right) !} \  T\left( v_0 \cdot
\left( \prod_{j=1}^{k}  x_{i_{1,j}}\right)
\left( \prod_{j=1}^{k_{13}}  x_{i_{3,j}}\right) \cdot  x_{n_1+1}^{k-k_{13}}
 \right)\nonumber\\& =&
 E_{i_{3,k_{13}},n_1+1} \cdots  E_{i_{3,1},n_1+1}  E_{n_1+1,i_{1,1}}  \left(v_0 \cdot
\left( \prod_{j=1}^{k}  x_{i_{1,j}}\right)  \cdot  x_{n_1+1}^{k}
 \right).
\end{eqnarray}
Redenote
\begin{equation} k_{13}=k-\al_{n_1+1}.\end{equation}
Then (\ref{a2.36}) becomes
\begin{eqnarray}\label{a2.38}
& &
 \frac{k! }{\al_{n_1+1}!} \  T\left( v_0 \cdot
\left( \prod_{j=1}^{k}  x_{i_{1,j}}\right)
\left( \prod_{j=1}^{k-\al_{n_1+1}}  x_{i_{3,j}}\right) \cdot  x_{n_1+1}^{k-k_{13}}
 \right)\nonumber\\& =&
 E_{i_{3,k-\al_{n_1+1}},n_1+1} \cdots  E_{i_{3,1},n_1+1}  E_{n_1+1,i_{1,1}}  \left(v_0 \cdot
\left( \prod_{j=1}^{k}  x_{i_{1,j}}\right)  \cdot  x_{n_1+1}^{k}
 \right).
\end{eqnarray}
Note that in the proof of Lemma 2.1, all the arguments still work if we replace the polynomial $v_0$ by the rational function
\begin{equation}v_0=\frac{v_1}{\prod_{j=1}^{k}  x_{i_{1,j}}}.\end{equation}
Substituting the above equation into (\ref{a2.38}), we obtain (\ref{a2.34}). Equation (\ref{a2.35}) can be obtained similarly.
\end{proof}

For $m\in\mbb N$ and $r\in\{1,2,3\}$, we denote
\begin{equation}\label{a2.40}
X^{m}_{J_r}=\{x_{j_1}^{\alpha_{j_1}}  x_{j_2}^{\alpha_{j_2}}\dots x_{j_k}^{\alpha_{j_k}}
 \mid \{j_1,...,j_k\}=J_r;\;\al_{j_1},...,\al_{j_k}\in\mbb N;\;  \sum_{i=1}^k  \alpha_{j_i}=m \}\end{equation}
and
\begin{equation}\label{a2.41}
Y^{m}_{J_r}=\{y_{j_1}^{\alpha_{j_1}}  y_{j_2}^{\alpha_{j_2}}\dots y_{j_k}^{\alpha_{j_k}}
\mid \{j_1,...,j_k\}=J_r;\;\al_{j_1},...,\al_{j_k}\in\mbb N;\;  \sum_{i=1}^k  \alpha_{j_i}=m \}.\end{equation}
Next we want to introduce a degree function ${\mfk d} : {\msr A}_{\la\ell_1,\ell_2\ra}  \rightarrow \mathbb{N}$ which can help us
determine $M_k/M_{k-1}$. Our motivation comes from the facts as follow.
Suppose $\ell_1=-m_1$ and $\ell_2=-m_2$ with $m_1,m_2 \in \mathbb{N}$. By (\ref{a2.11}), we have
\begin{equation}M_0=\mbox{Span}\{X_{J_1}^{m_1}Y_{J_3}^{m_2}\}.\end{equation}
 Applying $E_{i_2,i_1}$ and $E_{i_3,i_2}$ to $M_0$ with $i_r\in J_r$, we have
\begin{align}
\{X^{m_1+1}_{J_1}x_{i_2}Y^{m_2}_{J_3} ,X^{m_1 }_{J_1}y_{i_2} Y^{m_2+1}_{J_3}  \mid
i_2  \in J_2\}\subset M_1\setminus M_0
\end{align}
by (\ref{a2.1})-(\ref{a2.3}).
Fix $k,k_{13},k_{21}\in\mbb N$. For each $j$, $i_{1,j} \in J_1$, $i_{3,j} \in J_3$, $v_0\in \mathbb{F}\left[X_{J_1},X_{J_2 \setminus \{n_1+1\}},Y_{J_3} \right] \cup  \mathbb{F}\left[X_{J_1},Y_{J_2 \setminus \{n_1+1\}},Y_{J_3} \right]$, we have
\begin{eqnarray}
& &
( -1)^{k} k! \  T\left( X_{J_1}^{m_1}Y^{m_2}_{J_3}
\left( \prod_{j=1}^{k}  x_{i_{1,j}}\right)
\left( \prod_{j=1}^{k}  x_{i_{3,j}}\right)
 \right)\nonumber\\& =&
 E_{i_{3,k},n_1+1} \cdots  E_{i_{3,1},n_1+1}  E_{n_1+1,i_{1,1}}  \cdots E_{n_1+1,i_{1,k}}
 \left( X_{J_1}^{m_1}Y^{m_2}_{J_3}  \right)\subset M_{2k}.
\end{eqnarray}
When $k_{21}=k$ in (\ref{a2.15}), we find
\begin{eqnarray}& &
( -1)^{k} k! \   T\left(X_{J_1}^{m_1}Y^{m_2}_{J_3}
\left( \prod_{j=1}^{k}  y_{i_{1,j}}\right)
\left( \prod_{j=1}^{k }  y_{i_{3,j}}\right)
 \right) \nonumber\\&=&
 E_{n_1+1,i_{1,1}}  \cdots E_{n_1+1,i_{1,k}}  E_{i_{3,1},n_1+1} \cdots  E_{i_{3,k},n_1+1}
 \left(X_{J_1}^{m_1}Y^{m_2}_{J_3} \right)\subset M_{2k}.
\end{eqnarray}
The above three expressions motivate us to define ${\mfk d} : {\msr A}_{\la\ell_1,\ell_2\ra}  \rightarrow \mathbb{N}$ by
\begin{align}
&{\mfk d} \left(x^\alpha y^\beta \right) =2 \sum_{i \in J_3} \alpha_i +\sum_{i \in J_2}\alpha_i +2\sum_{i \in J_1}\beta_i+\sum_{i \in J_2}\beta_i -\frac{\ell_1+|\ell_1|+\ell_2+|\ell_2|}{2}, \label{a2.46}\\
&{\mfk d} \left( \lmd f\right) ={\mfk d} \left( f\right) , \  \forall \lmd   \in \mathbb{F} \setminus \{0\} ,\\
&{\mfk d}\left(\sum_{i=1}^m  f_i\right) =\max \{{\mfk d}\left(f_1 \right) ,\cdots,{\mfk d}\left(f_m \right) \},\;\;\text{where $f_i$ are monomials.}
\label{a2.48}
\end{align}
The last term in (\ref{a2.46}) is to ensure $\mfk d(M_0)=0$.

For $\al,\be\in\mbb N^n$ with  $\alpha_{n_1+1} \beta_{n_1+1} =0$ and $ \left(x_{n_1+1} y_{n_1+1}\right)  \left(\tilde{\Delta}+\partial_{x_{n_1+1}}\partial_{ y_{n_1+1}}\right)  \left( x^\alpha y^\beta \right) \neq 0$, we have
\begin{eqnarray}&
&{\mfk d}( (x_{n_1+1} y_{n_1+1})
(\tilde{\Dlt}+\partial_{x_{n_1+1}}\partial_{ y_{n_1+1}})( x^\alpha y^\beta))\nonumber \\
&=&{\mfk d}((x_{n_1+1} y_{n_1+1})
( \sum^{n_1}_{i=1} x_i \partial_{y_i}-\sum^{n_2}_{r=n_1+2} \partial_{x_r} \partial_{y_r}+\sum^{n}_{s=n_2+1} y_s \partial_{x_s}
)( x^\alpha y^\beta))\nonumber\\
&=&\max \{{\mfk d}( x_{n_1+1} y_{n_1+1} x_i \partial_{y_i}( x^\alpha y^\beta )),\;
{\mfk d}(  x_{n_1+1} y_{n_1+1}\partial_{x_r} \partial_{y_r}( x^\alpha y^\beta)),
\nonumber\\\pse
 && {\mfk d}(  x_{n_1+1} y_{n_1+1}  y_s \partial_{x_s}( x^\alpha y^\beta))\mid
i\in J_1, n_1+2 \leq r \leq n_2, s \in J_3\}\nonumber\\\pse
&=&{\mfk d} \left(x^\alpha y^\beta\right).
\end{eqnarray}
In general,
\begin{align}
{\mfk d}((x_{n_1+1} y_{n_1+1})^i
(\tilde{\Delta}+\partial_{x_{n_1+1}}\partial_{ y_{n_1+1}})^i ( x^\alpha y^\beta))
={\mfk d} (x^\alpha y^\beta) \text{ or } 0,
\label{a2.49}\end{align}
where $\alpha_{n_1+1} \beta_{n_1+1} =0$.

Recall that $T$ is injective. Hence
\begin{align}
{\mfk d} (T( x^\alpha y^\beta)) ={\mfk d}(x^\alpha y^\beta),\quad\text{ where $ \alpha_{n_1+1} \beta_{n_1+1} =0$.}
\label{a2.50}\end{align}

For convenience, we denote
\begin{align}\label{a2.51}
L_3=J_3 \times J_1, \quad
L_2=J_3\times J_2,\quad
L_1=J_2\times J_1,\quad
L=L_1 \cup L_2 \cup L_3.
\end{align}

From the above properties, we have
\begin{align}
{\mfk d} \left( E_{i,j}(f) \right)  \leq
\begin{cases}
{\mfk d} \left( f \right)+2  &\text{if} \left( i,j\right)  \in L_3, \\
{\mfk d} \left( f \right)+1  &\text{if}  \left( i,j\right)  \in   L_1  \cup   L_2,  \\
{\mfk d} \left( f \right) &\text{if}  \left( i,j\right)  \notin  L.
\end{cases}
\label{a2.52}
\end{align}

We set
\begin{equation}\label{a2.53}
TN(k) = \text{Span} \{  T (x^\alpha y^\beta )      \mid
{\mfk d} \left(x^\alpha y^\beta  \right) = k ,x^\alpha y^\beta   \in {\msr A}_{\left \langle \ell_1,\ell_2  \right \rangle  }, \alpha_{n_1+1} \beta_{n_1+1} =0  \},\end{equation}
\begin{equation}\label{a2.54}
TN_k = \text{Span} \{  T (x^\alpha y^\beta )      \mid
{\mfk d} \left(x^\alpha y^\beta  \right)\leq k ,x^\alpha y^\beta   \in {\msr A}_{\left \langle \ell_1,\ell_2  \right \rangle  }, \alpha_{n_1+1} \beta_{n_1+1} =0  \}.\end{equation}
According to (\ref{a2.1})-(\ref{a2.3}), the alternating polynomials
\begin{equation}\label{a2.55}
x_{i_1} x_{i_3}-y_{i_1} y_{i_3}=-\td{\pi}(E_{i_3,i_1})\qquad\for\;\; i_1 \in J_1,\; i_3\in J_3 \end{equation}
(cf. (\ref{aa1.13})). By (\ref{a2.7}),
\begin{equation}\label{a2.56}
T(f)(x_{i_1} x_{i_3}-y_{i_1} y_{i_3})=T(f(x_{i_1} x_{i_3}-y_{i_1} y_{i_3}))\qquad\for\;\;f\in\msr A.\end{equation}
Set
\begin{align}\label{a2.57}
P = \{x_{i_1} x_{i_3}-y_{i_1} y_{i_3} \mid i_1 \in J_1, i_3\in J_3 \}.
\end{align}
Then for $r\in\mbb N$,
\begin{eqnarray}& &TN\begin{pmatrix}
k_{12} & k_{12} & k_{13} \\
k_{21} & k_{22} &  k_{23}
\end{pmatrix}P^r \nonumber\\& =&\text{Span} \{
 T\left(f \right) p_1 \cdot p_2 \cdots p_r \mid
f \in N\begin{pmatrix}
k_{12} & k_{12} & k_{13} \\
k_{21} & k_{22} &  k_{23}
\end{pmatrix},
p_i \in P,1\leq i\leq r
\} \nonumber\\
& =& T\left( N\begin{pmatrix}
k_{12} & k_{12} & k_{13} \\
k_{21} & k_{22} &  k_{23}
\end{pmatrix}P^r
\right)\subset \msr H_{\la\ell_1,\ell_2\ra}.
\end{eqnarray}
Moreover, we have
\begin{align}
{\mfk d}\left(f \cdot p \right) = {\mfk d}\left(f \right)+2\qquad \for\;\; p\in P,\  0\neq f  \in  {\msr A}_{\left \langle \ell_1,\ell_2  \right \rangle  }.
\label{a2.59}
\end{align}

Set
\begin{align}\label{a2.60}
V_k &= \text{Span}\{
TN\left(i\right),TN\left(k-i\right) P^i \mid
i=0,1,\dots,k\}
\nonumber\\
&=\text{Span}\{
TN_k,TN\left(k-i\right) P^i \mid
i=1,\dots,k\}.
\end{align}
Then $M_0=V_0=TN_0=TN(0)$.  Denote $\mfk g=sl(n)$.

\begin{proposition} For $k\in \mbb N$, we have $M_k=U_k(\mfk g)(M_0)=V_k$. Moreover, $M= {\msr H}_{\la\ell_1,\ell_2\ra}=U(\mfk g)(M_0)$.
\label{ap2.1}\end{proposition}

\begin{proof} We prove $M_i=V_i$ for any $i\in\mbb N$ by induction. It holds when $i=0$.
Suppose that $M_i=V_i$ with $i \leq k - 1$. Consider $i=k$. Note
\begin{align}\label{a2.61}
M_k=U_k\left(\mfk g \right)(M_0) = \mfk g( M_{k-1})=\mfk g(V_{k-1}).
\end{align}
Let $ f \in TN_{k-1}$.
If $\left( i,j\right)  \notin L_3 $,
\begin{align}
{\mfk d} \left( E_{i,j}(f) \right)  \leq
{\mfk d} \left( f \right)+1\ \
\Rightarrow E_{i,j}(f) \in V_k.
\end{align}
When $\left( i,j\right)  \in L_3 $, (\ref{a2.55}) yields
\begin{equation}
E_{i,j}(f) \in TN_{k-1} P \subset V_k.\end{equation}

Fix  $g \in TN_{k-r-1} P^r$ with $r\in\ol{1,k-1}$.
If $\left( i,j\right)  \in L_3 $, then
\begin{equation}E_{i,j}(g)\in TN_{k-r-1} P^{r+1}  \subset V_k.\end{equation}
 According (\ref{a2.55})-(\ref{a2.57}), we may assume
\begin{equation}g=E_{i_1,j_1}\cdots E_{i_r,j_r}(g')\end{equation}
with $g'\in TN_{k-r-1}\subset V_{k-r-1}=M_{k-r-1}$ and $(i_1,j_1),...,(i_r,j_r)\in L_3$.
When $\left( i,j\right)  \notin L_3 $,
 \begin{eqnarray}\label{a2.66}
E_{i,j}(g)&=&\sum_{s=1}^r (E_{i_1,j_1}\cdots E_{i_{s-1},j_{s-1}}[E_{i,j},E_{i_s,j_s}]
  E_{i_{s+1},j_{s+1}}\cdots  E_{i_r,j_r}(g')) \nonumber\\
  & &+E_{i_1,j_1}\cdots E_{i_r,j_r}E_{i,j}(g')\subset M_{k-1}+P^rTN_{k-r}\subset V_k\end{eqnarray}
 by (\ref{a2.60}). Thus we have proved $M_k\subset V_k$.

Next we want to prove $V_k \subset M_k$. Since the case of $\ell_2\leq 0\leq \ell_1 $ and the case of  $\ell_1\leq 0\leq \ell_2$ are symmetric, it is enough to prove it in the following two cases.\psp

{\it Case 1. $\ell_1,\ell_2\leq 0$.}\psp

We write $\ell_1=-m_1$, $\ell_2=-m_2$, where $m_1,m_2 \in \mathbb{N}$. For $k_{13},\alpha_{n_1+1}\in \mbb N$,  Lemma 2.1 gives
\begin{align*}
\left( -1\right)^{k_{13}+\alpha_{n_1+1}} \frac{(k_{13}+\alpha_{n_1+1})! }{\alpha_{n_1+1} !} \  T\left(X_{J_1}^{m_1} X_{J_3}^{m_2} \cdot
\left( \prod_{j=1}^{k_{13}+\alpha_{n_1+1}}  x_{i_{1,j}}\right)
\left( \prod_{j=1}^{k_{13}}  x_{i_{3,j}}\right) \cdot  x_{n_1+1}^{\alpha_{n_1+1}}
 \right) =\\
 E_{i_{3,1},n_1+1} \cdots  E_{i_{3,k_{13}},n_1+1}  E_{n_1+1,i_{1,1}}  \cdots E_{n_1+1,i_{1,k_{13}+\alpha_{n_1}+1}}
 \left( X_{J_1}^{m_1} Y_{J_3}^{m_2} \right) ,
\end{align*}
where $i_{1,r} \in J_1, i_{3,s} \in J_3$.
This shows that
\begin{align}
 T\left(X_{J_1}^{m_1} Y_{J_3}^{m_2} \cdot
\left( \prod_{j=1}^{k_{13}+\alpha_{n_1+1}}  x_{i_{1,j}}\right)
\left( \prod_{j=1}^{k_{13}}  x_{i_{3,j}}\right) \cdot  x_{n_1+1}^{\alpha_{n_1+1}}
 \right)
\subset M_{2 k_{13}+\alpha_{n_1+1}},
\end{align}
which implies
\begin{align}
T\left(N\begin{pmatrix}
m_1+k_{13}+ \alpha_{n_1+1}& 0 & k_{13} \\
0 & 0 & m_2
\end{pmatrix}
\cdot x_{n_1+1}^{\alpha_{n_1+1}}\right)
\subset M_{2 k_{13}+\alpha_{n_1+1}}.
\end{align}

Next we want to prove
\begin{align}\label{a2.69}
T\left(
 N\begin{pmatrix}
m_1+k_{13}+\alpha_{n_1+1} & 0 & k_{13} \\
k_{21}  & 0 & m_2+k_{21}
\end{pmatrix}
\cdot x_{n_1+1}^{\alpha_{n_1+1}}
\right)
\subset
 M_{2 \left( k_{13}+k_{21} \right) +\alpha_{n_1+1}}.
\end{align}
by induction on $k_{21}$.
It holds for $k_{21}=0$. Suppose that it holds for $k_{21} \leq r$.
Let
\begin{equation}h \in N\begin{pmatrix}
m_1+k_{13}+\alpha_{n_1+1} & 0 & k_{13} \\
r  & 0 & m_2+r
\end{pmatrix}
\cdot x_{n_1+1}^{\alpha_{n_1+1}}.
\end{equation}
Note
\begin{align}
T\left( h\cdot y_{i_1} y_{i_3} \right) =
-T\left( h \cdot \left( x_{i_1} x_{i_3}- y_{i_1} y_{i_3} \right)  \right) +
T\left(h\cdot x_{i_1} x_{i_3} \right).
\end{align}
Moreover,
 \begin{eqnarray}& &T\left(h\cdot x_{i_1} x_{i_3} \right) \in \left(
 N\begin{pmatrix}
m_1+k_{13}+\alpha_{n_1+1}+1 & 0 & k_{13}+1
 \\ r & 0 & m_2+ r
\end{pmatrix}
\cdot x_{n_1+1}^{\alpha_{n_1+1}}
\right)\nonumber\\ & &
\subset
  M_{2 \left(k_{13}+1\right)+2r+\alpha_{n_1+1}}\end{eqnarray}by the inductional assumption, and
\begin{eqnarray}& &-T\left( h \cdot \left( x_{i_1} x_{i_3}- y_{i_1} y_{i_3} \right)  \right)\nonumber\\&=&
-( x_{i_1} x_{i_3}- y_{i_1} y_{i_3}) T(h)=E_{i_3,i_1}(T(h))\nonumber\\ &\subset&
E_{i_3,i_1}\left[T\left(N\begin{pmatrix}
m_1+k_{13}+\alpha_{n_1+1} & 0 & k_{13} \\
r  & 0 & m_2+r
\end{pmatrix}
\cdot x_{n_1+1}^{\alpha_{n_1+1}}\right)\right]\nonumber\\
&\subset& E_{i_3,i_1}(M_{2 k_{13}+2r+\alpha_{n_1+1}})\subset M_{2 k_{13}+2r+1+\alpha_{n_1+1}}\subset M_{2 k_{13}+2r+2+\alpha_{n_1+1}}\end{eqnarray}
by (\ref{a2.55}) and (\ref{a2.56}). So the above three expressions imply
\begin{equation}
T\left( h\cdot y_{i_1} y_{i_3} \right) \in M_{2 k_{13}+2r+2+\alpha_{n_1+1}}.\end{equation}
This proves (\ref{a2.69}). Symmetrically, we have
\begin{align}
T\left(
N\begin{pmatrix}
m_1+k_{13}& 0 & k_{13} \\
k_{21} & 0 &m_2+k_{21+\beta_{n_1+1}}
\end{pmatrix}
\cdot y_{n_1+1}^{\beta_{n_1+1}}
\right)
\in
M_{2 \left( k_{13}+k_{21}\right) +\beta_{n_1+1}}.
\end{align}

Let $k_{12},k_{22}\in \mbb N$. For convenience, we redenote
\begin{equation}k_1=m_1+k_{12}+ k_{13}+\al_{n_1+1},\;\;k_2=m_2+k_{21}+k_{22}\end{equation}
and
\begin{equation}
N'\begin{pmatrix}
k_1  &k_{12} & k_{13} \\
k_{21} & k_{22} &k_2
\end{pmatrix}=\left\{x^\al y^\be\in N\begin{pmatrix}
k_1  &k_{12} & k_{13} \\
k_{21} & k_{22} &k_2
\end{pmatrix}\mid\al_{n_1+1}=\be_{n_1+1}=0\right\}\end{equation}
(cf. (\ref{a2.9})). We want to prove
\begin{equation}\label{a2.78}
T\left(N'\begin{pmatrix}
k_1  &k_{12} & k_{13} \\
k_{21} & k_{22} &k_2
\end{pmatrix}
\cdot x_{n_1+1}^{\alpha_{n_1+1}}
\right)\subset
 M_{2 k_{21}+2 k_{13}+k_{12}+k_{22}+\alpha_{n_1+1}}
\end{equation}
by induction on $k_{12}+k_{22}$. When $k_{12}+k_{22}=0$, it is exactly (\ref{a2.69}). Assume that it holds for
 $k_{12}+k_{22} < k_0 $.  Let $i_1\in J_1$ and $i_2\in J_2\setminus\{n_1+1\}$. According to (\ref{a2.1})-(\ref{a2.3}) and (\ref{a2.7}),
 \begin{equation}\td{\pi}(E_{i_2,i_1})=-x_{i_1}x_{i_2}-y_{i_1}\ptl_{y_{i_2}},\;\;\td{\pi}(E_{i_2,i_1})T=T\td{\pi}(E_{i_2,i_1}).\end{equation}
 Let
 \begin{equation}h\in N'\begin{pmatrix}
k_1  &k_{12} & k_{13} \\
k_{21} & k_{22} &k_2
\end{pmatrix}
\cdot x_{n_1+1}^{\alpha_{n_1+1}}\;\;\mbox{with}\;\;k_{12}+k_{22}=k_0-1.\end{equation}
 We have
 \begin{equation}E_{i_2,i_1}(T(h))=T(E_{i_2,i_1}(h))=-T(x_{i_1}x_{i_2}h)-T(y_{i_1}\ptl_{y_{i_2}}(h)).\end{equation}
Note
 \begin{equation}x_{i_1}x_{i_2}h\in N'\begin{pmatrix}
k_1+1  &k_{12}+1 & k_{13} \\k_{21} & k_{22} &k_2\end{pmatrix}\cdot x_{n_1+1}^{\alpha_{n_1+1}}\end{equation}
and
\begin{eqnarray}& &T(y_{i_1}\ptl_{y_{i_2}}(h))\in
T\left(N'\begin{pmatrix}
k_1  &k_{12} & k_{13} \\
k_{21}+1 & k_{22}-1 &k_2
\end{pmatrix}
\cdot x_{n_1+1}^{\alpha_{n_1+1}}\right)\nonumber\\ & &\subset M_{2 k_{21}+2 k_{13}+k_{12}+k_{22}+1+\alpha_{n_1+1}}.\end{eqnarray}
 On the other hand,
 \begin{equation}E_{i_2,i_1}(T(h))\in E_{i_2,i_1}(M_{2 k_{21}+2 k_{13}+k_{12}+k_{22}+\alpha_{n_1+1}})\subset M_{2 k_{21}+2 k_{13}+k_{12}+k_{22}+1+\alpha_{n_1+1}}.\end{equation}
 Thus
 \begin{equation}T(x_{i_1}x_{i_2}h)=-E_{i_2,i_1}(T(h))-T(y_{i_1}\ptl_{y_{i_2}}(h))\in  M_{2 k_{21}+2 k_{13}+k_{12}+k_{22}+1+\alpha_{n_1+1}}.\end{equation}
 This proves
 \begin{equation}T\left(N'\begin{pmatrix}k_1+1  &k_{12}+1 & k_{13} \\k_{21} & k_{22} &k_2\end{pmatrix}\cdot x_{n_1+1}^{\alpha_{n_1+1}}\right)
 \subset M_{2 k_{21}+2 k_{13}+k_{12}+k_{22}+1+\alpha_{n_1+1}}.
 \end{equation}
 By applying $E_{i_3,i_2}$ with $i_2\in J_2\setminus \{n_1+1\}$ and $i_3\in J_3$, we can similarly obtain
 \begin{equation}T\left(N'\begin{pmatrix}k_1  &k_{12} & k_{13} \\k_{21} & k_{22}+1 &k_2+1\end{pmatrix}\cdot x_{n_1+1}^{\alpha_{n_1+1}}\right)
 \subset M_{2 k_{21}+2 k_{13}+k_{12}+k_{22}+1+\alpha_{n_1+1}}.
 \end{equation}
 Hence (\ref{a2.78}) holds for $k_{12}+k_{22}=k_0$. By induction, (\ref{a2.78}) holds for any $k_{12},k_{22}\in\mbb N$.
 Using the symmetry of $x_i's$ and $y_j's$, we have
 \begin{equation}\label{a2.88}
T\left(N'\begin{pmatrix}
k_1  &k_{12} & k_{13} \\
k_{21} & k_{22} &k_2
\end{pmatrix}
\cdot y_{n_1+1}^{\be_{n_1+1}}
\right)\subset
 M_{2 k_{21}+2 k_{13}+k_{12}+k_{22}+\be_{n_1+1}},
 \end{equation}
 where
 \begin{equation}k_1=m_1+k_{12}+ k_{13},\;\;k_2=m_2+k_{21}+k_{22}+\be_{n_1+1}.\end{equation}
Expressions (\ref{a2.53}), (\ref{a2.78}) and  (\ref{a2.88}) show
\begin{equation}\label{a2.90}
TN(\ell)\subset M_\ell\qquad\for\;\;\ell\in\mbb N. \end{equation}

For $0<r\in \mbb N$, any element $w\in TN(k-r)\cdot P^r$ can be written as
\begin{equation}\label{a2.91}
w=E_{i_1,j_1}\cdots E_{i_r,j_r}(h)\in E_{i_1,j_1}\cdots E_{i_r,j_r}(M_{k-r})\subset M_k
 \end{equation}
by (\ref{a2.55})-(\ref{a2.57}), where $h\in TN(k-r)\subset M_{k-r}$, $i_1,\dots,i_r\in J_1$ and $j_1,\dots,j_r\in J_2$.
Now Expressions (\ref{a2.54}), (\ref{a2.90}) and (\ref{a2.91}) yield $V_k\subset M_k$. Recall we have already proved $M_k\subset V_k$ before Case 1. Thus
$M_k=V_k$. By induction, the first conclusion in the proposition holds in this case.
\psp

{\it Case 2. $\ell_1\leq 0\leq \ell_2$}.\psp

In this case,
\begin{equation}M_0=TN\begin{pmatrix}
-\ell_1  & 0 & 0 \\
0 & \ell_2 &0
\end{pmatrix}.\end{equation}
For any $r\in\mbb N$ and
\begin{equation}\label{a2.93}
g\in N\begin{pmatrix}
-\ell_1  & 0 & 0 \\
0 & \ell_2+r &r
\end{pmatrix},\end{equation}
\begin{equation}(\td\Dlt+\ptl_{x_{n_1+1}}\ptl_{y_{n_1+1}})(g)=(\sum_{i=1}^{n_1}x_i\ptl_{y_i}-\sum_{r=n_1+2}^{n_2}\ptl_{x_r}\ptl_{y_r}+\sum_{s=n_2+1}^n
y_s\ptl_{x_s})(g)=0\end{equation}
by (\ref{a1.7}) and (\ref{a2.9}). Thus (\ref{a2.6}) and (\ref{a2.10}) give
\begin{equation}\label{a2.95}
TN\begin{pmatrix}
-\ell_1  & 0 & 0 \\
0 & \ell_2+r &r
\end{pmatrix}=\mbox{Span}\left\{N\begin{pmatrix}
-\ell_1  & 0 & 0 \\
0 & \ell_2+r &r
\end{pmatrix}\right\}.\end{equation}
According to (\ref{a2.1})-(\ref{a2.3}), we have
\begin{equation}\td{\pi}(E_{i_3,i_2})= x_{i_3} \partial_{x_{i_2}}+y_{i_3}y_{i_2}\qquad\for\;\;i_2\in J_2,\;i_3\in J_3.\end{equation}
For $g$ in (\ref{a2.93}), we get
\begin{equation}\label{a2.97}
E_{i_3,i_2}(g)=y_{i_3}y_{i_2}g\in N\begin{pmatrix}
-\ell_1  & 0 & 0 \\
0 & \ell_2+r+1 &r+1
\end{pmatrix}.\end{equation}
By (\ref{a2.95}), (\ref{a2.97}) and induction, we can prove
\begin{align}
TN\begin{pmatrix}
-\ell_1  & 0 & 0 \\
0 & k_{23}+\ell_2 &k_{23}
\end{pmatrix}
\subset M_{k_{23}}\qquad\for\;\;k_{23}\in\mbb N.
\end{align}

 Let $k_{21},k_{22},k_{23},\be_{n_1+1}\in\mbb N$ such that $k_{22}\geq \be_{n_1+1}$ and
 \begin{equation} \label{a2.99}
k_{21}+k_{22}-k_{23}=\ell_2.\end{equation}
Take
\begin{equation}v_1\in X_{J_1}^{-\ell_1}Y_{J_2\setminus\{n_1+1\}}^{k_{22}-\be_{n_1+1}}Y_{J_3}^{k_{23}}.\end{equation}
 Then
 \begin{equation}v_1y_{n_1+1}^{k_{21}+\be_{n_1+1}}\in
TN\begin{pmatrix}
-\ell_1  & 0 & 0 \\
0 & k_{23}+\ell_2 &k_{23}
\end{pmatrix}
\subset M_{k_{23}}\end{equation}
and
\begin{equation}v_1
\left( \prod_{j=1}^{k_{21}}  y_{i_{1,j}}\right) \cdot  y_{n_1+1}^{\beta_{n_1+1}}\in N\begin{pmatrix}
-\ell_1  & 0 & 0 \\
k_{21} & k_{22} &k_{23}
\end{pmatrix}.\end{equation}
By (\ref{a2.35}) with $k=k_{21}+\be_{n_1+1}$, we have
\begin{eqnarray}& &
 \frac{(k_{21}+\be_{n_1+1})! }{ \beta_{n_1+1}  !} \  T\left(v_1
\left( \prod_{j=1}^{k_{21}}  y_{i_{1,j}}\right) \cdot  y_{n_1+1}^{\beta_{n_1+1}}
 \right)\nonumber\\&=&
 E_{n_1+1,i_{1,1}}  \cdots E_{n_1+1,i_{1,k_{21}}}
 \left( v_1 y_{n_1+1}^{k_{21}+\be_{n_1+1}} \right)\subset M_{k_{21}+k_{23}}.
\end{eqnarray}
This shows
\begin{align}
TN\begin{pmatrix}-\ell_1  & 0 & 0 \\ k_{21} & k_{22} &k_{23}
\end{pmatrix}\subset M_{k_{21}+k_{23}}.
\end{align}

Since \begin{eqnarray}\label{a2.105}
T\left(h\cdot x_{i_1} x_{i_3} \right)&=&
T\left( h\cdot y_{i_1} y_{i_3} \right)
+T\left( h \cdot \left( x_{i_1} x_{i_3}- y_{i_1} y_{i_3} \right)  \right)\nonumber\\
&=&T\left( h\cdot y_{i_1} y_{i_3} \right)-E_{i_3,i_1}[
T\left( h   \right)]
\end{eqnarray}for $i_1\in J_1$ and $i_3\in J_3$, we have
\begin{eqnarray}\label{a2.106}&&
TN\begin{pmatrix}
-\ell_1+1  & 0 & 1 \\
k_{21} & k_{22} &k_{23}
\end{pmatrix}\nonumber \\&
&\subset {\mfk g}\left[TN\begin{pmatrix}
-\ell_1   & 0 & 0 \\
k_{21} & k_{22} &k_{23}
\end{pmatrix}\right]+  TN\begin{pmatrix}
-\ell_1   & 0 & 0 \\
k_{21}+1 & k_{22} &k_{23}+1
\end{pmatrix}\nonumber\\
&&\subset
M_{k_{21}+k_{23}+1}+M_{k_{21}+k_{23}+2}\subset M_{k_{21}+k_{23}+2}
.
\end{eqnarray}
Repeatedly applying (\ref{a2.105}), we can prove
\begin{align}\label{a2.107}
TN\begin{pmatrix}
k_{11}  & 0 & k_{13} \\
k_{21} & k_{22} &k_{23}
\end{pmatrix}
\subset
M_{k_{21}+k_{23}+2k_{13}}.
\end{align}
 by induction.

For $i_1\in J_1$ and $i_2\in J_2\setminus\{n_1+1\}$, (\ref{a2.1})-(\ref{a2.3}) and (\ref{a2.7}) show
\begin{align}\label{a2.108}
T\left( x_{i_1} x_{i_2}  h \right)=-E_{i_2,i_1}(T\left( h\right)) -T\left( y_{i_1} \partial_{y_{i_2}} h \right) ,
\end{align}
Denote
\begin{equation}N^\ast\begin{pmatrix}
k_{11}  & k_{12} & k_{13} \\
k_{21} & k_{22} &k_{23}
\end{pmatrix}=\left\{x^\al y^\be\in N\begin{pmatrix}
k_{11}  & k_{12} & k_{13} \\
k_{21} & k_{22} &k_{23}
\end{pmatrix}\mid\al_{n_1+1}=0\right\}.\end{equation}
Applying (\ref{a2.108}) to (\ref{a2.107}), we get
\begin{eqnarray}\label{b2.111}
& &
T\left[N^\ast\begin{pmatrix}
k_{11}  & 1 & k_{13} \\
k_{21} & k_{22} &k_{23}
\end{pmatrix}\right]\nonumber\\
&& \subset\mfk g\left[TN\begin{pmatrix}
k_{11}-1   & 0 & k_{13} \\
k_{21} & k_{22} &k_{23}
\end{pmatrix}\right]+  TN\begin{pmatrix}
k_{11}-1   & 0 & k_{13} \\
k_{21}+1 & k_{22}-1 &k_{23}
\end{pmatrix}\nonumber\\
&&\subset
M_{k_{21}+k_{23}+2k_{13}+1}+M_{k_{21}+1+k_{23}+2k_{13}}\subset M_{k_{21}+k_{23}+2k_{13}+1}.
\end{eqnarray}
By (\ref{a2.107}), (\ref{b2.111}) and induction on $k_{12}$, we can deduce
\begin{align}\label{a2.111}
T\left[N^\ast \begin{pmatrix}
k_{11}  & k_{12} & k_{13} \\
k_{21} & k_{22} &k_{23}
\end{pmatrix}\right]
\subset M_{k_{21}+k_{23}+2k_{13}+k_{12}}=M_{2k_{21}+k_{22}+2k_{13}+k_{12}-\ell_2}.
\end{align}

Denote
\begin{equation}N^\sharp\begin{pmatrix}
k_{11}  & k_{12} & k_{13} \\
k_{21} & k_{22} &k_{23}
\end{pmatrix}=\left\{x^\al y^\be\in N\begin{pmatrix}
k_{11}  & k_{12} & k_{13} \\
k_{21} & k_{22} &k_{23}
\end{pmatrix}\mid\be_{n_1+1}=0\right\}.\end{equation}
By the symmetry of $x_i's$ and $y_j's$, we have
\begin{align}\label{a2.113}
T\left[N^\sharp \begin{pmatrix}
k_{11}  & k_{12} & k_{13} \\
k_{21} & k_{22} &k_{23}
\end{pmatrix}\right]
\subset M_{k_{21}+k_{23}+2k_{13}+k_{12}}=M_{2k_{21}+k_{22}+2k_{13}+k_{12}-\ell_2}.
\end{align}
Expressions (\ref{a2.111}) and (\ref{a2.113}) yield $TN_k \subset M_k$. According to (\ref{a2.91}), we have $V_k\subset M_k$.
Recall we have already proved $M_k\subset V_k$ before Case 1. Thus $M_k=V_k$. By induction, the first conclusion in the proposition holds in this case.\pse

By the symmetry of $x_i's$ and $y_j's$, the first conclusion in the proposition holds in any case. Thanks  to (\ref{a2.54}) and (\ref{a2.60}),
 \begin{equation}M=\msr H_{\la\ell_1,\ell_2\ra}\subset\bigcup_{k=0}^\infty TN_k\subset \bigcup_{k=0}^\infty V_k=\bigcup_{k=0}^\infty M_k=M.\end{equation}
So (\ref{a1.9}) holds.
\end{proof}

According to (\ref{a2.54}) and (\ref{a2.60}), $M_k=V_k$ is spanned by the set
 \begin{eqnarray}\msr S_k&=&\{T(x^\al y^\be),T(x^\al y^\be)P^{k-\mfk d(x^\al y^\be)}\mid x^\al y^\be\in \msr A_{\la\ell_1,\ell_2\ra},\nonumber\\& &\;\al_{n_1+1}\be_{n_1+1}=0,\;\mfk d( x^\al y^\be)\leq k\}.\end{eqnarray}

\section{Products of  Alternating Polynomials}

In this section, we want to study the linear dependence among the elements in $\msr S_k$.

Denote
\begin{equation} Z = \{z_{j,i} \mid i \in J_1, j \in J_3 \} \end{equation}
 as a set of $n_1(n-n_2)$ variables. The notion is compatible with (\ref{a2.55}).
Set
\begin{equation}\label{a3.2}\msr B=\mbb F[Z]\end{equation}
the polynomial algebra in $Z$. Denote by $\msr B_m$ the subspace of its homogeneous polynomials with degree $m$ and by $Z^{(m)}$ the set of monomials with degree $m$.

Define associative algebra homomorphisms $\vf_x,\vf_y: \msr B\rta\msr A$ by
\begin{equation}\label{a3.3}
\vf_x(z_{j,i})=x_ix_j,\;\;\vf_y(z_{j,i})=y_iy_j\qquad\for\;\;i\in J_1,\;j\in J_3.\end{equation}
Then $\ker \varphi_x $ and $\ker \varphi_y $  are graded ideals of $\msr B$. For subset $S$ of a polynomial algebra, we use $\la S\ra$ to denote the ideal generated by $S$. Let
\begin{equation} \label{a3.4}\msr R_2=\left\la \left|\begin{array}{cc}
z_{j_1,i_1}  & z_{j_1,i_2}\\
z_{j_2,i_i
} & z_{j_2,i_2}\end{array}\right|\;\mid\;i_1,i_2\in J_1,\;j_1,j_2\in J_3\right\ra.\end{equation}

\begin{lemma}
\label{alem:2}
We have
\begin{align} \label{a3.5}
\ker \varphi_x = \ker \varphi_y = \msr R_2.
\end{align}

\end{lemma}

\begin{proof} First, we want to prove $\ker \varphi_x = \msr R_2$.
Let $S_r$ be the permutation group on $\{1,2,...,r\}$. We define a linear action of the direct product group $S_r\times S_r$ on  $\msr B_r$ by
\begin{equation}
(\sgm_1,\sgm_2)\left( \prod_{s=1}^r z_{j_s,i_s} \right)
=\prod_{s=1}^r z_{j_{\sigma_1\left( s\right) },i_{\sigma_2\left( s\right) }},
\end{equation}
where $(j_1,i_1)\geq (j_2,i_2) \geq \cdots \geq (j_r,i_r) $ in lexical order. Moreover,
\begin{align}
\varphi_x \left(g\right)= \varphi_x[\sigma \left(g\right)]
\end{align}
for each $\sigma \in S_r \times S_r  $ and $g \in \msr B_r$.

Define
\begin{equation}
\msr I_1\left( \prod_{s=1}^k z_{j_s,i_s} \right)= \{i_1,i_2,...,i_k  \}\subset J_1,\;\;\msr I_3\left( \prod_{s=1}^k z_{j_s,i_s} \right)= \{j_1,j_2,...,j_k \}\subset J_3.\end{equation}
For example, $\msr I_1\left(z_{n,1}z_{n-1,2}z_{n,2}\right)=\{1,2,2\}.$  Let $f\in \ker \varphi_x \cap \msr B_r$. Since the image under $\vf_x$ of the monomials in $\msr B$ with
different $\msr I_1$ and $\msr I_3$ are linearly independent, we may assume  that $f=\sum_{k=1}^m a_k f_k$ is a linear combination of the monomials with
\begin{equation}\msr I_1(f_1)=\msr I_1(f_2)=\cdots =\msr I_1(f_m),\quad \msr I_3(f_1)=\msr I_3(f_2)=\cdots =\msr I_3(f_m).\end{equation}
Moreover, $\sum_{k=1}^m a_k=0$ by $\vf_x(f)$=0. There exists $\sigma_k \in S_r \times S_r$ such that $\sigma_k \left(f_k\right) =f_1$.
Thus
\begin{equation}
f =\sum_{k=1}^m a_k f_k=\sum_{k=2}^m a_k \sigma_k(f_1) +a_1 f_1
=\sum_{k=2}^m a_k \left( \sigma_k(f_1) -f_1 \right) .
\end{equation}

Next, we want to prove  $\sigma \left( f \right)  -f\in \msr R_2$ for any $\sigma \in S_r \times S_r$ and
$f \in Z^{\left(r \right) } $. Firstly, $\msr R_2$ is invariant under $S_r \times S_r$ because
\begin{eqnarray}& &
\sigma \left( \left( z_{i,j}z_{i^\prime,j^\prime}-z_{i^\prime,j}z_{i,j^\prime}\right)  g \right)
\nonumber \\& =&\left( z_{\sigma_1\left( i\right) ,\sigma_2\left(j\right)}z_{\sigma_1\left( i^\prime\right),\sigma_2\left(j^\prime\right)}-z_{\sigma_1\left( i^\prime\right),\sigma_2\left(j\right)}z_{\sigma_1\left( i\right),\sigma_2\left(j^\prime\right)}\right) \sigma\left( g\right) \in \msr R_2
\end{eqnarray}
for each $g\in Z^{(r-2)}$. We set
\begin{align}
\sigma=\left( \sigma_1^1 \cdots \sigma_1^p,\sigma_2^1 \cdots \sigma_2^q \right) ,
\end{align}
where $\sigma_i^j \in S_r$ are transpositions. Note that
\begin{eqnarray}\label{a3.13}& &
\left( \left( \iota_1\: \iota_2\right),1 \right) .\left(  \prod_{s=1}^r z_{j_s,i_s}  \right) -\prod_{s=1}^r z_{j_s,i_s}\nonumber
\\ &=& \left(\prod_{s \neq \iota_1,\iota_2} z_{i_s,j_s}\right)  \cdot
\left( z_{j_{\iota_2},i_{\iota_1}} z_{j_{\iota_1},i_{\iota_2}}
-z_{j_{\iota_1},i_{\iota_1}} z_{j_{\iota_2},i_{\iota_2}}\right)  \in \msr R_2.
\end{eqnarray}
Symmetrically,
\begin{equation}\left(1, (\iota_1\: \iota_2)\right) \left( \prod_{s=1}^r z_{j_s,i_s}  \right) -\prod_{s=1}^r z_{j_s,i_s}\in\msr R_2.\end{equation}
Let $h\in Z^{(r)}$. We have
\begin{eqnarray}\sgm(g)-g&=&( \sigma_1^1 \cdots \sigma_1^p,\sigma_2^1 \cdots \sigma_2^q)(g)-g\nonumber\\&=&
\sum_{s=1}^p((\sgm_1^s,1)[(\sgm_1^{s+1}\cdots\sgm_1^p,\sigma_2^1 \cdots \sigma_2^q)(g)]-(\sgm_1^{s+1}\cdots\sgm_1^p,\sigma_2^1 \cdots \sigma_2^q)(g))\nonumber\\&&+\sum_{t=1}^q((1,\sgm_2^t)[(1,\sgm_2^{t+1}\cdots\sgm_2^q)(g)]-(1,\sgm_2^{t+1}\cdots\sgm_2^q)(g))\in\msr R_2.
\end{eqnarray}
Hence
\begin{equation}
f=\sum_{k=2}^m a_k \left( \sigma_k(f_1) -f_1 \right)\in\msr R_2
\end{equation}
by (\ref{a3.13}). So $\ker \varphi_x\cap\msr B_r\subset \msr R_2$ for each $r\geq 2$. Symmetrically, $\ker \varphi_y\cap\msr B_r\subset \msr R_2$.
Furthermore, (\ref{a3.3}) and (\ref{a3.5}) imply $\vf_x(\msr R_2)=\vf_y(\msr R_2)=\{0\}$. Thus $ \ker \varphi_x = \ker \varphi_y = \msr R_2$.
\end{proof}

Set
 \begin{equation}z_{n+1,0}=0,\;z_{n+1,i}=x_i,\;z_{j,0}=y_j\qquad\for\;\;i\in J_1,\;j \in J_3.
\label{a3.17} \end{equation}
Denote
\begin{equation}\label{a3.18}
J'_1=\{0\}\bigcup J_1,\quad J'_3=J_3\bigcup \{n+1\}.\end{equation}
Let
\begin{equation}\msr C=\mbb F[z_{j,i}\mid j\in J_3',\;i\in J_1',\;(j,i)\neq (n+1,0)]\end{equation}
be the algebra of polynomials in the variables $\{z_{j,i}\mid j\in J_3',\;i\in J_1',\;(j,i)\neq (n+1,0)\}.$
Denote by $\msr C_k$ the subspace of homogenous polynomials in $\msr C$ with degree $k$.
Define an associative algebra homomorphism $\phi: \msr C\rta\msr A$ by
\begin{equation}\phi(z_{j,i})=x_ix_j-y_iy_j,\;\;\phi(x_i)=x_i,\;\phi(y_j)=y_j\qquad i\in J_1,\;j\in J_3.\end{equation}

For $0<m\in\mbb N$, we set
\begin{equation}\wht J_1^m=\{\{i_1,i_2,...,i_m\}\mid i_s\in J_1\},\quad\wht J_3^m=\{\{j_1,j_2,...,j_m\}\mid i_t\in J_1\}.\end{equation}
Let $\{(j_1,i_1),(j_2,i_2),...,(j_r,i_r)\}\subset \mbb N^2$. If there exist three distinct $1\leq k_1,k_2,k_3\leq r$ such that
$i_{k_1}<i_{k_2}< i_{k_3}$ and $j_{k_1}<j_{k_2}<j_{k_3}$, we say that $\{(j_1,i_1),(j_2,i_2),...,(j_r,i_r)\}$ contains a 3-chain $\left( j_{k_1},i_{k_1}\right) \prec   \left( j_{k_2},i_{k_2}\right)\prec  \left( j_{k_3},i_{k_3}\right) $. For $k_1,k_2,k_3\in \mbb N$,
we define
\begin{equation}\msr I_1^x\left[\left(\prod_{t_1=1}^{k_1} x_{s_{1,t_1}}\right)  \left(\prod_{t_2=1}^{k_2} y_{s_{2,t_2}}\right)  \left( \prod_{t_3=1}^{k_3} z_{j_{t_3},i_{t_3}}\right)\right]=\{s_{1,1},...,s_{1,k_1}\},\label{a3.22}\end{equation}
\begin{equation}\msr I_1\left[\left(\prod_{t_1=1}^{k_1} x_{s_{1,t_1}}\right)  \left(\prod_{t_2=1}^{k_2} y_{s_{2,t_2}}\right)  \left( \prod_{t_3=1}^{k_3} z_{j_{t_3},i_{t_3}}\right)\right]=\{s_{1,1},...,s_{1,k_1},i_1,...,i_{k_3}\},\end{equation}
\begin{equation}\msr I_3^y\left[\left(\prod_{t_1=1}^{k_1} x_{s_{1,t_1}}\right)  \left(\prod_{t_2=1}^{k_2} y_{s_{2,t_2}}\right)  \left( \prod_{t_3=1}^{k_3} z_{j_{t_3},i_{t_3}}\right)\right]=\{s_{2,1},...,s_{2,k_2}\},\end{equation}
\begin{equation}\msr I_3\left[\left(\prod_{t_1=1}^{k_1} x_{s_{1,t_1}}\right)  \left(\prod_{t_2=1}^{k_2} y_{s_{2,t_2}}\right)  \left( \prod_{t_3=1}^{k_3} z_{j_{t_3},i_{t_3}}\right)\right]=\{s_{2,1},...,s_{2,k_2},j_1,...,j_{k_3}\},\end{equation}
\begin{eqnarray}& &\msr I\left[\left(\prod_{t_1=1}^{k_1} x_{s_{1,t_1}}\right)  \left(\prod_{t_2=1}^{k_2} y_{s_{2,t_2}}\right)  \left( \prod_{t_3=1}^{k_3} z_{j_{t_3},i_{t_3}}\right)\right]\nonumber\\ &=&
\{(n+1,s_{1,1}),...,(n+1,s_{1,k_1}),(j_1,i_1),...,(j_{k_3},i_{k_3}),(s_{2,1},0),...,(s_{2,k_2},0)\},\qquad\end{eqnarray}
where
\begin{equation}\{s_{1,1},...,s_{1,k_1},i_1,...,i_{k_3}\}\in \wht J_1^{k_1+k_3},\;\;\{s_{2,1},...,s_{2,k_2},j_1,...,j_{k_3}\}\in \wht J_3^{k_2+k_3}.
\end{equation}
In the following, we use the lexical order $\leq$ on $\mbb N^2$.
Fix ${\cal I}_1\in  \wht J_1^{k_1+k_3}$ and ${\cal I}_3\in \wht J_3^{k_2+k_3}$. we define
\begin{eqnarray}\label{a3.28}
G_{k_1,k_2,k_3}^{{\cal I}_1,{\cal I}_3}&=& \big\{g=\left(\prod_{t_1=1}^{k_1} x_{s_{1,t_1}}\right)  \left(\prod_{t_2=1}^{k_2} y_{s_{2,t_2}}\right)  \left( \prod_{t_3=1}^{k_3} z_{j_{t_3},i_{t_3}}\right)\mid
\msr I_1(g)={\cal I}_1,\;\nonumber\\& &\msr I_3(g)={\cal I}_3,\;\mbox{no 3-chains in}\;\msr I(g);\;s_{1,1}\leq\cdots\leq s_{1,k_1},\;\nonumber\\ & &s_{2,1}\leq\cdots \leq s_{1,k_2},\;(j_1,i_1)\leq ...\leq (j_{k_3},i_{k_3}) \big\}.\end{eqnarray}
Set
\begin{equation}\label{a3.29}V_{k_1,k_2,k_3}^{{\cal I}_1,{\cal I}_3}=\text{Span}\{G_{k_1,k_2,k_3}^{{\cal I}_1,{\cal I}_3}\}.\end{equation}
\pse

\begin{lemma}
\label{alem:3} The set $\phi(G_{k_1,k_2,k_3}^{{\cal I}_1,{\cal I}_3})$  is a linearly independent set in $\msr A$.
\end{lemma}

\begin{proof}
For convenience, we make the following convention:
\begin{equation}x_0=0,\quad y_{n+1}=0.\end{equation}
For $i_1,i_2\in J_1'$ with $i_1 \neq i_2$, we define an associative algebra endomorphism $\G^y_{i_1,i_2}$ of $\msr C$ by
\begin{equation}\label{a3.31}
\G^y_{i_1,i_2}(z_{j,i_1})=z_{j,i_2},\;\G^y_{i_1,i_2}(z_{j,i'})=z_{j,i'}\quad\for\;\;j\in J_3',\;i_1\neq i'\in J_1'.\end{equation}
In particular,
\begin{equation}\G^y_{i_1,0}(z_{j,i_1})=y_j,\quad\G^y_{i_1,0}(x_{i_1})=0\end{equation}
by (\ref{a3.17}) and (\ref{a3.31}).

For any ${\cal I}\in (\wht J_1^m\bigcup\wht J_3^m)$, we denote by
$\mfk c({\cal I})$ the number of distinct elements in ${\cal I}$. We want to prove  that $\phi(G_{k_1,k_2,k_3}^{{\cal I}_1,{\cal I}_3})$ is linearly dependent by induction on
\begin{equation}\ell_{k_1,k_2,k_3}^{{\cal I}_1,{\cal I}_3}=k_1+k_2+k_3+\mfk c({\cal I}_1)+\mfk c({\cal I}_3).\end{equation} First $\ell_{k_1,k_2,k_3}^{{\cal I}_1,{\cal I}_3} \geq 2$ by definition. When $\ell_{k_1,k_2,k_3}^{{\cal I}_1,{\cal I}_3}=2$, $G_{k_1,k_2,k_3}^{{\cal I}_1,{\cal I}_3}=\{x_i\}$ for some $i\in J_1$ or $\{y_j\}$ for some $j\in J_3$. The lemma trivially holds. Assume the lemma holds for $\ell_{k_1,k_2,k_3}^{{\cal I}_1,{\cal I}_3}<l\in\mbb N$. Consider $\ell_{k_1,k_2,k_3}^{{\cal I}_1,{\cal I}_3}=l$. Suppose that $\phi(G_{k_1,k_2,k_3}^{{\cal I}_1,{\cal I}_3})$ is linearly dependent. Then there exists $0\neq f\in V_{k_1,k_2,k_3}^{{\cal I}_1,{\cal I}_3}$ such that $\phi(f)=0$. Write
\begin{equation}\label{a3.34}f=\sum_{s=1}^ma_sf_s,\end{equation}
where $0\neq a_1,a_2,...,a_m\in\mbb F$ and $f_1,f_2,...,f_m\in G_{k_1,k_2,k_3}^{{\cal I}_1,{\cal I}_3}$ are distinct. If $\mbox{g.c.d}(f_1,...,f_m)=h\neq 1$, then $f_s=g_sh$ for $s\in\ol{1,m}$. We have $f=h(\sum_sa_sg_s)$ and $0\neq g=\sum_sa_sg_s\in V_{k_1',k_2',k_3'}^{{\cal I}_1',{\cal I}_3'}$ with $\ell_{k_1',k_2',k_3'}^{{\cal I}_1',{\cal I}_3'}<l$. By inductional assumption,
$\phi(G_{k_1',k_2',k_3'}^{{\cal I}_1',{\cal I}_3'})$ is linearly independent, which implies $\phi(g)\neq 0$. On the other hand,
\begin{equation}0=\phi(f)=\phi(h)\phi(g)\lra \phi(g)=0\end{equation}
because $\phi(h)\neq 0$ due to $h$ being a monomial. It is a contradiction. Thus
\begin{equation}\label{a3.36}\mbox{g.c.d}(f_1,...,f_m)=1.\end{equation}

If $k_3=0$, then $f=a\xi$ with $\xi\in\mbb F[X_{J_1},Y_{J_3}]$. Then $0=\phi(f)=f$, which leads to a contradiction. Hence $k_3>0$.
Write
\begin{equation}f=\sum_{s=1}^mh_sg_s\;\;\mbox{with}\;h_s\in\mbb F[X_{J_1},Y_{J_3}],\;g_s\in Z^{(k_3)}\end{equation}
(cf. the below of (\ref{a3.2})).
Denote
\begin{equation}i_0=\min\{\msr I_1(g_s)\mid s\in\ol{1,m}\}.\end{equation}
Due to (\ref{a3.36}), there exists some $\G^y_{i_0,0}(f_s)\neq 0$. For any $\G^y_{i_0,0}(f_t)\neq 0$, we have $i_0\not\in \msr I_1(h_t)$.
Write
\begin{equation}\G^y_{i_0,0}(f_t)=h_t'g_t'\;\;\mbox{with}\;h'_t\in\mbb F[X_{J_1},Y_{J_3}],\;g'_t\in Z^{(k_3)}.\end{equation}

\begin{figure}[h]
\centering
\includegraphics[width=0.4\textwidth]{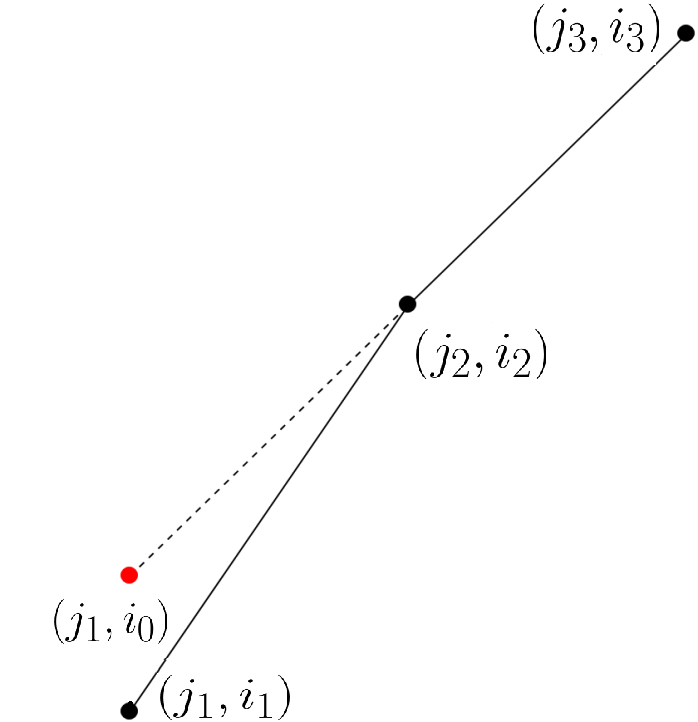}
\caption{\  The case $i_1=i_1'$}
\label{afig:1}
\end{figure}
There exist ${\cal I}^\ast_1\in \wht J_1^{k_1+k_3-\iota} $, ${\cal I}^\ast_3\in \wht J_3^{k_2+k_3} $ and $0<\iota\in\mbb Z$ such that
\begin{equation}\msr I_1(\G^y_{i_0,0}(f_t))={\cal I}^\ast_1,\;\msr I_3(\G^y_{i_0,0}(f_t))={\cal I}^\ast_3,\;|\msr I_1(g'_t)|=k_3-\iota.\end{equation}
If $\msr I(h_t'g_t')$ contains a 3-chain $(j_1,i_1) \prec   \left( j_2,i_2\right)\prec  \left( j_3,i_3\right) $, then $i_1=0$. So $\msr I(f_t)$ contains 3-chain $(j_1,i_0) \prec   \left( j_2,i_2\right)\prec  \left( j_3,i_3\right) $ (cf. Fig.1), which contradicts (\ref{a3.28}).

Thus $\msr I(h_t'g_t')$ does not contain 3-chains. Therefore
\begin{equation}\G^y_{i_0,0}(f)\in V_{k_1,k_2+\iota,k_3-\iota}^{{\cal I}_1^\ast,{\cal I}_3^\ast}.\end{equation}
Moreover,
\begin{equation}\mfk c({\cal I}_1^\ast)=\mfk c({\cal I}_1)-1,\;\mfk c({\cal I}_3^\ast)=\mfk c({\cal I}_3).\end{equation}
Hence
\begin{equation}\ell_{k_1,k_2+\iota,k_3-\iota}^{{\cal I}^\ast_1,{\cal I}^\ast_3}=l-1.\end{equation}
By inductional assumption, $\phi(G_{k_1,k_2+\iota,k_3-\iota}^{{\cal I}_1^\ast,{\cal I}_3^\ast})$ is linearly independent.
Note
\begin{equation}\phi(\G^y_{i_0,0}(f))=\phi(f)|_{x_{i_0}=0,y_{i_0}=1}=0.\end{equation}
So
\begin{equation}\label{a3.45}\G^y_{i_0,0}(f)= 0.\end{equation}

Take $i_1,i_2\in {\cal I}_1$ such that $i_1<i_2$ and there is no $i'\in  {\cal I}_1$ satisfying $i_1<i'<i_2$. Then
\begin{equation}\label{a3.46}
\G^y_{i_2,i_1}(f_s)\neq 0\qquad\for\;\;s\in\ol{1,m}.\end{equation}
Write
\begin{equation}\G^y_{i_2,i_1}(f_s)=\xi_s\eta_s\;\;\mbox{with}\;\xi_s \in\mbb F[X_{J_1},Y_{J_3}],\;\eta_s\in Z^{k_3}.\end{equation}
If $\msr I(\xi_s\eta_s)$ contains a 3-chain $(j_1',i_1') \prec   \left( j_2',i_2'\right)\prec  \left( j_3',i_3'\right)$, then $i_1\in\{i_1',i_2',i_3'\}$. When $i_1=i_1'$, $\msr I(f_s)$ contains the 3-chain $(j_1',i_2) \prec   \left( j_2',i_2'\right)\prec  \left( j_3',i_3'\right)$. In Case $i_1=i_2'$, $\msr I(f_s)$ contains the 3-chain $(j_1',i_1') \prec   \left( j_2',i_2\right)\prec  \left( j_3',i_3'\right)$ (cf. Fig.2(a)). If $i_1=i_3'$, $\msr I(f_s)$ contains the 3-chain $(j_1',i_1') \prec   \left( j_2',i_2'\right)\prec  \left( j_3',i_2\right)$ (cf. Fig.2(b)). These contradict (\ref{a3.28}). Thus $\msr I(\xi_s\eta_s)$ does not contain a 3-chain. There exist $0<\iota' \in \mathbb{Z}$ and ${\cal I}^\sharp_1\in \wht J_1^{k_1+k_3-\iota'}$ such that
\begin{equation}\G^y_{i_2,i_1}(f)\in V_{k_1,k_2,k_3}^{{\cal I}_1^\sharp,{\cal I}_3}.\end{equation}
\begin{figure}[h]
  \centering
 \subfigure[]{\includegraphics[width=0.4\textwidth]{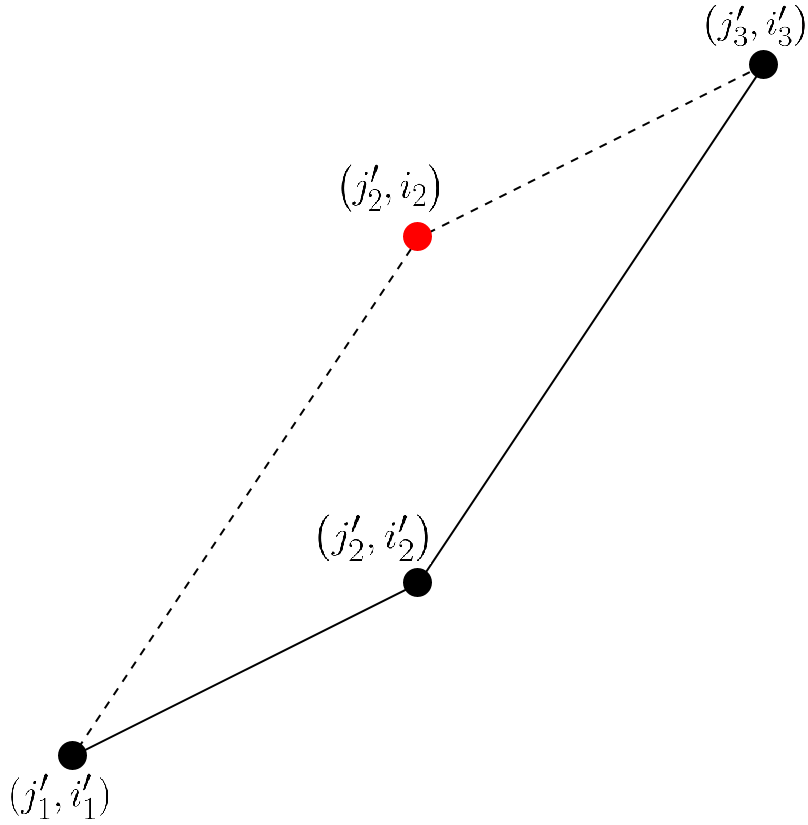}}
  \hfill
\subfigure[]{\includegraphics[width=0.4\textwidth]{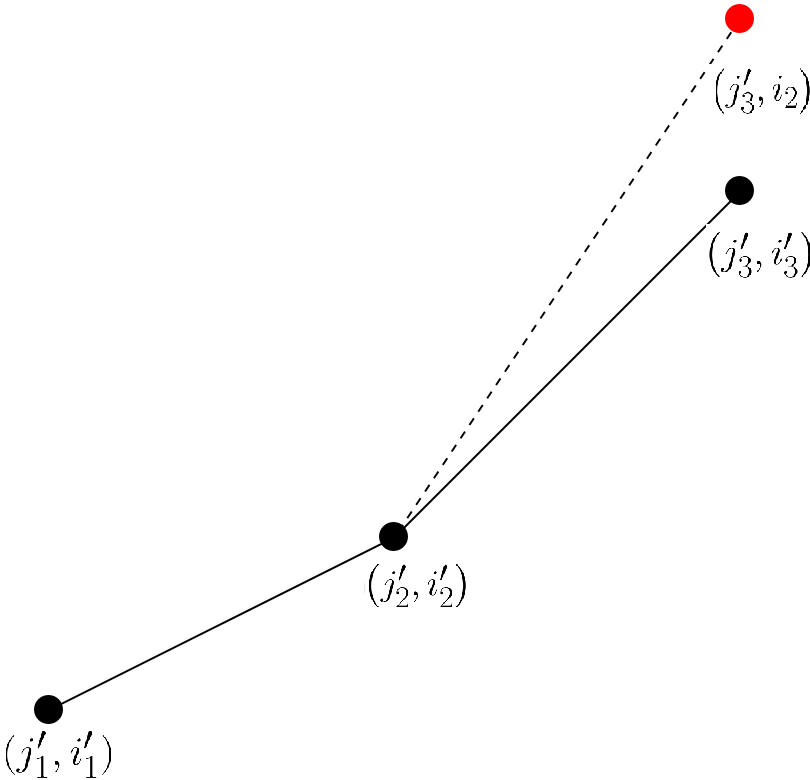}}
  \caption{\ The cases $i_1=i_2'$ and $i_1=i_3'$}
  \label{afig:2}
\end{figure}

In this case,
\begin{equation}\mfk c({\cal I}_1^\sharp)=\mfk c({\cal I}_1)-1.\end{equation}
So
\begin{equation}\ell_{k_1,k_2,k_3}^{{\cal I}_1^\sharp,{\cal I}_3}=l-1.\end{equation}
By inductional assumption, $\phi(G_{k_1,k_2,k_3}^{{\cal I}_1^\sharp,{\cal I}_3})$ is linearly independent.
Note
\begin{equation}\phi(\G^y_{i_2,i_1}(f))=\phi(f)|_{x_{i_2}=x_{i_1},y_{i_2}=y_{i_1}}=0.\end{equation}
This yields
\begin{equation}\label{a3.52}
\G^y_{i_2,i_1}(f)=0.\end{equation}

For $j_1,j_2\in J_3'$ with $j_1>j_2$, we define an associative algebra endomorphism $\G^x_{j_2,j_1}$ of $\msr C $ by
\begin{equation}\label{a3.53}
\G^x_{j_2,j_1}(z_{j_2,i})=z_{j_1,i},\quad \G^x_{j_2,j_1}(z_{j',i})=z_{j',i}\qquad\for\;\;j_2\neq j'\in J_3',\;i\in J_1'.\end{equation}
In particular,
\begin{equation}\G^x_{j_2,n+1}(z_{j_2,i})=x_i,\quad\G^x_{j_2,n+1}(y_{j_2})=0\end{equation}
by (\ref{a3.17}) and (\ref{a3.53}). Set
\begin{equation}j_0=\max\{\msr I_3(g_s)\mid s\in\ol{1,m}\}.\end{equation}
Take $j_1,j_2\in {\cal I}_3$ such that $j_1>j_2$ and there is no $j'\in {\cal I}_3$ satisfying $j_2<j'<j_1$.
Symmetrically, we have
\begin{equation}\label{a3.56}
\G^x_{j_0,n+1}(f)= 0\end{equation}and
\begin{equation}\label{a3.57}
\G^x_{j_2,j_1}(f)= 0.\end{equation}

Suppose that the set of distinct elements of ${\cal I}_1$ is
\begin{equation}\Upsilon_1=\{s_1,s_2,...,s_{\ell_1}\}\quad\mbox{with}\;\;s_1<s_2<\cdots<s_{\ell_1}.\end{equation}
By (\ref{a3.45}) and (\ref{a3.52}),
\begin{equation}\G_{s_1,0}^y(f)=\G_{s_2,s_1}^y(f)=\cdots=\G_{s_{\ell_1},s_{\ell_1-1}}^y(f)=0.\end{equation}
Write  the set of distinct elements of ${\cal I}_3$ as
\begin{equation}\Upsilon_3=\{t_1,t_2,...,t_{\ell_3}\}\quad\mbox{with}\;\;t_1<t_2<\cdots<t_{\ell_3}.\end{equation}
According to (\ref{a3.56}) and (\ref{a3.57}),
\begin{equation}\label{a3.61}\G_{t_1,t_2}^x(f)=\G_{t_2,t_3}^x(f)=\cdots=\G_{t_{\ell_3-1},t_{\ell_3}}^x(f)=\G_{t_{\ell_3},n+1}^x(f)=0.\end{equation}

If $k_1=0$ and $\G_{t_{\ell_3},n+1}^x(f)=0$, this implies $f=0$, which contradicts $f\neq 0$. Thus $k_1>0$. Symmetrically, $k_2>0$.
Suppose that $\msr I(f_\iota)$ does not contain any 2-chain $(j,0)\prec (j',s_1)$ with $j,j'\in {\cal I}_3$. Then
\begin{equation}j \geq j'\qquad \for\;\;j\in \msr I_3^y(f_\iota),\;(j',s_1)\in \msr I(f_\iota).\end{equation}
If $y_{t_{\ell_3}}\not|f_{\iota}$, there exists $(t_{\ell_3},i)\in\msr I(f_{\iota})$ for some $s_1<i\in{\cal I}_1$; that is $z_{t_{\ell_3},i}|f_{\iota}$. Since $\msr I(f_{\iota})$ does not contain 3-chain $(j_1,0)\prec (t_{\ell_3},i)\prec (n+1,i')$, we have
\begin{equation}i'\leq i\qquad\for\;\;i'\in \msr I^x_1(f_{\iota}).\end{equation} By $\G^x_{t_{\ell_3},n+1}(f)=0$, there exists another $f_{\iota'}$ such that $x_iz_{t_{\ell_3},i'}|f_{\iota'}$ with $i'<i$ and
\begin{equation}\msr I_3^y(f_\iota)=\msr I_3^y(f_{\iota'}).\end{equation}
Now $\msr I(f_{\iota'})$ contains a 3-chain $(j_1,0)\prec (t_{\ell_3},i')\prec (n+1,i)$, which contradicts (\ref{a3.28}).
Thus  $y_{t_{\ell_3}}|f_\iota$.

Assume $\msr I(f_{s'})$ contains a 2-chain $(j,0)\prec (j',s_1)$ with $j,j'\in {\cal I}_3$,
$y_{t_{\ell_3}}\not|f_{s'}$ and $z_{t_{\ell_3},i}|f_{\iota}$ for some $s_1<i\in{\cal I}_1$.
  Since $\msr I(f_{s'})$ does not contain 3-chain $(j_1,0)\prec (t_{\ell_3},i)\prec (n+1,i')$, we have
\begin{equation}i'\leq i\qquad\for\;\;i'\in \msr I^x_1(f_{s'}).\end{equation} By $\G^x_{t_{\ell_3},n+1}(f)=0$, there exists another $f_{{s'}'}$ such that $x_iz_{t_{\ell_3},i'}|f_{{s'}'}$ with $i'<i$ and
\begin{equation}\msr I_3^y(f_{s'})=\msr I_3^y(f_{{s'}'}).\end{equation}
Now $\msr I(f_{{s'}'})$ contains a 3-chain $(j_1,0)\prec (t_{\ell_3},i')\prec (n+1,i)$, which contradicts (\ref{a3.28}) (cf. Fig.3).
Thus $y_{t_{\ell_3}}|f_{s'}$.
\begin{figure}[h]
\centering
\includegraphics[width=0.5\textwidth]{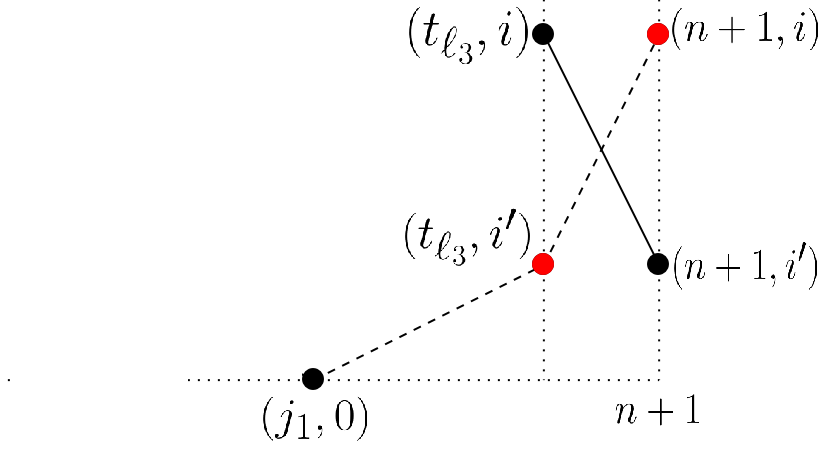}
\caption{\ If $y_{t_{\ell_3}}\not|f_{\iota}$}
\label{afig:4}
\end{figure}

Assume $\msr I(f_{s'})$ contains a 2-chain $(j,0)\prec (j',s_1)$ with $j,j'\in {\cal I}_3$,
$y_{t_{\ell_3}}\not|f_{s'}$ and $z_{t_{\ell_3},i}|f_{s'}$ forces $i=s_1$. If there exists $s_1<i'\in\msr I^x_1(f_{s'})$,
$\msr I(f_{s'})$ contains a 3-chain $(j,0)\prec (j',s_1)\prec (n+1,i')$, which contradicts (\ref{a3.28}). Therefore,
$\msr I^x_1(f_{s'})=\{s_1,s_1,...,s_1\}$, which yields $\msr I^x_1(\G_{t_{\ell_3},n+1}^x(f_{s'})) =\{s_1,s_1,...,s_1\}$. This leads $\G_{t_{\ell_3},n+1}^x(f)\neq 0$, and contradicts (\ref{a3.61}). So $y_{t_{\ell_3}}\mid f_{s'}$. This leads a contradiction to (\ref{a3.36}).
Then there does not exists $0\neq f\in V_{k_1,k_2,k_3}^{{\cal I}_1,{\cal I}_3}$ such that $\phi(f)=0$, that is, $\phi(G_{k_1,k_2,k_3}^{{\cal I}_1,{\cal I}_3})$ is linearly independent. \end{proof}

Let
\begin{equation}\label{a3.67}
\msr R_3=\left\la \left|\begin{array}{ccc}
z_{j_1,i_1}  & z_{j_1,i_2}& z_{j_1,i_3}\\
z_{j_2,i_1} & z_{j_2,i_2}& z_{j_2,i_3}\\
z_{j_3,i_1} & z_{j_3,i_2}& z_{j_3,i_3}
\end{array}\right|\;\mid\;i_1,i_2,i_3\in J_1',\;j_1,j_2,j_3\in J_3'\right\ra\end{equation}
be an ideal of $\msr C$.

\begin{lemma}
\label{alem:3.3}
For $k\in\mbb N$, we have $\ker \phi =\msr R_3$.\end{lemma}
\begin{proof} We take $\leq$ be the lexical order on $J_3'\times J_1'$ (cf. (\ref{aa1.13}) and (\ref{a3.18})). For
\begin{equation}S_1=\{(j_1,i_1),...,(j_k,i_k)\},\quad S_2=\{(j_1',i_1'),...,(j_k',i_k')\}\end{equation}
with $(j_s,i_s),(j_s',i_s')\in J_3'\times J_1'$.
Since the notion of sets $\{...\}$ is independent of the order of its elements, we may assume
\begin{equation}(j_1,i_1)\leq\cdots \leq (j_k,i_k),\qquad (j_1',i_1')\leq\cdots \leq (j_k',i_k').\end{equation}
Take the lexical order $\leq$ on $(J_3'\times J_1')^k$ and define $S_1\leq S_2$ if
\begin{equation}((j_1,i_1),\dots,(j_k,i_k))\leq ((j_1',i_1'),\dots,(j_k',i_k'))\;\;\mbox{in}\;\;(J_3'\times J_1')^k.\end{equation}
Assume the above  $S_1\leq S_2$ and
\begin{equation}T_1=\{(m_1,\ell_1),...,(m_\iota,\ell_\iota)\},\quad T_2=\{(m_1',\ell_1'),...,(m_\iota',\ell_\iota')\}\end{equation}
with $(m_t,\ell_t),(m_t',\ell_t')\in J_3'\times J_1'$
also satisfying $T_1\leq T_2$. Then
\begin{equation}\{S_1,T_1\}\leq \{S_2,T_2\}.\end{equation}

Set
\begin{eqnarray}{\wht G}_{k_1,k_2,k_3}^{{\cal I}_1,{\cal I}_3}&=&\big\{g=\left(\prod_{t_1=1}^{k_1} x_{s_{1,t_1}}\right)\left(\prod_{t_2=1}^{k_2} y_{s_{2,t_2}}\right)  \left( \prod_{t_3=1}^{k_3} z_{j_{t_3},i_{t_3}}\right)\nonumber\\& &\mid
\msr I_1(g)={\cal I}_1,\;\msr I_3(g)={\cal I}_3;\;s_{1,1}\leq\cdots\leq s_{1,k_1},\;\nonumber\\ & &s_{2,1}\leq\cdots \leq s_{1,k_12},\;(j_1,i_1)\leq ...\leq (j_{k_3},i_{k_3})\big\}.\end{eqnarray}
 Let $f\in\ker \phi\cap \msr C_k$ for some $k\in\mbb N$.
We may assume
\begin{equation}f=\sum_{s=1}^ma_sf_s\qquad\mbox{with}\;\;a_s\in\mbb F,\;f_s\in {\wht G}_{k_1,k_2,k_3}^{{\cal I}_1,{\cal I}_3}\end{equation}
for some fixed $k_1,k_2,k_3$ satisfying $k_1+k_2+k_3=k$, and ${\cal I}_1\in  \wht J_1^{k_1+k_3},\;{\cal I}_3\in \wht J_3^{k_2+k_3}$.

Suppose some $\msr I(f_t)$ contains a 3-chain $(j_1,i_1)\prec (j_2,i_2)\prec (j_3,i_3)$. Write $f_t = z_{j_1,i_1}z_{j_2,i_2}z_{j_3,i_3} g_t$ with $g_t\in\msr C_{k-3}$. Then
\begin{eqnarray} f_t &=& z_{j_1,i_1}z_{j_2,i_2}z_{j_3,i_3} g_t\equiv (z_{j_1,i_2}z_{j_2,i_1}z_{j_3,i_3}
+z_{j_1,i_3}z_{j_2,i_2}z_{j_3,i_1}\nonumber\\ & &+z_{j_1,i_1}z_{j_2,i_3}z_{j_3,i_2} -z_{j_1,i_2}z_{j_2,i_3}z_{j_3,i_1}-z_{j_1,i_3}z_{j_2,i_1}z_{j_3,i_2})g_t\;\;(\mbox{mod}\;\msr R_3).\end{eqnarray}
Note
\begin{eqnarray}\msr I(f_t)&=&\{(j_1,i_1),(j_2,i_2),(j_3,i_3), \msr I(g_t)\}\nonumber\\&<&\{(j_1,i_2),(j_2,i_1),(j_3,i_3), \msr I(g_t)\}=\msr I(z_{j_1,i_2}z_{j_2,i_1}z_{j_3,i_3}g_t),\end{eqnarray}
\begin{eqnarray}\msr I(f_t)&=&\{(j_1,i_1),(j_2,i_2),(j_3,i_3), \msr I(g_t)\}\nonumber\\&<&\{(j_1,i_3),(j_2,i_2),(j_3,i_1), \msr I(g_t)\}=\msr I(z_{j_1,i_3}z_{j_2,i_2}z_{j_3,i_1}g_t),\end{eqnarray}
\begin{eqnarray}\msr I(f_t)&=&\{(j_1,i_1),(j_2,i_2),(j_3,i_3), \msr I(g_t)\}\nonumber\\&<&\{(j_1,i_1),(j_2,i_3),(j_3,i_2), \msr I(g_t)\}=\msr I(z_{j_1,i_1}z_{j_2,i_3}z_{j_3,i_2}g_t),\end{eqnarray}
\begin{eqnarray}\msr I(f_t)&=&\{(j_1,i_1),(j_2,i_2),(j_3,i_3), \msr I(g_t)\}\nonumber\\&<&\{(j_1,i_2),(j_2,i_3),(j_3,i_1), \msr I(g_t)\}=\msr I(z_{j_1,i_2}z_{j_2,i_3}z_{j_3,i_1}g_t),\end{eqnarray}
\begin{eqnarray}\msr I(f_t)&=&\{(j_1,i_1),(j_2,i_2),(j_3,i_3), \msr I(g_t)\}\nonumber\\&<&\{(j_1,i_3),(j_2,i_1),(j_3,i_2), \msr I(g_t)\}=\msr I(z_{j_1,i_3}z_{j_2,i_1}z_{j_3,i_2}g_t).\end{eqnarray}
Continue this process for each $f_s$, we can prove that there exists $f'\in V_{k_1,k_2,k_3}^{{\cal I}_1,{\cal I}_3}$ (cf. (\ref{a3.29})) such that
\begin{equation}f\equiv f'\;\;(\mbox{mod}\;\msr R_3).\end{equation}
It is easy to verify $\phi(\msr R_3)=\{0\}$ (cf. (\ref{a3.67})). Thus $0=\phi(f)=\phi(f').$ By Lemma 3.2, $f'=0$; that is $f\in \msr R_3$.
Since $\ker\phi$ is a graded subspace of $\msr C$, we have $\ker \phi=\msr R_3$.

\end{proof}

\section{Proof of the Main Theorem}
In this section, we prove our main theorem. We define a notion {\it ord} to single out the powers of $P$ in $M_k$, and find the ideals of the associated varieties by the linear relationship we found in the last section.

For $f \in M_k=V_k$ (cf. (\ref{a2.60})), we define the P-{\it order} of $f$ in $M_k$ as
\begin{align}\label{a4.1}
\mbox{ord}_k \left( f \right)
=\min \left\lbrace
s \in \mathbb{N} | f \in \mbox{Span}\{ TN_k, TN(k-r) P^r |
r=0,1,2,\cdots , s \}
\right\rbrace  .
\end{align}
In particular, $\mbox{ord}_k \left(f \right)=0  \iff  f\in TN_k$.

There are some properties of P-{\it order}:
\begin{enumerate}
    \item If $f_i \in M_k$, then
    \begin{align}
        \mbox{ord}_k\left( \sum_{i=1}^m f_i \right) \leq \max \left\lbrace \mathrm{ord}_k \left( f_i \right) \mid i \in \overline{1,m} \right\rbrace.
    \end{align}

    \item If $f \in TN_{k-r}P^r$, then
    \begin{align}
        \mbox{ord}_k\left(f \right) \leq r.
    \end{align}

    \item If $f \in M_{k_1}$ and $k_1 \leq k_2$, then
    \begin{align}
        \mbox{ord}_{k_1}\left(f \right) \geq \mathrm{ord}_{k_2}\left(f \right).
        \label{a4.6}
    \end{align}
\end{enumerate}

\begin{lemma}
\label{alem:4}
For $f \in M_{k}$, we have
\begin{align}
{\mfk d} \left( f \right)  \leq k+ \mbox{ord}_k\left(f \right) .
\end{align}
Moreover, ${\mfk d} \left( f \right) = k+ \mbox{ord}_k\left(f \right) $ if and only if $f \notin TN_{k-1}$.
\end{lemma}

\begin{proof} For
$f \in V_{k} $ with
$\mbox{ord}_k\left(f \right)=r$, we write
\begin{align}\label{b4.6}
f=\sum_{s_r=1}^{\ell_r} g_{r,s_1} f_{r,s_r} +  \sum_{s_{r-1}=1}^{\ell_{r-1}} g_{r-1,s_{r-1}} f_{r-1,s_{r-1}} + \cdots +f_{0},
\end{align}
where $ g_{i,s_i}  \in P^{\left(i \right)}$,
$f_{i,s_i}  \in TN(k-i)$ for $  s_i\in\ol{1,\ell_i} $, and $f_{0} \in TN_k$.  Then
\begin{eqnarray}\label{a4.9}
{\mfk d} \left( f \right) & =&\max \left\lbrace {\mfk d} \left( g_{i,s_i} f_{i,s_i}  \right),f_0\mid 1 \leq i \leq r,\;s_i \in \ol{1,\ell_i}  \right\rbrace\nonumber\\
&=&\max \left\lbrace {\mfk d} \left( f_{i,s_i}  \right)+2i,  {\mfk d} (f_0)\mid 1 \leq i \leq r,\;s_i \in \ol{1,\ell_i}  \right\rbrace\nonumber \\
&\leq &\max\left\lbrace k+i\mid i\in\ol{0,r} \right\rbrace =k+r.
\end{eqnarray}
In other words,
\begin{align}
{\mfk d}\left(f \right)  \leq k+ \mbox{ord}_k \left(f \right) .
\end{align}
Note that in Equation (\ref{a4.9}), equality holds if and only if ${\mfk d} (f_0)=k$. Thus ${\mfk d} \left( f \right) = k+ \mbox{ord}_k\left(f \right) $ if and only if $f \notin TN_{k-1}$.

\end{proof}

Fix $1\leq p\in\mbb N$, we recall $\mfk g^p\subset U(\mfk g)$ with $\mfk g=sl(n)$ and $M_k=U_k(\mfk g)(M_0)$ (cf. (\ref{a2.11})).
Define
\begin{equation}\label{a4.12}
\mfk k_{(p)}=\{\xi\in\mfk g^p\mid \xi(M_k)\subset M_{k+p-1},\;k\in\mbb N\}.\end{equation}
Recall the associated graded $\mfk g$-module $\ol{M}$ in (\ref{a1.10}) and
the annihilator of $\ol{M}$ in (\ref{a1.11}) is a graded ideal in $S( \mfk {g})$. Define $S_{(p)}(\mfk g)=U_p(\mfk g)/U_{p-1}(\mfk g)$ and
\begin{align}\label{a4.13}
I_{(p)}=\left\lbrace \eta \in S_{(p)} \left( \mathfrak{g} \right)  | \eta(v)=0\ \for\  v \in \ol{M} \right\rbrace.
\end{align}
Then $\operatorname{Ann}_{S( \mfk g)} (\ol{M})=\sum_{p=1}^\infty I_{(p)}$ and we are going to compute $I_{(p)}$ for different values of $p$.
First, we assume $|J_1|,|J_2|,|J_3|\geq 2$. Denote
\begin{eqnarray}
\label{a4.14}
 \Delta_{i_1,i_2,i_3}^{j_1,j_2,j_3}
 &=&   E_{j_1,i_1} E_{j_2,i_2} E_{j_3,i_3}+E_{j_1,i_2} E_{j_2,i_3} E_{j_3,i_1}+E_{j_1,i_3} E_{j_2,i_1} E_{j_3,i_2}\nonumber\\
 & &-E_{j_1,i_2} E_{j_2,i_1} E_{j_3,i_3}-E_{j_1,i_3}E_{j_2,i_2} E_{j_3,i_1}-E_{j_1,i_1} E_{j_2,i_3} E_{j_3,i_2},\end{eqnarray}
\begin{align}\begin{split}\label{a4.15}
\Delta_{i_1,i_2}^{j_1,j_2}=E_{j_1,i_1} E_{j_2,i_2}-E_{j_1,i_2} E_{j_2,i_1}\end{split}\end{align}
for $i_1,i_2,i_3,j_1,j_2,j_3\in\ol{1,n}$.

 According to (\ref{a2.1})-(\ref{a2.3}), we have
 \begin{equation}\td{\pi}(E_{j_1,i_1})=-x_{i_1}x_{j_1}-y_{i_1}\ptl_{y_{j_1}},
 \;\td{\pi}(E_{j_2,i_2})=x_{j_2}\ptl_{x_{i_2}}+y_{i_2}y_{j_2}\end{equation}
 \begin{equation}
 \td{\pi}(E_{j_3,i_3})=-x_{i_3}x_{j_3}+y_{i_3}y_{j_3}\qquad\for\;\;(j_s,i_s)\in L_s.\end{equation}
 (cf. (\ref{a2.51})). Thus
\begin{equation}E_{j,i}(M_0) \not\subset M_0,\;E_{j',i'}(M_k) \subset M_k\qquad\mbox{when}\;\;(j,i)\in L,\;(j',i')\not\in L,\;k\in\mbb N\end{equation}
 and
\begin{align}\label{a4.19}
I_{(1)}=\mbox{Span}\{E_{j,i}  |  \left(j,i \right)  \notin L\}.
\end{align}

\begin{lemma}\label{alem:4.2} If $\ell_1,\ell_2\leq 0$,
\begin{equation}\label{a4.20}
I_{(2)}\equiv \left\{ \left. \Delta^{j_1,j_1'}_{i_1,i_l'}  ,\Delta^{j_2,j_2'}_{i_2 ,i_2'} \right\vert (j_1,i_1), (j_1',i_1')\in L_1, (j_2,i_2), (j_2',i_2')\in L_2\right\}\quad (\mbox{mod}\; \la I_{(1)}\ra).
\end{equation}
If $\ell_1\leq 0$ and $\ell_2>0$,
\begin{equation}\label{a4.21}
I_{(2)} \equiv \left\{ \left. \Delta^{j_2,j_2'}_{i_2 ,i_2'}  \right\vert    (j_2,i_2), (j_2',i_2')\in L_2\right\}\quad (\mbox{mod}\; \la I_{(1)}\ra),
\end{equation}
and
\begin{equation} \label{b4.19}
\left\{\left. \left(\Delta^{j_1,j_1 '}_{i_1,i_l'} \right)^{\ell_2+1} \right\vert  (j_1,i_1), (j_1',i_1')\in L_1 \right\}  \in I_{(2\ell_2+2)} .
\end{equation}
In the case $\ell_1>0$ and $\ell_2\leq 0$,
\begin{equation}\label{a4.22}
I_{(2)} \equiv \left\{\left. \Delta^{j_1,j_1 '}_{i_1,i_l'} \right\vert  (j_1,i_1), (j_1',i_1')\in L_1 \right\}\quad (\mbox{mod}\; \la I_{(1)}\ra), 
\end{equation}
and
\begin{equation} \label{b4.21}
 \left\{ \left. \left(\Delta^{j_2,j_2'}_{i_2 ,i_2'} \right)^{\ell_1+1}  \right\vert    (j_2,i_2), (j_2',i_2')\in L_2\right\} \in I_{(2\ell_1+2)} .
\end{equation}
\end{lemma}

\begin{proof}
 Suppose
\begin{align}
\xi=  \sum_{(j,i),(j',i')\in L}a_{j,i;j',i'}E_{j,i}E_{j',i'}\in \mfk k_{(2)}
\label{a4.23}
\end{align}
and set
\begin{equation}\label{a4.24}v_0=x_1^{\max\{0,-\ell_1\}} y_n^{\max\{0,-\ell_2\}}x_{n_1+1}^{\max\{\ell_1,0\}}y_{n_1+1}^{\max\{\ell_2,0\}} \in M_0. \end{equation}
Then $\xi(v_0) \in M_1$ by (\ref{a4.12}).
Note that
\begin{equation}\label{a4.25}
\mfk d(x_{i_1}x_{j_1})=\mfk d(y_{i_2}y_{j_2})=1,\;\mfk d(x_{i_3}x_{j_3})=\mfk d(y_{i_3}y_{j_3})=2\end{equation}
as multiplication operators on $M$ for $(j_s,i_s)\in L_s$ (cf. (\ref{a2.46})-(\ref{a2.48})).
Hence
\begin{equation}\mbox{the coefficient of}\; x_{i_1}x_{i_1'}x_{j_1}x_{j_1'}v_0=a_{j_1,i_1;j_1',i_1'}+a_{j_1',i_1';j_1,i_2}+
a_{j_1',i_1;j_1,i_1'}+a_{j_1,i_1';j'_1,i_1}=0,\end{equation}
\begin{equation}\label{a4.27}
\mbox{the coefficient of}\; y_{i_2}y_{i_2'}y_{j_2}y_{j_2'}v_0=a_{j_2,i_2;j_2',i_2'}+a_{j_2',i_2';j_2,i_2}+
a_{j_2',i_2;j_2,i_2'}+a_{j_2,i_2';j'_2,i_2}=0,\end{equation}
\begin{equation}\mbox{the coefficient of}\; x_{i_3}x_{j_3}y_{i_3'}y_{j_3'}v_0=a_{j_3,i_3;j_3',i_3'}+a_{j_3',i_3';j_3,i_3}=0,\end{equation}
\begin{equation}\mbox{the coefficient of}\; x_{i_1}x_{j_1}y_{i_3}y_{j_3}v_0=a_{j_1,i_1;j_3,i_3}+a_{j_3,i_3;j_1,i_1}=0,\end{equation}
\begin{equation}\mbox{the coefficient of}\; x_{i_1}x_{j_1}y_{i_2}y_{j_2}v_0=a_{j_1,i_1;j_2,i_2}+a_{j_2,i_2;j_1,i_1}=0,\end{equation}
\begin{equation}\label{a4.31}
\mbox{the coefficient of}\; y_{i_2}y_{j_2}y_{i_3}y_{j_3}v_0=a_{j_2,i_2;j_3,i_3}+a_{j_3,i_3;j_2,i_2}=0\end{equation}
for $(j_s,i_s)\in L_s$ with $s=1,2,3.$ On the other hand,
\begin{equation}\label{b4.29}
E_{j,i}E_{j',i'}-E_{j,i}E_{j',i'}\in U_1(\mfk g)=\mfk g.\end{equation}
By Lemma 3.1,
\begin{eqnarray}\label{a4.33}
\xi& \in &\sum_{(j_1,i_1),(j_1',i_1')\in L_1;(j_1,i_1)\leq (j_1',i_1')}a_{j_1,i_1;j_1',i_1'}(E_{j_1,i_1}E_{j_1',i_1'}-
E_{j_1,i_1'}E_{j_1',i_1})\nonumber\\& &+\sum_{(j_2,i_2),(j_2',i_2')\in L_2;(j_2,i_2)\leq (j_2',i_2')}a_{j_2,i_2;j_2',i_2'}
(E_{j_2,i_2}E_{j_2',i_2'}-E_{j_2,i_2'}E_{j_2',i_2}) +U_1(\mfk g).\end{eqnarray}

Let $v_i=T(x^\al y^\be)\in M_i$ with ${\mfk d}(T(x^\al y^\be))={\mfk d}(x^\al y^\be) =i$ (cf. (\ref{a2.50})).
If $j_1=j_1'$, we naturally have
\begin{equation}E_{j_1,i_1}E_{j_1',i_1'}-
E_{j_1,i_1'}E_{j_1',i_1}=0.\end{equation}
So we assume $j_1\neq j_1'$. Suppose that there exists $j_3\in J_3$ such that $\be_{j_3}>0$.
According to (\ref{a1.7}),
\begin{equation}\label{a4.35}
 [ \td{\Dlt} , x_{j_1'}\ptl_{y_{j_1}}-x_{j_1}\ptl_{y_{j_1'}}]= \ptl_{y_{j_1}}\ptl_{y_{j_1'}}-\ptl_{y_{j_1'}}\ptl_{y_{j_1}}=0.\end{equation}

Assume $\be \neq 0$. Since
\begin{equation}x_{i_1}T=Tx_{i_1},\;\;y_{j_3}T=Ty_{j_3}\qquad\for\;\;i_1\in J_1, j_3\in J_3\end{equation}
by (\ref{a1.7}) and (\ref{a2.6}), we have
\begin{eqnarray}\label{b4.36}
\Delta_{i_1,i_1'}^{j_1,j_1'}(v_i)&=&(E_{j_1,i_1}E_{j_1',i_1'}-
E_{j_1,i_1' }E_{j_1',i_1})(v_i)
\nonumber\\&=&
 (x_{i_1}x_{j_1}+y_{i_1}\ptl_{y_{j_1}})(x_{i_1'}x_{j_1'}+y_{i_1'}\ptl_{y_{j_1'}})
-(x_{i_1'}x_{j_1}+y_{i'_1}\ptl_{y_{j_1}})(x_{i_1}x_{j_1'}+y_{i_1}\ptl_{y_{j_1'}}) (v_i)
\nonumber\\&=&(x_{i_1'}y_{i_1}-x_{i_1}y_{i_1'})(x_{j_1'}\ptl_{y_{j_1}}-x_{j_1}\ptl_{y_{j_1'}})(T(x^\al y^\be))\nonumber\\&=&(x_{i_1'}y_{i_1}y_{j_3}-x_{i_1}y_{i_1'}y_{j_3})(x_{j_1'}\ptl_{y_{j_1}}-x_{j_1}\ptl_{y_{j_1'}})( T(x^\al y^{\be-\es_{j_3}}) )
\nonumber\\&=&E_{j_3,i_1}(  (x_{j_1'}\ptl_{y_{j_1}}-x_{j_1}\ptl_{y_{j_1'}}) T(x^{\al+\es_{i_1'}} y^{\be-\es_{j_3}}))\nonumber\\& &
-E_{j_3,i_1'}(  (x_{j_1'}\ptl_{y_{j_1}}-x_{j_1}\ptl_{y_{j_1'}})T(x^{\al+\es_{i_1 }} y^{\be-\es_{j_3}}) )
 \end{eqnarray}
Then \begin{equation}   \td{\Dlt} ( (x_{j_1'}\ptl_{y_{j_1}}-x_{j_1}\ptl_{y_{j_1'}})T(x^{\al+\es_{i_1'}} y^{\be-\es_{j_3}}) )=
(x_{j_1'}\ptl_{y_{j_1}}-x_{j_1}\ptl_{y_{j_1'}})\td{\Dlt} (T(x^{\al+\es_{i_1'}} y^{\be-\es_{j_3}})   )=0.\end{equation}
By (\ref{a1.4}), we have $ (x_{j_1'}\ptl_{y_{j_1}}-x_{j_1}\ptl_{y_{j_1'}})T(x^{\al+\es_{i_1'}} y^{\be-\es_{j_3}}) \in M$.
Furthermore, according to (\ref{a2.50}),
\begin{eqnarray} &&\mfk d\left( (x_{j_1'}\ptl_{y_{j_1}}-x_{j_1}\ptl_{y_{j_1'}})T(x^{\al+\es_{i_1'}} y^{\be-\es_{j_3}}) \right)
\nonumber\\&=& \mfk d\left( (x_{j_1'}\ptl_{y_{j_1}}-x_{j_1}\ptl_{y_{j_1'}}) x^{\al+\es_{i_1'}} y^{\be-\es_{j_3}}  \right)\leq i \end{eqnarray}
So \begin{equation} (x_{j_1'}\ptl_{y_{j_1}}-x_{j_1}\ptl_{y_{j_1'}})T(x^{\al+\es_{i_1'}} y^{\be-\es_{j_3}}) \in M_{i }.\end{equation}
Similarly, we have \begin{equation} (x_{j_1'}\ptl_{y_{j_1}}-x_{j_1}\ptl_{y_{j_1'}})T(x^{\al+\es_{i_1 }} y^{\be-\es_{j_3}}) \in M_{i } .\end{equation}
Hence \begin{equation}\Delta^{j_1,j_1'}_{i_1,i_1'}(v_i)  \in M_{i+1}.\end{equation}

Next we consider $\be=0$. If $j_1,j_1'\neq n_1+1$, then
 \begin{eqnarray}\Delta_{i_1,i_1'}^{j_1,j_1'}(v_i)& =&(x_{i_1'}y_{i_1}-x_{i_1}y_{i_1'})(x_{j_1'}\ptl_{y_{j_1}}-x_{j_1}\ptl_{y_{j_1'}})(T(x^\al))
 \nonumber\\& =&(x_{i_1'}y_{i_1}-x_{i_1}y_{i_1'})(T(x_{j_1'}\ptl_{y_{j_1}}-x_{j_1}\ptl_{y_{j_1'}})(x^\al))=0 \end{eqnarray}
by (\ref{a4.35}). So we may assume $j_1=n_1+1$ and $j_1'\neq n_1+1$.
 Now
\begin{equation}v_i=T(x^\al)=x^\al+\sum_{\iota=1}^\infty\frac{\al_{n_1+1}!(x_{n_1+1}y_{n_1+1})^\iota(\td\Dlt+\ptl_{x_{n_1+1}}\ptl_{y_{n_1+1}})^{\iota-1}}
{(\al_{n_1+1}+\iota)!\iota!}\sum_{j=n_2+1}^n\al_jy_jx^{\al-\es_j}.\end{equation}
Now
\begin{eqnarray}
\Delta_{i_1,i_1'}^{j_1,j_1'}(v_i)
&=&\sum_{\iota=1}^\infty\sum_{j=n_2+1}^n\frac{\al_{n_1+1}!\al_j}{(\al_{n_1+1}+\iota)!\iota!}y_j(x_{i_1'}y_{i_1}-x_{i_1}y_{i_1'})
\nonumber\\&&\times(x_{j_1'}\ptl_{y_{n_1+1}}-x_{n_1+1}\ptl_{y_{j_1'}})(x_{n_1+1}y_{n_1+1})^\iota(\td\Dlt+\ptl_{x_{n_1+1}}\ptl_{y_{n_1+1}})^{\iota-1}
x^{\al-\es_j}\nonumber\\&=&\sum_{j=n_2+1}^n\frac{\al_j(x_{i_1'}y_jy_{i_1}-x_{i_1}y_jy_{i_1'})}{\al_{n_1+1}+1}\sum_{\iota=1}^\infty\frac{(\al_{n_1+1}+1)!}{(\al_{n_1+1}+1+\iota-1)!(\iota-1)!}
\nonumber\\&&\times (x_{n_1+1}y_{n_1+1})^{\iota-1}(\td\Dlt+\ptl_{x_{n_1+1}}\ptl_{y_{n_1+1}})^{\iota-1}
x^{\al+\es_{n_1+1}+\es_{j_1'}-\es_j}
\nonumber\\&=&\sum_{j=n_2+1}^n\frac{\al_j(x_{i_1'}(y_jy_{i_1}-x_jx_{i_1})+x_{i_1}(x_jx_{i_1'}-y_jy_{i_1'}))}{\al_{n_1+1}+1}T(x^{\al+\es_{n_1+1}+\es_{j_1'}-\es_j})
\nonumber\\&=&\sum_{j=n_2+1}^n\frac{\al_j(x_{i_1'}E_{j,i_1}-x_{i_1}E_{j,i_1'})}{\al_{n_1+1}+1}T(x^{\al+\es_{n_1+1}+\es_{j_1'}-\es_j})
\nonumber\\&=&\sum_{j=n_2+1}^n\frac{\al_j}{\al_{n_1+1}+1}
[E_{j,i_1}(T(x^{\al+\es_{i_1'}+\es_{n_1+1}+\es_{j_1'}-\es_j}))\nonumber\\&&-E_{j,i_1'}(T(x^{\al+\es_{i_1}+\es_{n_1+1}+\es_{j_1'}-\es_j}))],
\end{eqnarray}where we have used
\begin{eqnarray}&&[x_{j_1'},(\td\Dlt+\ptl_{x_{n_1+1}}\ptl_{y_{n_1+1}})^{\iota-1}]
(x^{\al+\es_{n_1+1}-\es_j})
\nonumber\\&=&(\iota-1)(\td\Dlt+\ptl_{x_{n_1+1}}\ptl_{y_{n_1+1}})^{\iota-2}\ptl_{y_{j_1'}}(x^{\al+\es_{n_1+1}-\es_j})=0.
\end{eqnarray}
Since
\begin{equation}\mfk d(x_{n_1+1})=\mfk d(x_{j_1'})=1,\;\;\mfk d(x_j)=2\end{equation}
as operators on $\msr A$, we have
\begin{equation}\mfk d(T(x^{\al+\es_{i_1'}+\es_{n_1+1}+\es_{j_1'}-\es_j})),\mfk d(T(x^{\al+\es_{i_1}+\es_{n_1+1}+\es_{j_1'}-\es_j}))\leq i.\end{equation}
Thus
\begin{equation}\Delta_{i_1,i_1'}^{j_1,j_1'}(v_i)\in M_{i+1}.\end{equation} 

Suppose $\ell_1,\ell_2<0$, then there always exist $j_3\in J_3$ and $i_1\in J_1$ such that $\be_{j_3}>0$ and $\al_{i_1}>0$.
For $t_1,...,t_s\in J_1,\;t_1',...,t_s'\in J_3$ and $T(x^\al y^\be)\in M_{i-s}$,
\begin{eqnarray}\label{a4.50}
&&\Delta_{i_1,i_1'}^{j_1,j_1'}[(y_{t_1}y_{t_1'}-x_{t_1}x_{t_1'})\cdots(y_{t_s}y_{t_s'}-x_{t_s}x_{t_s'})T(x^\al y^\be)]
\nonumber\\&=&\Delta_{i_1,i_1'}^{j_1,j_1'}[E_{t_1',t_1}\cdots E_{t_s',t_s}(T(x^\al y^\be))]
\nonumber\\&\in& E_{t_1',t_1}\cdots E_{t_s',t_s}[\Delta_{i_1,i_1'}^{j_1,j_1'}(T(x^\al y^\be))]+U_{s+1}(\mfk g)(T(x^\al y^\be))
\nonumber\\&\subset& U_s(\mfk g)(M_{i-s+1})+U_{s+1}(\mfk g)(M_{i-s})\subset M_{i+1}.
\end{eqnarray}
By (\ref{a2.60}) and Proposition 2.1, we get
\begin{equation}\label{a4.51}
\Delta_{i_1,i_1'}^{j_1,j_1'}(M_i)\subset M_{i+1}.\end{equation}
Hence  $\Delta^{j_1,j_1'}_{i_1,i_l'} \in I_{(2)}$ for each $(j_1,i_1), (j_1',i_1')\in L_1$ (cf. (\ref{a2.51}) and (\ref{a4.13})).
Symmetrically, when there exists $i_1\in J_1$ such that $\al_{i_1}>0$, we also have $\Delta^{j_2,j_2'}_{i_2 ,i_2'}\in I_{(2)}$ for each $(j_2,i_2), (j_2',i_2')\in L_2$.
Note that there always exist $j_3\in J_3$ and $i_1\in J_1$ such that $\be_{j_3}>0$ and $\al_{i_1}>0$, when $\ell_1,\ell_2<0$.
Thus (\ref{a4.20}) holds.

Suppose $\ell_2>0$ and $\ell_1\leq 0$.  Take $\xi$ in (\ref{a4.23}) and $v_0$ in (\ref{a4.24}). Expressions (\ref{a4.25})-(\ref{a4.33}) still hold.
For each $j_1\in J_2$, we set
$w_{j_1}= y_{j_1}^{  \ell_2 }x_{1}^{ -\ell_1 } \in M_0$.
Recall that $\xi(w_{j_1}) \in M_1$ (cf. (\ref{a4.23}) and (\ref{a4.12})). For $(j_1,i_1),(j_1',i_1')\in L_1$,
\begin{equation}\label{a4.52}
\mbox{the coefficient of}\; x_{j_1'} x_{i_1'} y_{i_1} y_{j_1}^{  \ell_2-1 }x_{1}^{ -\ell_1 }\;\mbox{in}\;\xi(w_{j_1})= \ell_2(a_{ j_1,i_1;j_1',i_1' }+a_{ j_1',i_1';j_1,i_1 })=0.\end{equation}
By (\ref{a4.23}), (\ref{a4.27})-(\ref{b4.29}) and (\ref{a4.52}),
\begin{equation}
\xi \in  \sum_{(j_2,i_2),(j_2',i_2')\in L_2;(j_2,i_2)\leq (j_2',i_2')}a_{j_2,i_2;j_2',i_2'}
(E_{j_2,i_2}E_{j_2',i_2'}-E_{j_2,i_2'}E_{j_2',i_2})+ U_1(\mfk g).
\end{equation}
Using the same arguments as in (\ref{a4.35})-(\ref{a4.51}), we can prove (\ref{a4.21}).

Let $v_i=T(x^\al y^\be)\in M_i$ with ${\mfk d}(T(x^\al y^\be))={\mfk d}(x^\al y^\be) =i$ (cf. (\ref{a2.50})).
When $\sum_{j\in J_3} \be_j>0$, we have shown that $\Delta_{i_1,i_1'}^{j_1,j_1'}(v_i) \in M_{i+1}$ for $(j_1,i_1),(j_1',i_1') \in L_1$.
Consider $\sum_{j_3\in J_3} \be_{j_3}=0$, we have $\sum_{j_2\in J_2} \be_{j_2} \leq \ell_2$. Then 
\begin{equation}
\left(\Delta^{j_1,j_1 '}_{i_1,i_l'} \right)^{\ell_2+1}(v_i)=(x_{i_1'}y_{i_1}-x_{i_1}y_{i_1'})^{\ell_2+1}(x_{j_1'}\ptl_{y_{j_1}}-x_{j_1}\ptl_{y_{j_1'}})^{\ell_2+1}(T(x^\al y^\be))=0.
\end{equation}
Hence $\left(\Delta^{j_1,j_1 '}_{i_1,i_l'} \right)^{\ell_2+1}(v_i) \in M_{i+2\ell_2+1}$. Using the same arguments as in (\ref{a4.50})-(\ref{a4.51}), we can prove (\ref{b4.19}).
Similarly, when $\ell_1>0$ and $\ell_2\leq 0$, (\ref{a4.22}) and (\ref{b4.21}) holds.

\end{proof}

Denote  
\begin{equation}
\ol{I}_{(2)}=\left\{ \left. \Delta^{j_1,j_1'}_{i_1,i_l'}  ,\Delta^{j_2,j_2'}_{i_2 ,i_2'} \right\vert (j_1,i_1), (j_1',i_1')\in L_1, (j_2,i_2), (j_2',i_2')\in L_2\right\}.
\end{equation}

\begin{lemma}\label{alem:4.3}
For each $p\geq 3$, $\ell_1\leq 0$ or $\ell_2<0$, we have
\begin{equation}\label{a4.54}
  I_{(p)}\equiv\left\langle \left. \Delta^{j_1,j_2,j_3}_{i_1,i_2,i_3 }  \right\vert  i_1,i_2,i_3\in J_1\cup J_2; j_1,j_2,j_3 \in J_2\cup J_3\right\rangle \quad(\mbox{mod}\:\la I_{(1)},\ol{I}_{(2)}\ra). \end{equation}
\end{lemma}

\begin{proof}

Suppose $p \geq 3$, we set
\begin{align}
\xi =  \sum_{l=1}^m c_l\left(
\prod_{s=1}^{k_{1,l}} E_{j_{1,l,s},i_{1,l,s}}
\prod_{t=1}^{k_{2,l}} E_{j_{2,l,t},i_{2,l,t}}
\prod_{r=1}^{k_{3,l}}   E_{j_{3,l,r},i_{3,l,r}}
\right)
 \in   \mfk k_{(p)}
\end{align}
(cf. (\ref{a4.12})), where $k_{1,l},k_{2,l},k_{3,l}\in\mbb N$ satisfying $k_{1,l}+k_{2,l}+k_{3,l}=p$, and
\begin{equation}(j_{1,l,s},i_{1,l,s})\in L_1,\ (j_{2,l,t},i_{2,l,t})\in L_2,\ (j_{3,l,r},i_{3,l,r})\in L_3,
 \end{equation}\begin{equation}c_l\in \mathbb{F} \quad \ \
 \text{for each } l \in \ol{1,m},  \  s\in \ol{1,k_{1,l}},\ t\in \ol{1,k_{2,l}},\ r\in \ol{1,k_{3,l}}.
\end{equation}

For any $v_0\in M_0$, we consider the partial sum $h=\sum_{l=1}^m   h_l$ of highest degree in $\xi(v_0)$
, where
\begin{align}\label{a4.58}
h_l= (-1)^{k_{1,l}}c_l\left(
\prod_{s=1}^{k_{1,l}} x_{j_{1,l,s}}x_{i_{1,l,s}}
\prod_{t=1}^{k_{2,l}} y_{j_{2,l,t}}y_{i_{2,l,t}}
\prod_{r=1}^{k_{3,l}} ( y_{j_{3,l,r}}y_{i_{3,l,r}} -x_{j_{3,l,r}}x_{i_{3,l,r}})
\right)v_0.
\end{align}
We will show that $h=0$ can be derived from $\xi (v_0)\in M_{p-1}$.

Suppose that the set of distinct elements of $\{k_{3,l} \vert l \in \ol{1,m}\}$ is $\{\jmath_1,\dots,\jmath_N\}$, and $\jmath_1<\jmath_2<\cdots<\jmath_N$. For $\iota\in \ol{1,N}$, we set
\begin{equation}
\ol{h}_\iota =\sum_{l \in \ol{1,m}:k_{3,l}=\jmath_\iota}   h_l.
\end{equation}

Suppose that $\ol{h}_\iota \neq 0$ for some $\iota \in \ol{1,N}$. According to (\ref{a2.48}):
\begin{equation}\label{b4.57}
{\mfk d} \left( \xi(v_0) \right) \geq {\mfk d}\left( \ol{h}_\iota  \right) = p+ \jmath_\iota .
\end{equation}
Denote the degree in $\{x_{i_3} \mid   i_3 \in J_3  \} $ as $\deg_{x_{_{J_3}}} \left( \xi(v_0)\right)$ and the degree in $\{y_{i_1} \mid   i_1 \in J_1  \} $ as $\deg_{y_{_{J_1}}} \left( \xi(v_0)\right)$.
For $\iota' \in \mathbb{N}$ and $f=\sum_{i=1}^{m'} f_i\in \msr A$, where $f_i$ are distinct monomials in $\msr A$, we define
\begin{equation}
{\mfk r} (f,\iota')=\min \{ \deg_{x_{_{J_3}}}(f_i) + \deg_{y_{_{J_1}}} (f_i) \mid \mfk d (f_i)=\iota', i\in \ol{1,m'} \}.
\end{equation}
Then we have
\begin{equation}\label{b4.59}
{\mfk r} (\xi(v_0),p+ \jmath_\iota) \leq {\mfk r}\left( \ol{h}_\iota , \mfk d (\ol{h}_\iota) \right)=\jmath_\iota.
\end{equation}

Since $\xi (v_0) \in M_{p-1}$ and $\xi (v_0) \notin TN_{p-1}$, according to Lemma \ref{alem:4}, we have
\begin{equation}
{\mfk d} \left( \xi (v_0) \right)= \mbox{ord}_{p-1}\left(\xi (v_0)\right) + p-1
\Rightarrow \mbox{ord}_{ p-1}\left( \xi (v_0) \right)  \geq 1+ \jmath_\iota.
\end{equation}
Recall (\ref{b4.6}) and (\ref{a2.1}). For $r= \mbox{ord}_{ p-1}\left( \xi (v_0) \right)  \geq 1+ \jmath_\iota $ in this case, we can write
\begin{align}
\xi (v_0)=\sum_{s_r=1}^{\ell_r} g_{r,s_r} f_{r,s_r} +  \sum_{s_{r-1}=1}^{\ell_{r-1}} g_{r-1,s_{r-1}} f_{r-1,s_{r-1}} + \cdots +f_{0},
\end{align}
where $ g_{i,s_i}  \in P^{\left(i \right)}$, $f_{i,s_i}  \in TN(p-1-i)$ for $  s_i\in\ol{1,\ell_i} $, and $f_{0} \in TN_{p-1}$.
Note that for $  s_i\in\ol{1,\ell_i} $ and $i\in \ol{1,r}$,
\begin{equation}
{\mfk d} ( g_{i,s_i} f_{i,s_i} ) = p-1-i+2i=p-1+i
\end{equation}
by (\ref{a2.59}).
Since $\ol{h}_\iota\neq 0$ and ${\mfk d} (\ol{h}_\iota)=p+\jmath_\iota$, we have
\begin{equation}  \sum_{s_{\jmath_\iota+1}=1}^{\ell_{\jmath_\iota+1}} g_{\jmath_\iota+1,s_{\jmath_\iota+1}} f_{\jmath_\iota+1,s_{\jmath_\iota+1}} \neq 0.\end{equation}
Therefore,
\begin{equation}
{\mfk r} (\xi(v_0),p+\jmath_\iota) = {\mfk r}\left( \sum_{s_{\jmath_\iota+1}=1}^{\ell_{\jmath_\iota+1}} g_{\jmath_\iota+1,s_{\jmath_\iota+1}} f_{\jmath_\iota+1,s_{\jmath_\iota+1}} , p+ \jmath_\iota \right) \geq \jmath_\iota+1,
\end{equation}
which contradicts (\ref{b4.59}).
Hence, $\ol{h}_\iota=0$ for $w\in \ol{1,N}$.
Therefore $h=0$.

By considering different monomials in $X_{J_2}$ and $Y_{J_2}$, we may assume that $k_{1,l},k_{2,l} ,k_{3,l}$ in (\ref{a4.58}) are independent  of $l \in \ol{1,m}$; that is, $k_{1,l}=k_1,k_{2,l}=k_2 ,k_{3,l}=k_3$ for $l \in \ol{1,m}$. Thus (\ref{a4.58}) yields
\begin{align}
0=\sum_{l=1}^m c_l\left(
\prod_{s=1}^{k_{1}}x_{i_{1,l,s}}
\prod_{t=1}^{k_{2}} y_{j_{2,l,t}}
\prod_{r=1}^{k_{3}} ( y_{j_{3,l,r}}y_{i_{3,l,r}} -x_{j_{3,l,r}}x_{i_{3,l,r}})
\right)=h'.
\end{align}
According to (\ref{a3.17}),
\begin{equation}\phi\left( \sum_{l=1}^m c_l\left(
\prod_{s=1}^{k_{1 }}  z_{n+1,i_{1,l,s}}
\prod_{t=1}^{k_{2 }} z_{j_{2,l,t},0}
\prod_{r=1}^{k_{3 }}  z_{j_{3,l,r},i_{3,l,r}}
\right)
\right)=h'=0.
\end{equation}

By Lemma \ref{alem:3.3}, we have
\begin{equation}  \sum_{l=1}^m c_l\left(
\prod_{s=1}^{k_{1 }}  z_{n+1,i_{1,l,s}}
\prod_{t=1}^{k_{2 }} z_{j_{2,l,t},0}
\prod_{r=1}^{k_{3 }}  z_{j_{3,l,r},i_{3,l,r}}
\right)   \in \msr R_3.
\end{equation}

Let $i_1,i_2,i_3\in J_1'$ and $j_1,j_2,j_3\in J_3'$ (cf. (\ref{a3.18})) such that $i_1<i_2<i_3$ and $j_1<j_2<j_3$.
Denote
\begin{eqnarray}
\Lambda^{j_1,j_2,j_3}_{i_1,i_2,i_3}=
&z_{j_1,i_1}z_{j_2,i_2}z_{j_3,i_3}+z_{j_1,i_2}z_{j_2,i_3}z_{j_3,i_1}+z_{j_1,i_3}z_{j_2,i_1}z_{j_3,i_2}\nonumber\\
&-z_{j_1,i_1}z_{j_2,i_3}z_{j_3,i_2}-z_{j_1,i_2}z_{j_2,i_1}z_{j_3,i_3}
 -z_{j_1,i_3}z_{j_2,i_2}z_{j_3,i_1}
\end{eqnarray}
Then we can write
\begin{equation}h^\prime=\phi\left( \sum_{s=1}^{m'} \Lambda^{j_{1,s},j_{2,s},j_{3,s}}_{i_{1,s},i_{2,s},i_{3,s}} g_{s}\right) , \end{equation}
where $g_s \in \msr C_{p-3}$, $j_{1,s},j_{2,s},j_{3,s} \in J_3'$ and $i_{1,s},i_{2,s},i_{3,s}\in J_1'$.

Consider the correspondence between the terms in $h'$ and the terms in $\xi$.
When $j_{1,s},j_{2,s},j_{3,s} \in J_3 $ and $i_{1,s},i_{2,s},i_{3,s}\in J_1 $,
$ \phi( \Lambda^{j_{1,s},j_{2,s},j_{3,s}}_{i_{1,s},i_{2,s},i_{3,s}} )$ corresponds to $\Delta^{j_{1,s},j_{2,s},j_{3,s}}_{i_{1,s},i_{2,s},i_{3,s}} $.
Note $x_{j_{1,l,t}}x_{j_{1,l,t'}}=x_{j_{1,l,t'}}x_{j_{1,l,t}}$, but
$E_{j_{1,l,t},i_{1,l,t}}E_{j_{1,l,t'},i_{1,l,t'}}$ may not be equal to $E_{j_{1,l,t'},i_{1,l,t}}E_{j_{1,l,t},i_{1,l,t'}}.$
So $\prod_{t=1}^{k_{1,l}} z_{n +1,j_{1,l,t} } $ does not only correspond to $\prod_{t=1}^{k_{1,l}} E_{j_{1,l,t},i_{1,l,t} } $, but also corresponds to $\prod_{t=1}^{k_{1,l'}} E_{j_{1,l',t},i_{1,l,t} } $ with
\begin{equation}\{ j_{1,l,t} \mid t\in \ol{1,k_{1 }} \}=\{ j_{1,l',t} \mid t\in \ol{1,k_{1 }} \}=\msr I_1^x \left( \prod_{t=1}^{k_{1 }}x_{ j_{1,l,t} } \right).\end{equation}(cf.(\ref{a3.22})).
According to Lemma \ref{alem:2}, $\prod_{t=1}^{k_{1 }} z_{n +1,j_{1,l,t} } $ corresponds to
\begin{equation}\label{a4.74}
\prod_{t=1}^{k_{1 }} E_{j_{1,l,t},i_{1,l,t} }+ \left\langle \Delta^{j_1,j_1'}_{i_1,i_1'}
\mid (j_1,i_1), (j_1',i_1')\in L_1
\right\rangle\quad\mbox{in}\;S(\mfk g). \end{equation}
Similarly, $\prod_{t=1}^{k_{2 }} z_{j_{2,l,t},0} $ corresponds to
\begin{equation}\label{a4.75}
\prod_{t=1}^{k_{2 }} E_{j_{2,l,t},i_{2,l,t} }+ \left\langle \Delta^{j_2,j_2'}_{i_2,i_2'}
\mid (j_2,i_2), (j_2',i_2')\in L_2
\right\rangle\quad\mbox{in}\;S(\mfk g). \end{equation}

Fix $y_{i'}$ in $\prod_{t=1}^{k_2}y_{i_{2,l,t}}$ (cf. (\ref{a4.58})). Then
$\Lambda_{0,i_{2,s},i_{3,s}}^{j_{1,s},j_{2,s},j_{3,s}} y_{i'}$ corresponds to $\Delta_{i',i_{2,s},i_{3,s}}^{j_{1,s},j_{2,s},j_{3,s}}$ in the sense as (\ref{a4.75}) when $i_{2,s},i_{3,s}\in  J_1$ and $j_{1,s},j_{2,s},j_{3,s} \in J_3 $.
Fix $x_{j'}$ in $\prod_{s=1}^{k_{1,l}}x_{j_{1,l,s}}$ (cf. (\ref{a4.58})).
Similarly, $\Lambda_{i_{1,s},i_{2,s},i_{3,s}}^{j_{1,s},j_{2,s},n+1} x_{j'}$ corresponds to $\Delta_{i_{1,s},i_{2,s},i_{3,s}}^{j_{1,s},j_{2,s},j'} $ in the sense as (\ref{a4.74}) when $i_{1,s},i_{2,s},i_{3,s}\in  J_1$ and $j_{1,s},j_{2,s}  \in J_3 $.

When $j_1,i_3\in J_2$, $j_2,j_3\in J_3$ and $i_1,i_2\in J_1$, we recall (\ref{a4.14}) and have
\begin{eqnarray}\label{a4.76}
 \Delta^{j_1,j_2,j_3}_{i_1,i_2,i_3 }
&\equiv&  E_{j_1,i_1} E_{j_2,i_2} E_{j_3,i_3}+E_{j_1,i_2} E_{j_2,i_3} E_{j_3,i_1}
 \nonumber\\& &-E_{j_1,i_2} E_{j_2,i_1} E_{j_3,i_3} -E_{j_1,i_1} E_{j_2,i_3} E_{j_3,i_2}\quad(\mbox{mod}\;\la I_{(1)}\ra)
\end{eqnarray}
as elements in $S(\mfk g)$. Then $\Lambda^{j_{1,s},j_{2,s},n+1}_{0,i_{2,s},i_{3,s}} x_{i'}y_{j'}$ corresponds to $ \Delta^{j_{1,s},j_{2,s},j'}_{i',i_{2,s},i_{3,s}}$ in the sense as (\ref{a4.76}) when $ j_{2,s},j_{3,s}\in  J_3$ and $i_{1,s},i_{2,s}  \in J_1 $.
Hence
\begin{eqnarray}\label{a4.77}
\xi \in&\la \Delta^{j_{1},j_{2},j_{3}}_{i_{1},i_{2},i_{3}}\mid j_1,j_2,j_3\in J_2\cup J_3,\;i_{1},i_{2},i_{3} \in J_1\cup J_2\ra
+\la I_{(1)},\ol{I}_{(2)}\ra 
 \end{eqnarray}

Next we want to verify that the above $\Delta^{j_{1 },j_{2 },j_{3 }}_{i_{1 },i_{2 },i_{3 }}\in I_{(3)} $ for $j_1<j_2<j_3$ and $i_1<i_2<i_3$, which can always be obtained by re-indexing up to a sign. We divide it into six cases.
\pse

(1) {\it $j_1 \in J_2$ and $i_3 \in J_2$ do not occur at the same time, $j_1,j_2,j_3 \in J_2 \cup J_3$, $i_1,i_2,i_3 \in J_1 \cup J_2$.}\pse

In this case, we have
\begin{align}\label{b4.75}
 \Delta^{j_1,j_2,j_3}_{i_1,i_2,i_3}|_{\msr A}   =0.
\end{align}
For example, for $i_1,i_2  \in J_1$, $i_3 \in J_2$, $j_1,j_2,j_3 \in J_3$,
\begin{eqnarray}
 E_{j_1,i_1} E_{j_2,i_2} E_{j_3,i_3} |_{\msr A}  = (y_{j_1}y_{i_1}-x_{j_1}x_{i_1}) (y_{j_2}y_{i_2}-x_{j_2}x_{i_2})( x_{j_3}\ptl_{x_{i_3}} +y_{j_3}y_{i_3}  )
\end{eqnarray}

Define an isomorphism $\nu$ from $\mathbb{F} [X_{J_1\cup J_3},Y_{\ol{1,n}}, \partial_{x_{n_1+1}},...,\partial_{x_{n_2}}]$ to $\msr A$ by
\begin{equation}\nu(\partial_{x_{i_3}})= -x_{i_3},\;\  \nu(x_{j' })=x_{j' },\;\  \nu(y_{j } )=y_{j }\end{equation} for  $i_3\in J_2$, $j\in \ol{1,n},j' \in J_1\cup J_3$.
Then
\begin{eqnarray}
& &\nu \left( (y_{j_1}y_{i_1}-x_{j_1}x_{i_1}) (y_{j_2}y_{i_2}-x_{j_2}x_{i_2})( x_{j_3}\ptl_{x_{i_3}} +y_{j_3}y_{i_3}  )\right) \\
&=& (y_{j_1}y_{i_1}-x_{j_1}x_{i_1}) (y_{j_2}y_{i_2}-x_{j_2}x_{i_2})( y_{j_3}y_{i_3}  -x_{j_3} x_{i_3}
)
\end{eqnarray}
We have already verified  $ \nu (\Delta^{j_1,j_2,j_3}_{i_1,i_2,i_3}|_{\msr A} )=0$ in the proof of Lemma \ref{alem:3.3}. So (\ref{b4.75}) holds.
\pse

(2) {\it $j_1,i_3\in J_2$, $j_2,j_3\in J_3$ and $i_1,i_2\in J_1$.}\pse

Let $v_i=T(x^\al y^\be)\in M_i$ with ${\mfk d}(T(x^\al y^\be))={\mfk d}(x^\al y^\be) =i$. Then
\begin{eqnarray}
& &\Delta^{j_1,j_2,j_3}_{i_1,i_2,i_3} (v_i) \nonumber\\
&\in &   \left( E_{j_1,i_1} E_{j_2,i_2} E_{j_3,i_3}+E_{j_1,i_2} E_{j_2,i_3} E_{j_3,i_1}
 -E_{j_1,i_2} E_{j_2,i_1} E_{j_3,i_3} -E_{j_1,i_1} E_{j_2,i_3} E_{j_3,i_2} \right)(v_i )+M_{i+2} \nonumber\\
 &\subset& (x_{i_1}y_{i_2}-x_{i_2}y_{i_1})(x_{j_2}y_{j_3}-x_{j_3}y_{j_2})(x_{j_1}\ptl_{x_{i_3}}-y_{i_3}\ptl_{y_{j_1}})(v_i  )+M_{i+2} \nonumber\\
 &\subset&\left( E_{j_3,i_1}E_{j_2,i_2} -E_{j_3,i_2}E_{j_2,i_1}\right) (x_{j_1}\ptl_{x_{i_3}}-y_{i_3}\ptl_{y_{j_1}})(v_i  )+M_{i+2}
\end{eqnarray}
According to (\ref{a1.7}), we have
\begin{equation}
[\td{\Delta},  x_{j_1}\ptl_{x_{i_3}}-y_{i_3}\ptl_{y_{j_1}} ]=0 \end{equation}
So
\begin{equation}
\td{\Delta}\left( (x_{j_1}\ptl_{x_{i_3}}-y_{i_3}\ptl_{y_{j_1}})(v_i)\right)=0; \end{equation}
that is
\begin{equation}
 \left( (x_{j_1}\ptl_{x_{i_3}}-y_{i_3}\ptl_{y_{j_1}})(v_i)\right) \in M.\end{equation}

Furthermore,
\begin{equation}
 \mfk d \left( (x_{j_1}\ptl_{x_{i_3}}-y_{i_3}\ptl_{y_{j_1}})(v_i)\right)\leq i \end{equation}
Hence
\begin{equation}
  (x_{j_1}\ptl_{x_{i_3}}-y_{i_3}\ptl_{y_{j_1}})(v_i) \in M_i.\end{equation}
Therefore,
\begin{eqnarray}
& &   \Delta^{j_1,j_2,j_3}_{i_1,i_2,i_3}(v_i) \nonumber\\
 &\in &\left( E_{j_3,i_1}E_{j_2,i_2} -E_{j_3,i_2}E_{j_2,i_1}\right) \left( (x_{j_1}\ptl_{x_{i_3}}-y_{i_3}\ptl_{y_{j_1}})(v_i)\right)+M_{i+2}\subset M_{i+2}.
\end{eqnarray}
Moreover, we can prove as (\ref{a4.50}) that
\begin{equation} \Delta^{j_1,j_2,j_3}_{i_1,i_2,i_3} ( M_i)\in M_{i+2} \end{equation}
for each $i\in \mathbb{N}$, thus
\begin{equation}\Delta^{j_1,j_2,j_3}_{i_1,i_2,i_3 } \in I_{(3)}.\end{equation}

\pse

(3) {\it $j_1,j_2,i_3\in J_2$, $ j_3\in J_3$ and $i_1 ,i_2\in J_1$.}
\pse

Recall (\ref{a4.14}) and $v_i=T(x^\al y^\be)\in M_i$ with ${\mfk d}(T(x^\al y^\be))={\mfk d}(x^\al y^\be) =i$. In this case,
\begin{equation}\label{a4.88}
\Delta^{j_{1 },j_{2 },j_{3 }}_{i_{1},i_{2 },i_{3 }}(v_i)\equiv
\Delta^{j_1,j_2}_{i_1,i_2}  E_{j_3,i_3} (v_i) \quad(\mbox{mod}\;\la I_{(1)}\ra)
.\end{equation}
Then
\begin{eqnarray}
& & \Delta^{j_1,j_2}_{i_1,i_2}  E_{j_3,i_3} (v_i) \nonumber\\
 &=&( x_{i_2}y_{i_1}-x_{i_1}y_{i_2} ) ( x_{j_2}\ptl_{y_{j_1}}-x_{j_1}\ptl{y_{j_2}} )  (x_{j_3}\ptl_{x_{i_3}}+y_{j_3} y_{i_3})(v_i)  \nonumber\\
 &=& ( x_{i_2}x_{j_3}y_{i_1}-x_{i_1}x_{j_3}y_{i_2} ) ( x_{j_2}\ptl_{y_{j_1}}-x_{j_1}\ptl{y_{j_2}} )  \ptl_{x_{i_3}} (v_i)  +\nonumber\\
& &( x_{i_2}y_{i_1}  y_{j_3}-x_{i_1}y_{i_2}   y_{j_3}) ( x_{j_2}\ptl_{y_{j_1}}-x_{j_1}\ptl{y_{j_2}} )   y_{i_3} (v_i)  \nonumber\\
 &=&(-E_{j_3,i_2}y_{i_1}+E_{j_3,i_1}y_{i_2} )  ( x_{j_2}\ptl_{y_{j_1}}-x_{j_1}\ptl{y_{j_2}} )  \ptl_{x_{i_3}} (v_i) +\nonumber\\
& &( E_{j_3,i_1}x_{i_2}-E_{j_3,i_2}x_{i_1}) ( x_{j_2}\ptl_{y_{j_1}}-x_{j_1}\ptl{y_{j_2}} )   y_{i_3} (v_i)  \nonumber\\
 &=&  E_{j_3,i_1}(y_{i_2}\ptl_{x_{i_3}} +x_{i_2}\ptl_{y_{i_3}})(x_{j_2}\ptl_{y_{j_1}}-x_{j_1}\ptl_{y_{j_2}})(v_i)  \nonumber\\
& & -E_{j_3,i_2} (y_{i_1}\ptl_{x_{i_3}} +x_{i_1}\ptl_{y_{i_3}}) ( x_{j_2}\ptl_{y_{j_1}}-x_{j_1}\ptl{y_{j_2}} ) (v_i)
\end{eqnarray}
According to (\ref{a1.7}),
\begin{equation} [ \td{\Dlt} ,x_{j_2}\ptl_{y_{j_1}}-x_{j_1}\ptl{y_{j_2}}]=0,\end{equation}
and
\begin{equation} [ \td{\Dlt} ,y_{i_1}\ptl_{x_{i_3}} +x_{i_1}\ptl_{y_{i_3}}]=0,\end{equation}
Thus
\begin{equation}
\td{\Delta}\left( (y_{i_2}\ptl_{x_{i_3}} +x_{i_2}\ptl_{y_{i_3}})(x_{j_2}\ptl_{y_{j_1}}-x_{j_1}\ptl_{y_{j_2}})(v_i)\right)=0 \end{equation}
Since
\begin{equation}\label{a4.97}
 \mfk d \left((y_{i_2}\ptl_{x_{i_3}} +x_{i_2}\ptl_{y_{i_3}})(x_{j_2}\ptl_{y_{j_1}}-x_{j_1}\ptl_{y_{j_2}})(v_i)\right)\leq i+1,\end{equation}
we have \begin{equation}
 (y_{i_2}\ptl_{x_{i_3}} +x_{i_2}\ptl_{y_{i_3}})(x_{j_2}\ptl_{y_{j_1}}-x_{j_1}\ptl_{y_{j_2}})(v_i) \in M_{i+1}.\end{equation}
Therefore,
\begin{equation}
\Delta^{j_1,j_2,j_3}_{i_1,i_2,i_3 } (v_i) \in  \Delta^{j_1,j_2}_{i_1,i_2}  E_{j_3,i_3} (v_i) +M_{i+2}
\subset M_{i+2}.\end{equation}
Furthermore, we can prove as (\ref{a4.50}) that
\begin{equation}
\Delta^{j_1,j_2,j_3}_{i_1,i_2,i_3 } (M_i) \in M_{i+2}\end{equation}
for $i\in \mathbb{N}$.
Then
\begin{equation}\label{a4.101}\Delta^{j_1,j_2,j_3}_{i_1,i_2,i_3 } \in I_{(3)}.\end{equation}
\pse

(4) {\it  $j_1,i_2,i_3\in J_2$, $j_2,j_3\in J_3$ and $i_1 \in J_1$.}
\pse

It is symmetrical to the case (3), and so (\ref{a4.101}) still holds.

\pse

(5) {\it  $j_1,j_2\in J_2,\;j_3\in J_2\cup J_3$ and $i_1\in J_1\cup J_2,\; i_2,i_3 \in J_2$.}
\pse

According to (\ref{a4.14}) and (\ref{a4.19}),
\begin{equation}\Delta^{j_1,j_2,j_3}_{i_1,i_2,i_3 }
\in\la  I_{(1)} \ra  .\end{equation}
Naturally,
\begin{equation}\Delta^{j_1,j_2,j_3}_{i_1,i_2,i_3 }
\in\  I_{(3)}  .\end{equation}
\pse

(6) {\it $i_1,i_2,i_3\in J_2$ and $j_1,j_2,j_3\in J_2\cup J_3$, or $i_1,i_2,i_3\in J_1\cup J_2$ and $j_1,j_2,j_3\in J_2$.}
\pse

When $i_1,i_2,i_3\in J_2$,
\begin{equation}\Delta^{j_1,j_2,j_3}_{i_1,i_2,i_3 }
=E_{j_3,i_1}\Delta^{ j_1,j_2}_{ i_2,i_3 }-E_{j_3,i_2}\Delta^{ j_1,j_2}_{ i_1,i_3 }
+E_{j_3,i_3}\Delta^{ j_1,j_2}_{ i_1,i_2 }
\in\  \la I_{(1)}\ra  .\end{equation}
Then \begin{equation}\Delta^{j_1,j_2,j_3}_{i_1,i_2,i_3 }
\in\  I_{(3)}  .\end{equation}
The case of $j_1,j_2,j_3\in J_2$ is similar.
\pse

In summary, \begin{equation}\label{a4.106}
\Delta^{j_1,j_2,j_3}_{i_1,i_2,i_3 }\in I_{(3)}\qquad\for\;\;i_1,i_2,i_3 \in J_1\cup J_2,\;j_1,j_2,j_3 \in J_2\cup J_3.  \end{equation} According to (\ref{a4.77}), we have proved that (\ref{a4.54}) holds.
This completes the proof of the lemma.
\end{proof}
\psp

In order to calculate the associated variety of the irreducible module ${\msr H}_{\la\ell_1,\ell_2\ra}$ of all the cases in Proposition 1, we consider
  $\ell_1,\ell_2>0$ and $n_1<n_2=n$. First we take
\begin{equation}
M_0 
 = \text{Span} \left\{  TN\begin{pmatrix}
0 & \ell_1 & 0 \\
k_{21} & k_{22} & 0
\end{pmatrix}  
\Big\vert
k_{21}+k_{22}=\ell_2
\right\}.
\end{equation}
In particular, when $n_1+1=n_2=n$, 
\begin{eqnarray}\label{a4.105}
M_0 = \text{Span} \left\{    \left. T(x_{n}^{\ell_1} y^\beta)      \right\vert  x_{n}^{\ell_1} y^\beta \in \msr A,
 \sum_{s=1}^{n-1} \beta_s =\ell_2 ,\beta_{n}=0
\right\}  .
\end{eqnarray}
Set
\begin{eqnarray}\label{b4.110}
V_k'
&=&\text{Span} \left\{  \left.  T(x^\alpha y^\beta)      \right\vert
x^\alpha y^\beta\in \msr A_{\la\ell_1,\ell_2\ra}  ,\sum_{t=1}^{n_1} \alpha_t  \leq k, \sum_{s=1}^{n} \beta_s =\ell_2,\al_{n_1+1}\be_{n_1+1}=0
\right\} \nonumber\\
&=&\text{Span} \left\{
TN\begin{pmatrix}
t & \ell_1+t & 0 \\
k_{21} & k_{22} & 0
\end{pmatrix} \Big\vert 0\leq t \leq k,k_{21}+k_{22}=\ell_2
\right\}.
\end{eqnarray}
In particular, when $n_1+1=n_2=n$, 
\begin{align}\label{a4.106}
V_k' =\text{Span} \left\{  \left.  T(x^\alpha y^\beta)      \right\vert
x^\alpha y^\beta\in \msr A_{\la\ell_1,\ell_2\ra}  ,\sum_{t=1}^{n-1} \alpha_t  \leq k, \sum_{s=1}^{n-1} \beta_s =\ell_2
\right\} 
\end{align}
Define ${\mfk d'} : {\msr A}_{\la\ell_1,\ell_2\ra}  \rightarrow \mathbb{N}$ as follows
\begin{align}
&{\mfk d'} \left(x^\alpha y^\beta \right) = \sum_{i \in J_1}\al_i ,\label{aa1} \\
&{\mfk d'} \left( \lmd f\right) ={\mfk d'} \left( f\right) , \  \forall \lmd   \in \mathbb{F} \setminus \{0\} , \\
&{\mfk d'}\left(\sum_{i=1}^m  f_i\right) =\max \{{\mfk d'}\left(f_1 \right) ,\cdots,{\mfk d'}\left(f_m \right) \},\text{where $f_i$ are monomials.}
\end{align}
Note that $\mfk d'$ is different from $\mfk d$ (cf. (\ref{a2.46})),  in order to ensure that $\mfk d'(M_0)=\{0\}$.

In this case, $L=L_1=J_2 \times J_1$ as (\ref{a2.51}).
Similar to (\ref{a2.48})-(\ref{a2.52}), the following properties of $\mfk d'$ still hold:
\begin{align}
{\mfk d'} \left( E_{i,j}(f) \right)  \leq
\begin{cases}
{\mfk d'} \left( f \right)+1  &\text{if}  \left( i,j\right)  \in     L ,  \\
{\mfk d'} \left( f \right) &\text{if}  \left( i,j\right)  \notin  L.
\end{cases}
\label{aa2}
\end{align}
For each $i \in \mathbb{N}$,
\begin{align}
{\mfk d'} \left(  \left(x_{n_1+1 } y_{n_1+1 }\right)^i
\left(\tilde{\Delta}+\partial_{x_{n_1+1 }}\partial_{ y_{n_1+1 }}\right)^i \left( x^\alpha y^\beta \right)  \right)
={\mfk d'} \left(x^\alpha y^\beta\right) \text{ or } 0,
\end{align}
where $\alpha_{n_1+1} \beta_{n_1+1 } =0$.
Then
\begin{align}
{\mfk d'} \left( T\left( x^\alpha y^\beta \right)  \right) ={\mfk d'} \left(x^\alpha y^\beta\right),\text{ where $ \alpha_{n_1+1} \beta_{n_1+1} =0$.}
\end{align}

\begin{lemma}For $k\in \mbb N$, $M_k=U_k(sl(n))(M_0)=V'_k$  holds.\end{lemma}

\begin{proof}
We prove $M_i=V_i'$ for any $i\in\mbb N$ by induction.
According to (\ref{b4.110}), $M_0=V_0' $ holds. Suppose that $M_i=V_i'$ with $i \leq k - 1$. Consider $i=k$.
Since (\ref{aa2})  holds, we can prove that $M_k\subset V_k'$ as (\ref{a2.61})-(\ref{a2.66}) with $V_k$ replaced by $V_k'$.
Next we want to prove $V_k'\subset M_k$.
Let $f=T (x^\al y^\be) \in V_k'$ with $ x^\alpha y^\beta\in \msr A_{\la\ell_1,\ell_2\ra}$, $\sum_{t=1}^{n_1} \alpha_t  \leq k$, $\sum_{s=1}^{n} \beta_s =\ell_2$, $\al_{n_1+1}\be_{n_1+1}=0$.
When $n_1+1<n_2$, we suppose $ \sum_{n_1+2\leq i_2\leq n} \al_{i_2}>0$. Then there exists $i_1\in J_1$ such that $\al_{i_1}>0$.\
According to (\ref{a2.7}), we have 
\begin{equation}
f=  -E_{i_2,i_1}  (T(x^{\al-\es_{i_1}-\es_{i_2} } y^\be))                                                                                                                                                                                                                                                                                                                                                                                                                                                                                                                                                                                                                                                                                                                                                                                                                                                                                                                                                                                                                                                                                                                                                                                                                                                                                                                                                                                                                                                                                                                                                                                                                                                                                         - T(\ptl_{y_{i_2}}y^{\be+\es_{i_1}})\in M_{k}.
\end{equation}
On the other hand, when $\sum_{n_1+2\leq i_2\leq n} \al_{i_2} =0$, we write 
\begin{equation}
x^\alpha y^\beta= gh_1h_2 x_{n_1+1}^{\ell_1+k} 
\end{equation} with $g \in X^{k}_{J_1}$, $h_1 =\prod_{t=1}^{m}  y_{s_t}^{\beta_{s_t}}\in  Y_{J_1}^{k_{21}}$, $h_2 \in Y^{k_{22}}_{J_2\setminus \{n_1+1\}}$, $\beta_{s_t} \neq 0$ for each $t \in \ol{1,m}$.
In particular, when $n_1+1=n_2$, we have $h_2=1$ and $k_{21}=\ell_2$ by (\ref{a4.106}).
We write $g=x_{i_0}g'$ for some $i_0 \in J_1$, and denote $v_0'=g'x_{n}^{\ell_1+k} h_1h_2$.
Then $f=T(x_{i_0}v_0^\prime)$. For $t  \in\ol{1,m}$, we set
\begin{equation}\label{b4.120}
v_t=T\left(\frac{v_0'  x_{s_t} y_{i_0}}{y_{s_t}}\right) \in V'_k,
\end{equation}
\begin{equation}\label{a4.122}
w_t=T\left( \frac{v_0' y_{i_0}}{y_{s_t}x_{n_1+1}}\right) \in  V'_{k-1}=M_{k-1}.
\end{equation}

According to (\ref{a4.106}), we may assume that \begin{equation}
-E_{n_1+1,s_t} \left(w_t\right) =\sum_{l=1}^{m'}  T(\zeta_l),
\end{equation}
where $\zeta_l $ is a monomial in $\msr A$ for $j \in \ol{1,m'}$, $ T(\zeta_l) \in V'_k$ and $\mfk d' (\zeta_l )\leq k$. Since the monomials in $T(\zeta_l)$ all have the common factor $x_{n_1+1}y_{n_1+1}$ except for the first term $\zeta_l$. The monomials in $\sum_{l=1}^{m'}  T(\zeta_l)$ that do not have the common factor $x_{n_1+1}y_{n_1+1}$ exactly are  $\zeta_1,\cdots,\zeta_{m'}$.
By (\ref{a4.122}), we have
\begin{equation}
w_t \equiv  \frac{v_0' y_{i_0}}{y_{s_t}x_{n_1+1}}  + \sum_{\substack{1\leq j\leq m  \\ j\neq t}}  \frac{\beta_{s_j}v_0' x_{s_j}y_{i_0} y_{n_1+1}}{(k+\ell_1)y_{s_j} y_{s_t}}+
 \frac{(\beta_{s_t}-1)v_0' x_{s_t}y_{i_0} y_{n_1+1}}{(k+\ell_1) y_{s_t}^2 }
\quad(\mbox{mod }  (x_{n_1+1}y_{n_1+1} )^2 ).
\end{equation}
Thus
\begin{equation}
-E_{n_1+1,s_t}(w_t) \equiv    \frac{v_0' x_{s_t} y_{i_0}}{y_{s_t} }  +\sum_{\substack{1\leq j\leq m  \\ j\neq t}} \frac{\be_{s_j}v_0' x_{s_j}y_{i_0}  }{(k+\ell_1)y_{s_j}  }+
 \frac{(\beta_{s_t}-1)v_0' x_{s_t}y_{i_0}  }{ (k+\ell_1) y_{s_t}  }
\quad (\mbox{mod }   x_{n_1+1}y_{n_1+1}  ) .
\end{equation}
Therefore,
\begin{eqnarray}\label{b4.130}
-E_{n_1+1,s_t}(w_t) &=&T\left( \frac{v_0' x_{s_t} y_{i_0}}{y_{s_t} }  +\sum_{\substack{1\leq j\leq m  \\ j\neq t}} \frac{\be_{s_j}v_0' x_{s_j}y_{i_0}  }{(k+\ell_1)y_{s_j}  }+
 \frac{(\beta_{s_t}-1)v_0' x_{s_t}y_{i_0}  }{ (k+\ell_1) y_{s_t}  }  \right)\nonumber\\
&=&\frac{\beta_{s_t}+k+\ell_1-1}{k+\ell_1}v_{ t}  + \sum_{\substack{1\leq j\leq m  \\ j\neq t}} \frac{\beta_{s_j}}{k+\ell_1}  v_{j}
\end{eqnarray}
by (\ref{b4.120}).

We set the matrix
\begin{equation}
\mathbf{A} = \frac{1}{k+\ell_1}\left( \beta_{s_j}+\delta_{i j}\cdot (k+\ell_1-1) \right)_{1\leq i,j \leq m}.
\end{equation}
Note that $E_{n_1+1,s_t}(w_t) \in  M_{k}$ for each $t \in \ol{1,m}$, and
\begin{equation}
-\begin{pmatrix}
 E_{n_1+1,s_1} (w_1 )  \\
 E_{n_1+1,s_2} (w_2)   \\
 \vdots   \\
 E_{n_1+1,s_m} (w_m)
\end{pmatrix}
=\mathbf{A} \begin{pmatrix}
v_1   \\
v_2 \\
\vdots    \\
v_m
\end{pmatrix}.
\end{equation}

Since
\begin{equation}
\det (\mathbf{A})=\frac{(k+\ell_1-1)^{m-1}}{(k+\ell_1)^m} (k_{21}+\ell_1+k-1) \neq 0,
\end{equation}
$A$ is invertible. Thus $v_t\in M_k$ for each $t \in \ol{1,m}$.

If $i_0  \in \{s_1,\cdots,s_m\}$, there exist $t_0\in \ol{1,m}$  such that $s_{t_0}=i_0$. Then  $f=v_{t_0}  \in M_k$.
Suppose that $i_0 \not\in \{s_1,\cdots,s_m\}$. Since $T\left( \frac{v_0'}{x_{n_1+1}}\right) \in V'_{k-1}=M_{k-1}$,
\begin{equation}
E_{n_1+1,i_0} \left(T\left( \frac{v_0'}{x_{n_1+1}}\right)\right) \in M_{k} \subseteq V'_{k}.
\end{equation}
Similarly, we have
\begin{equation}
T\left( \frac{v_0'}{x_{n_1+1}}\right) \equiv  \frac{v_0'}{x_{n_1+1}} + \sum_{j=1}^m  \frac{\beta_{s_j}v_0' x_{s_j}y_{n_1+1}}{(k+\ell_1)y_{s_j}} \quad(\mbox{mod }  (x_{n_1+1}y_{n_1+1} )^2 ).
\end{equation}
Then
\begin{equation}
-E_{n_1+1,i_0} T\left( \frac{v_0'}{x_{n_1+1}}\right) \equiv v_0'x_{i_0} + \sum_{j=1}^m  \frac{\beta_{s_j}v_0' x_{s_j}y_{i_0}}{(k+\ell_1)y_{s_j} }\quad(\mbox{mod }  x_{n_1+1}y_{n_1+1}  ).
\end{equation}
Hence
\begin{equation}\label{a4.127}
-E_{n_1+1,i_0} T\left( \frac{v_0'}{x_{n_1+1}}\right) =T\left(  v_0'x_{i_0} + \sum_{j=1}^m  \frac{\beta_{s_j}v_0' x_{s_j}y_{i_0}}{(k+\ell_1)y_{s_j} }\right)=f+\sum_{j=1}^m \frac{\beta_{s_j}}{k+\ell_1} v_j.
\end{equation}
Therefore, \begin{equation}f=-E_{n_1+1,i_0} T\left( \frac{v_0'}{x_{n_1+1}}\right)-\sum_{j=1}^m \frac{\beta_{s_j}}{k+\ell_1} v_j \in M_k.\end{equation}

In summary, $f \in V'_k  \Rightarrow  f\in M_k$.
Recall we have already proved $M_k\subset V_k'$. Thus $M_k=V'_k$.
We have proved the lemma by induction on $k$.
\end{proof}
\pse

Next we calculate the associated variety of $M$ in this case.
\begin{lemma}\label{alem:4.5}
If $\ell_1,\ell_2>0$ and $ n_1<n_2=n$,
\begin{equation}\label{b4.134}
\msr V( {\msr H}_{\la\ell_1,\ell_2\ra})= \msr{V}  ( \{E_{j,i},\Delta^{j_1,j_1'}_{i_1,i_l'}  \mid (j,i)\not\in L; (j_1,i_1),(j_1',i_1')\in L \} ).\end{equation}  
In particular, when $\ell_1,\ell_2>0$ and $ n_1+1=n_2=n$, 
\begin{equation}\label{a4.135}
\msr V( {\msr H}_{\la\ell_1,\ell_2\ra})= \msr{V}  ( \{E_{j,i} \mid (j,i)\not\in L \} ).\end{equation} \end{lemma}

\begin{proof}
Since
\begin{align}
E_{j,i} (M_k) \subset  M_k  \qquad\for\;\; (j,i)\not\in L   ,
\end{align}
we have $\{ E_{j,i} \mid (j,i)\not\in L \} \subseteq I_{(1)}$.

For $i\in \ol{1,n_1}$ and $j\in \ol{n_1+2,n}$,
\begin{equation}
 E_{n_1+1,i}\left(T(x_{n_1+1}^{\ell_1} y_i^{\ell_2})\right)=-\frac{\ell_1+\ell_2+1}{\ell_1+1}T(x_i x_{n_1+1}^{\ell_1+1} y_i^{\ell_2}) \not\in M_0
\end{equation}
by (\ref{b4.130}) with $s_t=i$, and 
\begin{equation}
 E_{j,i}\left(T(x_{n_1+1}^{\ell_1} y_i^{\ell_2})\right)=-T(x_i x_j x_{n_1+1}^{\ell_1 } y_i^{\ell_2}) \not\in M_0
\end{equation}
by (\ref{a2.7}).
Thus $I_{(1)}=\{ E_{j,i} \mid (j,i)\not\in L \} $.

For $ \xi \in  \mfk k_{(p)}$ with $p>1$ (cf. (\ref{a4.12})), we write
\begin{equation}
\xi = \sum_{s=1}^m a_s\prod_{t=1}^p  E_{j_{s,t},i_{s,t}}
\quad\mbox{ where }( j_{s,t},i_{s,t}) \in L,\;  a_s\in \mathbb{F}.
\end{equation}
Fix $v_0\in M_0$, we have
\begin{eqnarray}& &\mfk d' \left(   \left( \prod_{t=1}^p E_{j_{s,t},i_{s,t}}\right)(v_0)\right)  \nonumber\\
&=&\mfk d' \left(  \left( \prod_{t=1}^p (- x_{i_{s,t}} x_{j_{s,t}}- y_{i_{s,t}} \ptl_{y_{j_{s,t}}})\right)(v_0)\right)  = p.
\end{eqnarray} Then  $\xi (v_0)=0$ due to $\xi (v_0) \in M_{p-1}$.

Consider the partial sum of highest degree in $\xi(v_0)$:
\begin{equation}
\sum_{s=1}^m (-1)^p a_s\left( \prod_{t=1}^p (x_{i_{s,t}} x_{j_{s,t}} )\right)(v_0) =0
\end{equation}
When $n_1+1=n_2$, we have
\begin{equation}
\sum_{s=1}^m a_s\left( \prod_{t=1}^p  x_{i_{s,t}}  \right)=0.
\end{equation}
Hence $\xi \in U_{p-1} (\mfk g)$. Therefore, for $p>1$,
\begin{equation}
I_{(p)}  \equiv \{  0\} \quad (\mbox{mod}\; \la I_{(1)}\ra).
\end{equation}
Then 
\begin{equation}
 \operatorname{Ann}_{S( \mfk g)} (\ol{M}) = \left\langle
E_{j,i} \mid (j,i)\not\in L  \right\rangle.
\end{equation}
and (\ref{a4.135}) holds. 

When $n_1+1<n_2$, we have
\begin{equation}
\sum_{s=1}^m a_s\left( \prod_{t=1}^p  x_{i_{s,t}} x_{j_{s,t}}  \right)=0.
\end{equation}
According to the Lemma \ref{alem:2}, 
\begin{equation}
\xi \in  \la \Delta^{j_1,j_1'}_{i_1,i_l'}  \mid   (j_1,i_1),(j_1',i_1')\in L \ra.
\end{equation}
Recall (\ref{b4.36}), for $v_i\in M_i$ and $(j_1,i_1),(j_1',i_1') \in L$,
\begin{equation}
\left(\Delta_{i_1,i_1'}^{j_1,j_1'}\right)^{\ell_2+1}(v_i)=
 (x_{i_1'}y_{i_1}-x_{i_1}y_{i_1'})^{\ell_2+1}\left(x_{j_1'}\ptl_{y_{j_1}}-x_{j_1}\ptl_{y_{j_1'}}\right)^{\ell_2+1}(v_i)=0 ,
 \end{equation}
since $\sum_{j\in J_2}\deg_{y_j}(v_i) \leq \ell_2$.
Hence 
\begin{equation}
\sqrt{I}  = \la E_{j,i},\Delta^{j_1,j_1'}_{i_1,i_l'}  \mid (j,i)\not\in L; (j_1,i_1),(j_1',i_1')\in L \ra
\end{equation}
and (\ref{b4.134}) holds.

\end{proof}
\pse

Remind that we denoted the determinantal ideal $I_t (S_1,S_2)$ to be the ideal generated by all the $t$-minors in the matrix $(z_{i,j})_{i \in S_1,j\in S_2}$, and $\msr{V}_t(S_1,S_2) =\msr{V}(I_t)$ is called  a {\it determinantal variety}, where $S_1$ and $S_2$ are subsets of $\ol{1,n}$. In particular,
\begin{align}I_2 (S_1,S_2)=\left\langle  \Delta_{i_1,i_2}^{ j_1,j_2} \mid    j_1, j_2 \in S_1;\  i_1, i_2 \in S_2 \right\rangle\end{align} and
\begin{align}I_3 (S_1,S_2)=\left\langle  \Delta_{i_1,i_2,i_3}^{j_1,j_2,j_3} \mid    j_1, j_2,j_3 \in S_1;\ i_1,i_2,i_3 \in S_2  \right\rangle.\end{align}
For convenience, we denote $I_t(S_1,S_2)=0$ and $\msr{V}(I_t(S_1 , S_2)) =\mbb A^{|S_1||S_2|}$, when $t>|S_1|$ or $t>|S_2|$.
Now we have our main theorem:

\begin{theorem}
If $n_1<n_2$, $\ell_1 \leq 0$ or $\ell_2 \leq 0$,
\begin{equation}\label{b4.150}
\msr V({\msr H}_{\la\ell_1,\ell_2\ra})\cong \msr{V}( I_3(J_2\cup J_3, J_1\cup J_2))\cap \msr{V}(I_2(J_2 , J_1))\cap \msr{V}(I_2(J_3 , J_2)) \cap\msr{V}(I_1(J_2 , J_2)).\end{equation}
If $n_1=n_2$,
\begin{equation}\label{a4.149}
\msr V({\msr H}_{\la\ell_1,\ell_2\ra})\cong \msr{V}\left(  I_3(  J_3, J_1 ) \right).\end{equation}
When $n_1<n_2=n$ and $\ell_1,\ell_2>0$,
\begin{equation}
\msr V( {\msr H}_{\la\ell_1,\ell_2\ra})\cong \msr{V}\left(  I_2(  J_2, J_1 ) \right).\end{equation}

In particular, for the irreducible $sl(n)$-module ${\msr H}_{\la\ell_1,\ell_2\ra}$ in Proposition 1, the associated variety is covered by the above-mentioned cases.
\end{theorem}

\begin{proof}
If $\ell_1,\ell_2>0$ and $n_1<n_2=n$, then $J_3=\emptyset, |J_1|=n_1, |J_2|=n-n_1$ and 
\begin{eqnarray}\label{b4.154}
\msr V( {\msr H}_{\la\ell_1,\ell_2\ra})&=&  \msr{V}  ( \{E_{j,i},\Delta^{j_1,j_1'}_{i_1,i_l'}  \mid (j,i)\not\in L; (j_1,i_1),(j_1',i_1')\in L \} )\nonumber\\
&\cong &\msr{V}\left(  I_2(  J_2, J_1 ) \right)
\end{eqnarray}
as an affine variety by Lemma \ref{alem:4.5}. Similarly, when $\ell_1\leq 0\leq \ell_2$ and $n_1<n_2=n$, (\ref{b4.154}) still holds.

Suppose $|J_1|,|J_2|,|J_3|\geq 2$. According to Lemma \ref{alem:4.2} and Lemma \ref{alem:4.3}, we have:
if $\ell_1 \leq 0$ or $\ell_2 \leq 0$,
\begin{eqnarray}\label{a4.151}
& &\msr V(  {\msr H}_{\la\ell_1,\ell_2\ra})\nonumber\\
& =&\msr V \left( \left\{ \left. \begin{aligned}
E_{i,j},\Delta^{j_1,j_2,j_3}_{i_1,i_2,i_3 } ,  \Delta^{k_2,k_2'}_{k_1,k_1' } ,\Delta^{k_3,k_3'}_{k_2,k_2' } \quad \end{aligned} \right\vert \begin{aligned} &
 i_1,i_2,i_3\in J_1\cup J_2; j_1,j_2,j_3 \in J_2\cup J_3;\nonumber\\&  (i,j)\notin L;  k_s,k_s'\in J_s,s=1,2,3  \end{aligned}\right\} \right) \nonumber\\
&\cong& \msr{V}\left( I_3(J_2\cup J_3, J_1\cup J_2)+ I_2(J_2 , J_1)+ I_2(J_3 , J_2) +I_1(J_2 , J_2)\right);
\end{eqnarray}
So (\ref{b4.150}) hold.

When $n_1<n_2$, $|J_1|<2$ or $|J_2|=1$ or $ |J_3|<2$, (\ref{a4.33}) and (\ref{a4.77}) still hold, but
some terms in these two expressions are zero in this case. They correspond the cases: $I_2(J_2 , J_1) =0$ if $|J_2|\leq 1$ or $|J_1|\leq 1$, $I_2(J_3 , J_2) =0$ if $|J_3 |\leq 1$ or $|J_2|\leq 1$, $I_3(J_2\cup J_3 , J_1\cup J_2) =0$ if $|J_2\cup J_3|\leq 2$ or $|J_1\cup J_2|\leq 2$.
Hence (\ref{a4.151}) still hold.

Suppose $n_1=n_2$.
In this case,
\begin{equation}\td\Dlt=\sum_{i=1}^{n_1}x_i\ptl_{y_i}+\sum_{s=n_2+1}^n
y_s\ptl_{x_s}, \end{equation}
and (\ref{a2.8}) no longer holds.

Suppose $n_1=n_2<n$. First we consider the case $\ell_1=-m_1,\ell_2=-m_2 \leq 0$. According to (6.6.52) in \cite{Xx},
\begin{eqnarray}M&=&\mbox{Span}\{[\prod_{r=1}^{n_1}x_r^{l_r}]
[\prod_{s=1}^{n-n_1}y_{n_1+s}^{k_s}][\prod_{r=1}^{n_1}\prod_{s=n_1+1}^n(x_rx_s-
y_ry_s)^{l_{r,s}}]\nonumber\\& &\qquad\;\;\mid
l_r,k_s,l_{r,s}\in\mbb{N};\sum_{r=1}^{n_1}l_r=m_1;\sum_{s=1}^{n-n_1}k_s=m_2\}.\end{eqnarray}
We take
\begin{equation}M_0 =   \mbox{Span}\{X_{J_1}^{m_1}Y_{J_3}^{m_2}\}  \end{equation}(cf. (\ref{a2.40}) and (\ref{a2.41})).
Then by the arguments from (6.6.45) to (6.6.49) in \cite{Xx}, we can conclude that
\begin{equation}M_k =   \mbox{Span}\{X_{J_1}^{m_1}Y_{J_3}^{m_2}P^i \mid i\in \ol{0,k} \}  . \end{equation}
In this case, $L=L_3$ and we can calculate that
\begin{eqnarray}\msr V( {\msr H}_{\la-m_1,-m_2\ra})
&=&\msr{V}  ( \{E_{j,i} , \Delta^{j_1,j_2,j_3}_{i_1,i_2,i_3 }\mid (j,i)\not\in L;(j_s,i_s)\in L ,s=1,2,3  \} ) \nonumber \\
&\cong& \msr{V}\left(  I_3(  J_3, J_1 ) \right)  \end{eqnarray}
by Lemma \ref{alem:3.3}.

Next we consider the case $1<n_1=n_2<n $, $\ell_1 \leq 0$ and $\ell_2 \geq 0$. We set $\ell_1=-m_1-m_2$ and $\ell_2= m_2$ with $m_1,m_2 \in \mathbb{N}$. In this case,
\begin{eqnarray}M&=&{\msr H}_{\la
-m_1-m_2,m_2\ra}\nonumber\\&=&\mbox{Span}\{[\prod_{r=1}^{n_1}x_r^{l_r}]
[\prod_{1\leq p<q\leq
n_1}(x_py_q-x_qy_p)^{k_{p,q}}][\prod_{r=1}^{n_1}\prod_{s=n_1+1}^n(x_rx_s-
y_ry_s)^{l_{r,s}}]\nonumber\\& &\qquad\;\;\mid
l_r,k_{p,q},l_{r,s}\in\mbb{N};\sum_{r=1}^{n_1}l_r=m_1;\sum_{1\leq
p<q\leq n_1}k_{p,q}=m_2\}\end{eqnarray}
by (6.6.53) in \cite{Xx}. Take
\begin{equation}M_0 =   \mbox{Span}\left\{  \left. \left[ \prod_{ 1\leq p<q\leq n_1} (x_py_q-x_q y_p)^{k_{p,q}}\right] X_{J_1}^{m_1} \right\vert
k_{p,q}\in \mathbb{N}, \sum_{ 1\leq p<q\leq n_1}k_{p,q} =m_2
\right\} . \end{equation}
Then by (6.6.51) in \cite{Xx}, we can conclude that
\begin{equation}M_k =   \mbox{Span}\{ M_0 P^i \mid i\leq k
\} . \end{equation}
and
\begin{eqnarray}\msr V( {\msr H}_{\la\ell_1,\ell_2\ra})&=& \msr{V}  ( \{E_{j,i} ,\  \Delta^{j_1,j_2,j_3}_{i_1,i_2,i_3 }\mid (j,i)\not\in L;(j_s,i_s)\in L ,s=1,2,3  \} ) \nonumber\\
&\cong& \msr{V}\left(  I_3(  J_3, J_1 ) \right).\end{eqnarray}
holds. Symmetrically, when $ n_1=n_2<n $, $\ell_1 \geq 0$ and $\ell_2 \leq 0$, the above equation still holds.

In summary, (\ref{a4.149}) holds if $n_1=n_2$.
This completes the proof of the theorem.

\end{proof}
\psp

In particular, when $n_1=n_2$ and $|J_1|,|J_3| \geq 3$,
\begin{equation}\msr V({\msr H}_{\la\ell_1,\ell_2\ra}) \cong
\msr{V}_3 (J_3 , J_1)\end{equation}
 is exactly a determinantal variety. When $n_2=n$, the projective version of the associated variety
\begin{equation}\msr V({\msr H}_{\la\ell_1,\ell_2\ra}) \cong
\msr{V}_2 (J_2 , J_1)\end{equation}
is the well-known Serge variety in algebraic geometry.

The GK-dimension is equal to the dimension of the associated variety (as an affine variety). We can calculate that

\begin{equation}
\dim \msr{V}(M)=
\begin{cases}
2n-3,\ \quad \text{if }n_2\neq n \text{ and } n_1\neq n_2; \\
2n-4,\ \quad \text{if }1 < n_1=n_2<n-1; \\
n-1,\ \quad \text{if } 1=n_1=n_2 \text{ or } n_1=n_2=n-1\text{ or } n_1\neq n_2=n,
\end{cases}
\end{equation}
which was obtained earlier by Bai in \cite{Bz}.

 \end{document}